\documentclass[11pt]{amsart}
\usepackage[utf8]{inputenc}

\usepackage{overpic, amssymb, times, graphicx, tikz-cd,float}
\usepackage[all]{xy}
\usepackage[backref=page]{hyperref}
\usepackage{verbatim}
\usepackage{cleveref}
\usepackage{bm}
\usepackage{amsmath,amsthm}

\hypersetup{
    colorlinks=true,
    linkcolor=blue,
    filecolor=blue,      
    urlcolor=blue,
    citecolor=blue,
}
\numberwithin{equation}{section}

\newtheorem{theorem}{\textbf{Theorem}}[section]
\newtheorem*{theorem*}{\textbf{Theorem}}
\newtheorem{thmx}{Theorem}

\newtheorem{conj}[theorem]{\textbf{Conjecture}}

\newtheorem{definition}[theorem]{\textbf{Definition}}
\newtheorem{proposition}[theorem]{\textbf{Proposition}}
\newtheorem{lemma}[theorem]{\textbf{Lemma}}
\newtheorem{question}[theorem]{\textbf{Question}}

\newtheorem{claim}[theorem]{\textbf{Claim}}
\newtheorem{corollary}[theorem]{\textbf{Corollary}}
\newtheorem{remark}[theorem]{\textbf{Remark}}

\newtheorem{example}[theorem]{\textbf{Example}}

\newtheorem{definition/proposition}[theorem]{\textbf{Definition/Proposition}}

\def\N{{\mathbb N}}
\def\R{\mathbb{R}}
\def\Z{{\mathbb Z}}

\def\C{{\mathbb C}}
\def\D{{\mathbb D}}

\def\Q{{\mathbb Q}}

\newcommand{\CP}{\mathbb{C}\mathbb{P}}

\def\cA{{\mathcal A}}

\def\cM{{\mathcal M}}

\def\cO{{\mathcal O}}

\def\mc{{\mathfrak {mc}}}

\def\rd{{\rm d}}

\def\la{\langle\,}
\def\ra{\,\rangle}

\def\std{\rm std}

\def\CZ{\rm CZ}

\DeclareMathOperator{\Ima}{im}
\DeclareMathOperator{\ind}{ind}

\DeclareMathOperator{\Id}{id}

\DeclareMathOperator{\virdim}{virdim}

\DeclareMathOperator{\APT}{APT}
\DeclareMathOperator{\AT}{AT}
\DeclareMathOperator{\Pl}{P}

\DeclareMathOperator{\symp}{Symp}
\DeclareMathOperator{\OB}{OB}
\DeclareMathOperator{\BO}{BO}
\DeclareMathOperator{\SOB}{SOB}

\DeclareMathOperator{\CH}{CH}

\DeclareMathOperator{\Openbook}{OB}
\DeclareMathOperator{\DG}{DG}
\DeclareMathOperator{\LCH}{LCH}
\newcommand{\Addresses}{{
		\bigskip
		\footnotesize

	    Zhengyi Zhou, \par\nopagebreak
        \textsc{State Key Laboratory of Mathematical Sciences, Chinese Academy of Sciences;}\par\nopagebreak
	    \textsc{Morningside Center of Mathematics, Chinese Academy of Sciences;}\par\nopagebreak
         \textsc{Academy of Mathematics and Systems Science, Chinese Academy of Sciences, China}\par\nopagebreak
		\textit{E-mail address}: \href{mailto:zhyzhou@amss.ac.cn}{zhyzhou@amss.ac.cn}

}}

\title{Algebraic planar torsion in contact manifolds}
\author{Zhengyi Zhou}

\begin{document}
	\maketitle
\begin{abstract}
We demonstrate that the functorial properties of the symplectic field theory under strong cobordisms and surgery cobordisms can produce finite algebraic (planar) torsions from simple examples, which gives a unified treatment of most of the known computations of algebraic (planar) torsions. In addition, we obtain many families of new examples, notably including (1) stably fillable examples in all dimensions $\ge 5$ with algebraic (planar) torsion precisely $k$ for any given $k\in \N_+$, confirming a conjecture of Latschev and  Wendl; (2) contact structures on spheres of all dimensions at least $5$ with finite algebraic planar torsion at least $1$, which implies that tight not weakly fillable contact structures are ubiquitous in higher dimensions. We also explain that all known examples of contact manifolds without strong/weak fillings in dimension $\ge 5$ have algebraic planar torsion.
\end{abstract}

\section{Introduction}
A fundamental dichotomy in contact topology distinguishes between \emph{overtwisted} and \emph{tight} contact structures. As established in \cite{Eli89,BEM}, overtwisted structures are essentially topological objects, satisfying $h$-principles that reduce questions of existence and classification to classical obstruction theory. Tight contact structures, in contrast, are inherently geometric and generally more difficult to construct and classify. An important sufficient condition for tightness is symplectic fillability: every contact manifold that is \emph{symplectically fillable} is necessarily tight \cite{Gro85,Eli90,Nie06,MNW,BEM,schmaltz2023nonfillability}. Here we consider both \emph{strong} fillings (where the contact structure is defined by a Liouville vector field near the boundary) and \emph{weak} fillings (where the symplectic form is compatible with the canonical conformal symplectic structure on the contact hyperplanes). These notions yield the hierarchy
\[
\{\text{strongly fillable}\}
\subseteq \{\text{weakly fillable}\} \subseteq \{\text{tight}\},
\]
which is strict in all odd dimensions $\ge 3$ \cite{Eli96,EtnHon,MNW,BGM}.

Obstructions to strong or weak fillings in dimension at least $5$ often arise from ``wrong cobordisms'' to contact manifolds that exhibit strong filling rigidity.\footnote{In dimension $3$, the first examples of tight non-fillable contact manifolds \cite{EtnHon} also admit cobordisms to contact manifolds with no symplectic fillings \cite{Lisca}, where the rigidity comes from gauge theory.} For instance,  all but one tight contact structure on $T^3$ \cite{Eli96} admits connected cobordisms to disconnected contact manifolds where one component is not cofillable. More generally, contact manifolds with Giroux torsion \cite{zbMATH00721020} or Wendl's planar torsion \cite{zbMATH06218377} admit cobordisms to overtwisted contact manifolds \cite{Gay,Wen13}. Contact manifolds with higher-dimensional Giroux torsion \cite{MNW} admit connected cobordisms to the disjoint union of a uniruled stable Hamiltonian manifold and a contact manifold. Moreover, Bourgeois contact manifolds \cite{Bourgeois02} admit cobordisms to flexibly fillable contact manifolds, which impose strong restrictions on their fillings \cite{BGZ,zbMATH07367119,zbMATH07673358,bowden2022tight}.

Symplectic field theory (SFT), introduced by Eliashberg, Givental, and Hofer \cite{zbMATH01643843}, provides a framework for constructing functors from symplectic cobordism categories to algebraic categories. This makes SFT a natural tool for studying symplectic fillings, which are special cases of cobordisms. Many geometric properties find algebraic interpretations through SFT: for example, overtwisted manifolds have vanishing contact homology \cite{zbMATH05709738}, leading to the concept of algebraically overtwisted manifolds \cite{zbMATH05658836}. In dimension $3$, Giroux torsion and Wendl's planar torsion have algebraic counterparts identified by Latschev and Wendl \cite{LW} through \emph{algebraic torsion} in full SFT (which counts curves of all genera). More recently, Moreno and the author \cite{moreno2024landscape} introduced \emph{algebraic planar torsion} in rational SFT (RSFT), which counts only rational curves, to explain filling obstructions in this context. Finiteness of algebraic (planar) torsion obstructs augmentations in the corresponding algebraic theory. Since augmentations arise from counting curves in symplectic fillings, finiteness of torsion implies an obstruction to fillings. Our first theorem demonstrates that, in dimension at least $5$, all known examples of contact manifolds without strong or weak fillings fall within the realm of finite algebraic planar torsion.

\begin{thmx}\label{thm:everything}
    All known examples of contact manifolds without strong or weak fillings in dimension at least $5$ admit finite algebraic planar torsion (with twisted coefficients).\footnote{It is unknown to the author whether the analogous claim holds in dimension $3$. Some earlier examples of tight non-fillable contact structures, such as $(+1)$-surgery along a trefoil knot in $(S^3,\xi_{\std})$ from \cite{zbMATH02105207} (where the obstruction comes from Seiberg--Witten theory), do have finite algebraic planar torsion; see \cite{zbMATH07699409,AOT,zbMATH07983165}.}
\end{thmx}

In particular, all known non-fillable examples in higher dimensions admit no augmentations in RSFT. This leads to natural questions:

\begin{question}
    In general dimensions (or in dimension at least $5$), do there exist contact manifolds whose filling obstructions are detected by full SFT but not by RSFT? Are there obstructions to strong and weak fillings that go beyond SFT augmentation obstructions? Are there obstructions arising from higher-dimensional moduli spaces, in the spirit of Floer homotopy theory?
\end{question}

Given the high flexibility of contact structures in higher dimensions, affirmative answers to both questions seem likely. However, both the theoretical construction of those invariants and the geometric construction of potential examples are missing in the literature and are very challenging.

Although algebraic (planar) torsion is defined in all dimensions, it has been studied mostly in dimension $3$ by Wendl \cite{zbMATH06218377} and Latschev--Wendl \cite{LW}. The corresponding geometric structures beyond dimension $3$ are not yet fully understood. Massot, Niederkr{\"u}ger, and Wendl \cite{MNW} introduced higher-dimensional Giroux torsion, which implies algebraic (planar) torsion at most $1$ by Moreno's work \cite{moreno2019algebraic}. Moreno \cite{moreno2019algebraic} also constructed examples with algebraic torsion at most $k$ for any $k \in \mathbb{N}$, though whether this upper bound is sharp was unclear. The methods in \cite{LW,moreno2019algebraic} are based on explicit computation of holomorphic curves from intersection theory and adiabatic limit. This paper focuses on computing these torsions in general dimensions by using the functoriality of RSFT applied to the ``wrong cobordisms'' (often strong cobordisms) that are used to obstruct fillings, instead of working out explicit curve counts in the target contact manifold. This approach confirms the following conjecture of Latschev and Wendl.

\begin{conj}[Latschev--Wendl {\cite[\S 5]{LW}}]
    For all integers $k \ge 1$ and $n \ge 2$, there exist infinitely many closed $(2n+1)$-dimensional contact manifolds with algebraic torsion $k$. There also exist $(2n+1)$-dimensional contact manifolds with (untwisted) algebraic torsion $k$ that admit stable symplectic fillings.
\end{conj}

Our first main result confirms the analogous conjecture in the context of RSFT, whose analytical foundation was established in \cite{moreno2024landscape}.

\begin{thmx}\label{thm:finite_APT}
    For any integers $k \ge 1$ and $n \ge 2$, there exist infinitely many $(2n+1)$-dimensional contact manifolds with algebraic planar torsion $k$. Moreover, there exist $(2n+1)$-dimensional contact manifolds with (untwisted) algebraic planar torsion $k$ that admit stable symplectic fillings.
\end{thmx}

The original conjecture of Latschev and Wendl can be confirmed under the assumption that full SFT is well-defined with the expected properties.

\begin{thmx}\label{thm:finite_AT}
    Assume full SFT is well-defined. Then for any integers $k \ge 1$ and $n \ge 2$, there exist infinitely many $(2n+1)$-dimensional contact manifolds with algebraic torsion $k$. Moreover, there exist $(2n+1)$-dimensional contact manifolds with (untwisted) algebraic torsion $k$ that admit stable symplectic fillings.
\end{thmx}

\begin{remark}
    The well-definedness of full SFT in \Cref{thm:finite_AT} requires:
    \begin{enumerate}
        \item Full SFT is packaged as a (cocurved) $IBL_\infty$ algebra in the sense of \cite[Definition 5.9]{moreno2024rsft}.
        \item The analogue of \cite[Theorem 3.10]{moreno2024landscape} and \cite[Proposition 3.9]{moreno2024rsft} holds. Note that we \emph{do not} require SFT to be a well-defined functor from the exact symplectic cobordism category to the homotopy category of $IBL_\infty$ algebras.
        \item The above requirements are consequences of a definition of virtual count for each moduli space of curves of virtual dimension $0$, satisfying a quadratic master equation coming from the boundary of moduli spaces of virtual dimension $1$. We are essentially assuming the existence of such virtual counts, such that if a compactified moduli space of virtual dimension $0$ is cut out transversely, the virtual count can be taken as the geometric count.
    \end{enumerate}
\end{remark}

As noted earlier, finite algebraic torsion is produced via RSFT functoriality and strong cobordisms to contact manifolds with filling rigidities. Similar results have appeared in \cite{+1}. More precisely, one family of strong cobordisms we consider is the following:

\begin{definition}\label{def:torsion_cobordism}
    A \emph{torsion cobordism} (of type I) is a strong cobordism $W$ such that the convex boundary $\partial_+ W = \partial(\Sigma\times \D)\# Y$ for a Liouville domain $\Sigma$, and there exists a homology class $A\in H_k(\Sigma;\Q)$ that is mapped to zero in $H_k(W;\Q)$ via the map $\Sigma\times \{(0,1)\}\subset \partial (\Sigma \times \D) \subset W$. If $Y\ne \emptyset$, we require in addition that $c_1^{\Q}(\Sigma)=0$.\footnote{This requirement simplifies the proof of technical results in \Cref{prop:curve}. There are other technical conditions so that analogues of \Cref{prop:curve} hold.}
\end{definition}

For example, any overtwisted contact manifold can be realized as the concave boundary of a torsion cobordism; see \Cref{ex:ot}. Hence, by \cite{Wen13}, any domain with planar torsion can be realized as the concave boundary of a torsion cobordism. The main theorem of the paper is the following.

\begin{thmx}\label{thm:torsion}
   Let $W$ be a torsion cobordism (of type I). Then the algebraic planar torsion $\APT(\partial_-W)<\infty$ for the concave boundary $\partial_-W$ of $W$.
\end{thmx}

We will obtain more precise upper bounds for $\APT(\partial_-W)$ if we have more precise control of the non-exactness of $W$. The flexibility of higher-dimensional contact geometry allows many constructions of torsion cobordisms, via contact $(+1)$-surgeries, Liouville connected sums, spinal open books, contact manifolds with multiple fillings, and non-trivial contactomorphisms, etc. Moreover, \Cref{def:torsion_cobordism} is one of the simplest cobordisms that force finite algebraic planarity of the concave boundary. A similar philosophy of proof suggests many variants of torsion cobordisms. We will introduce them in the main text, but the list will not be exhaustive. The rule of thumb is that if the convex boundary exhibits some filling rigidity (uniqueness, non-existence, etc.) via rational holomorphic curves (more precisely, if it can be phrased in RSFT), then a ``wrong cobordism'' violating those properties will force finite algebraic planarity on the concave boundary.

Bowden, Gironella, Moreno, and the author \cite{bowden2022tight} constructed exotic contact structures on spheres, yielding many tight but non-strongly-fillable contact manifolds. By \Cref{thm:torsion}, such examples should exhibit finite algebraic planar torsion.

\begin{thmx}\label{thm:sphere}
    Let $(S^{2n+1},\xi_{\mathrm{ex}})$ be a tight non-fillable contact structure constructed in \cite{bowden2022tight} for $n \ge  2$. Then $\APT(S^{2n+1},\xi_{\mathrm{ex}})=1$ if $n\ge 3$ and $1\le \APT(S^{5},\xi_{\mathrm{ex}})<\infty$.
\end{thmx}

Since $S^{2n+1}$ has no topology, the algebraic planar torsion in \Cref{thm:sphere} can be interpreted as fully twisted algebraic planar torsion $\APT_{\mathrm{tw}}$, which obstructs weak fillings. This allows us to sharpen the results of \cite{bowden2022tight} as follows:

\begin{thmx}\label{thm:all}
    For $n \ge 2$, let $(Y^{2n+1},J)$ be an almost contact structure admitting a tight contact structure. Then $(Y^{2n+1},J)$ admits a tight, non-weakly-fillable contact structure $\xi$ with finite $\APT_{\mathrm{tw}}(Y,\xi)$.
\end{thmx}

In particular, we remove the condition $c^{\Q}_1(J)=0$ when $n=2$ in \cite{bowden2022tight} with the help of higher-dimensional intersection theory. In other words, in dimensions at least $5$, the class of manifolds admitting tight non-weakly-fillable contact structures is as large as the class admitting tight contact structures.

\subsection*{Organization of the paper}
We review several constructions in contact topology and introduce several constructions of torsion cobordisms in \S \ref{s:geometric_construction}. In \S \ref{s:SFT}, we review RSFT/SFT and algebraic planar torsions/algebraic torsion. We prove \Cref{thm:torsion} and its quantitative version as well as \Cref{thm:sphere} and \Cref{thm:all} in \S \ref{s:torsion}. We also discuss variants of torsion cobordisms in \S \ref{s:torsion}. We establish lower bounds for algebraic planar torsion in some constructions and prove \Cref{thm:finite_APT} and \Cref{thm:finite_AT} in \S \ref{s:Lower}. We compare our results to other constructions of finite algebraic (planar) torsion contact manifolds and justify \Cref{thm:everything} in \S \ref{s:comparison_application}, as well as some applications in symplectic mapping class groups.

\subsection*{Acknowledgments}
The author would like to thank Russell Avdek and Fabio Gironella for helpful conversations. The author is supported by the National Key R\&D Program of China under Grant No.\ 2023YFA1010500, the National Natural Science Foundation of China under Grant No.\ 12288201 and 12231010.

\section{Geometric constructions}\label{s:geometric_construction}
In this section, we will review relevant constructions in contact topology and explain several constructions of torsion cobordisms.

\subsection{Contact $(+1)$-surgeries}
Here we review contact $(+1)$-surgeries following Conway--Etnyre \cite{zbMATH07206659}. The definition of a contact $(+1)$-surgery along a Legendrian sphere is implicit in the theory of Weinstein handle attachments \cite{CieEli,zbMATH04147116,zbMATH00011093}. We first recall a model for the Weinstein $k$-handle for $k\le n$ as follows. Let $\omega = \sum_{i=1}^k\rd q_i\wedge \rd p_i + \sum_{i=1}^{n-k}\rd x_i\wedge \rd y_i$ be the standard symplectic structure on $\R^{2k}\times \R^{2n-2k}$; then we have the Liouville vector field 
$$v=\sum_{i=1}^k(-p_i\partial_{p_i}+2q_i\partial_{q_i})+\frac{1}{2}\sum_{i=1}^{n-k}(x_i\partial_{x_i}+y_i\partial_{y_i}).$$
For $a,b>0$, let $D_a$ be the disk of radius $a$ in the $p_i$-subspace and $D_b$ the disk of radius $b$ in the $(q_i,x_i,y_i)$-subspace. Then $(H_{a,b}=D_a\times D_b,\omega)$ serves as a model for the Weinstein $k$-handle. The Liouville vector field points outward along $\partial_+H_{a,b}:=D_a\times \partial D_b$ and inward along $\partial_-H_{a,b}:=\partial D_a \times D_b$, which induces contact structures on the boundary.

Note that $S_a=\partial D_a \times \{0\}$ is an isotropic sphere and $S_b=\{0\} \times \partial D_b$ is a coisotropic sphere. By Moser's trick, the germ of the contact structure along $S_a\subset \partial_-H_{a,b}$ is contactomorphic to the germ of a contact structure along an isotropic sphere with a trivial conformal symplectic normal bundle (the quotient of the symplectic orthogonal of the tangent bundle by the tangent bundle) in any contact manifold, and similarly the germ along $S_b\subset \partial_+H_{a,b}$ is contactomorphic to the germ along a coisotropic sphere in any contact manifold \cite[Lemma 3.4]{zbMATH07206659}. Consequently, given an isotropic sphere with a trivialization of the conformal symplectic normal bundle in a contact manifold $Y$, we can glue $H_{a,b}$ to a neighborhood of the isotropic sphere along $\partial_-H_{a,b}$ for $b\ll 1$ using the Liouville vector field. This yields a Weinstein cobordism with concave boundary $Y$ and a new convex boundary $Y'$; we call this procedure an \emph{isotropic surgery}. On the other hand, given a coisotropic sphere in $Y$, we can glue $H_{a,b}$ to a neighborhood of the coisotropic sphere along $\partial_+H_{a,b}$ for $a\ll 1$. This yields a Weinstein cobordism with convex boundary $Y$ and a new concave boundary $Y''$; this procedure is referred to as a \emph{coisotropic surgery}. The isotropic surgery and coisotropic surgery are reverse operations to each other, i.e.\ we can undo an isotropic/coisotropic surgery by applying a coisotropic/isotropic surgery to the coisotropic/isotropic sphere in the surgery handle \cite[Lemma 3.9, Proposition 3.10]{zbMATH07206659}.

\begin{definition}
A contact $(+1)$-surgery along a Legendrian sphere $\Lambda \subset Y$ is a coisotropic surgery along the Legendrian. We use $Y_{+1}(\Lambda)$ to denote the resulting contact manifold, and $W_{\Lambda}$ to denote the surgery Weinstein cobordism from $Y_{+1}(\Lambda)$ to $Y$.  
\end{definition}

\begin{remark}
    The $(+1)$-surgery depends on a parametrization $S^{n-1}\simeq \Lambda$; in general, the resulting contact manifold, and even the underlying smooth manifold, depends on this parametrization. We suppress this choice in our theorems, as our results do not depend on the parametrization.
\end{remark}

\begin{example}\label{ex:ot}
The standard overtwisted sphere $(S^{2n-1},\xi_{ot})$ can be obtained by applying $(+1)$ surgery to a Legendrian sphere $\Lambda$ in $\partial(D^*{S^{n-1}}\times \D)$, where $D^*S^{n-1}$ is the unit co-disk bundle of $S^{n-1}$ and $\D\subset \C$ is the unit disk both equipped with the standard Liouville structure. Here, $\Lambda$ is the Legendrian lift of the zero section in the page $D^*S^{n-1}$ of the trivial open book decomposition of $\partial(D^*{S^{n-1}}\times \D)$. Then $W_{\Lambda}$ is a torsion cobordism, as the fundamental class of $\Lambda$ is annihilated in the cobordism. Since any overtwisted contact manifold $(Y,\xi)$ can be written as the contact connected sum $(Y\# S^{2n-1},\xi\# \xi_{ot})$, we can produce a torsion cobordism from $(Y,\xi)$ to $Y\#\partial(D^*S^{n-1}\times \D)$.
\end{example}

\subsection{Transverse surgery}
We first recall the contact fiber sum following \cite[Theorem 7.4.3]{Geiges}. Let $M_1,M_2$ be two codimension $2$ contact submanifolds in $Y_1, Y_2$ respectively, both with trivial normal bundles\footnote{\cite[Theorem 7.4.3]{Geiges} works for non-trivial normal bundles, but we will only use the trivial normal bundle case in this paper.}. Let $\phi:M_1\to M_2$ be a contactomorphism. The fiber connected sum $Y_1\#_{\phi}Y_2$ is defined as follows. We may assume that the contact forms $\alpha_1,\alpha_2$ on $Y_1,Y_2$, restricted to tubular neighborhoods of $M_1,M_2$, have the following form:
$$\alpha_1=\beta+r^2\rd\theta \text{ on } M_1\times \D, \quad \alpha_2=\phi_*\beta-r^2\rd\theta \text{ on } M_2\times \D,$$
where $\beta$ is a contact form on $M_1$. 
$$Y_1\#_{\phi}Y_2=(Y_1\backslash M_1\times \D)\cup_{M_1\times \{1\}\times S^1} (M_1\times [-1,1]\times S^1) \cup_{M_1\times \{-1\}\times S^1\stackrel{\phi}{\sim} M_2\times \partial \D} (Y_2\backslash M_2\times \D)$$
with contact forms $\alpha_1$, $\beta+f(r)\rd\theta$, and $\alpha_2$ on the three components, respectively. Here $f(r)$ is a smooth function on $[-1,1]$ satisfying $f(r)=r^2$ for $r$ close to $1$, $f(r)=-r^2$ for $r$ close to $-1$, with $f'(r)>0$ and $f(0)=0$. Alternatively, we blow up $M_*$ in $Y_*$ so that the contact forms on the boundaries $M_*\times S^1$ are $\beta$ and $\phi_*\beta$; that is, the boundaries are round hypersurfaces as in \cite[\S 4.1]{MNW}. The fiber connected sum is obtained by gluing them along the round hypersurface using $\phi$ and the model neighborhood from \cite[Lemma 4.1]{MNW}.

Now suppose $Y$ is a contact manifold with a codimension $2$ contact submanifold $M$ with a trivial normal bundle. On the other hand, given a Liouville filling $\Sigma$ of $M$, $M\times \{0\}$ is clearly a contact submanifold in $\partial (\Sigma \times \D)$ with a trivial normal bundle. This leads to the following definition.
\begin{definition}[Transverse surgery]
Let $\phi$ be a contactomorphism from $M$ to itself. The $(\Sigma,\phi)$-surgery of $Y$ is $\partial (\Sigma \times \D)\#_\phi Y$, denoted by $Y_{M,\Sigma,\phi}$.
\end{definition}
When $\dim Y=3$ and $M$ is a transverse knot, then $Y_{M,\D,\Id}$ corresponds to applying an inadmissible transverse $0$-surgery to $M$ in the sense of \cite[\S 2.3]{zbMATH07076087}. 

The reverse process of transverse surgery involves the round handle attachment introduced by Massot--Niederkr\"uger--Wendl \cite{MNW}. To describe this, we first recall Giroux's \cite{Giroux20} notion of ideal Liouville domains, as presented in \cite{MNW}.
\begin{definition}[Ideal Liouville domain, {\cite[\S 4.2]{MNW}}]
Let $\Sigma$ be a compact $2n$-dimensional manifold with boundary, $\omega$ a symplectic form on the interior $\Sigma^\circ$ of $\Sigma$, and $\xi$ a contact structure on $\partial \Sigma$. The triple $(\Sigma, \omega, \xi)$ is an ideal Liouville domain if there exists an auxiliary $1$-form $\beta$ on $\Sigma^\circ$ such that:
\begin{enumerate}
    \item $\rd \beta = \omega$ on $\Sigma^{\circ}$;
    \item For any (hence every) smooth function $f:\Sigma \to [0, \infty)$ with regular level set $\partial \Sigma= f^{-1}(0)$, the $1$-form $f\beta$ extends smoothly to $\partial \Sigma$ such that its restriction to $\partial \Sigma$ is a contact form for $\xi$.
\end{enumerate}
In this situation, $\beta$ is called a Liouville form for $(\Sigma, \omega, \xi)$.
\end{definition}
Given a Liouville domain $(\Sigma,\beta_{\Sigma})$, the triple $(\overline{\Sigma},\rd(f\beta_{\Sigma}),\ker(\beta_{\Sigma}|_{\partial \Sigma}))$ is an ideal Liouville domain for some positive function $f$ on $\Sigma$. This correspondence is one-to-one up to homotopy of (ideal) Liouville domains. Hence we will not distinguish between Liouville domains and ideal Liouville domains when the context is clear.

Given an ideal Liouville domain $(\Sigma,\omega,\xi)$ and a Liouville form $\beta$, one can endow $\Sigma\times \R$ with the contact structure $\ker(f \rd t +f \beta)$ for any smooth function $f:\Sigma\to[0,\infty)$ with regular level set $f^{-1}(0) = \partial \Sigma$. Over the interior of $\Sigma$, $\ker(f \rd t +f\beta)=\ker(\rd t +\beta)$, so one recovers the standard notion of the contactization of the Liouville manifold defined by $\beta$. On the boundary we have $f\rd t=0$, so the contact hyperplanes are $\xi \oplus T\R$. Any two contact structures obtained in this way from different Liouville forms are isotopic relative to the boundary. Since the contact forms constructed on $\Sigma \times \R$ are $\R$-invariant, one can just as well replace $\R$ by $S^1$. Following \cite[\S 5.3]{MNW}, $\Sigma\times S^1$ with the contact structure defined in this way is called the \emph{Giroux domain} associated to $(\Sigma,\omega,\xi)$. The boundary $\partial \Sigma \times S^1$ with the above contact form is the round hypersurface in \cite[\S 4.1]{MNW}. 

\begin{proposition}\label{prop:transverse_Giroux}
    A $(\Sigma,\phi)$ transverse surgery on $Y$ yields a Giroux domain $\Sigma \times S^1$.
\end{proposition}
\begin{proof}
Let $\lambda_{\Sigma}$ be the Liouville form on $\Sigma$ and $\alpha_M=\lambda_{\Sigma}|_{M=\partial \Sigma}$ the induced contact form. The contact form on $\partial(\Sigma \times \D)$ is given by $\lambda_{\Sigma}+\rd \theta$ on $\Sigma \times S^1$ and $\alpha_M+r^2\rd \theta$ on $M\times \D$, we then need to smooth the corner. Let $\D^*=\D\backslash\{0\}$. The fiber connected sum construction modifies the contact structure on $M\times \D^*$ to $\alpha_M+f(r)\rd \theta$, where $f(r)=r^2$ for $r$ close to $1$, and $f(0)=0$, $f'(r)>0$. The contact structure on $M\times \D^*$ is $\ker (\alpha_M/f(r)+\rd \theta)$, which extends to the real blow-up $M\times [1,0]_r\times S^1$. Therefore, the blow-up of $M$ in $\partial(M\times \D)$ yields $(\Sigma\cup_{M\times \{1\}} M\times [1,0]_r)\times S^1$ with contact form (before smoothing the corner) $(\lambda_{\Sigma}\cup\alpha_M/f(r))+\rd\theta$, which is precisely the Giroux domain associated to the ideal Liouville domain $(\Sigma,\lambda_{\Sigma})$, since $\alpha_M/f(r)$ is an ideal Liouville form on the collar end. From the perspective of blow-up and gluing along round hypersurfaces, we see that transverse surgery yields a Giroux domain $\Sigma \times S^1$.
\end{proof}

For a contact manifold $Y$ with a Giroux domain $\Sigma \times S^1$, one can delete the Giroux domain and blow down the boundary to obtain a new contact manifold following \cite[\S 4.1]{MNW}. Namely, we delete the Giroux domain $\Sigma \times S^1$ from $Y$. Near the boundary $M\times S^1$, by \cite[Lemma 4.1]{MNW}, after reversing the orientation on $S^1$ from the discussion before \Cref{prop:transverse_Giroux}, the contact form is given by $s\rd t+\alpha_M$ on $[0,\epsilon)_s\times S^1_t\times M$ for a contact form $\alpha_M$ on $M=\partial \Sigma$. Blowing down the Giroux domain consists of gluing $(\D \times M, \alpha_M+r^2\rd \theta)$ for $r^2<\epsilon$ to $(0,\epsilon)_s\times S^1_t\times M$ via $r^2=s$. This blow-down produces a contact submanifold $M$ from $\{0\}\times M$, which we refer to as the \emph{belt} of the blow-down. It is clear from the construction that blowing down the Giroux domain in $Y_{M,\Sigma,\phi}$ recovers $Y$. Moreover, this blow-down process can be described by a strong cobordism obtained from a round handle attachment as in \cite[\S 5.1]{MNW}.

\begin{proposition}[{\cite[Theorem 5.1]{MNW}}]\label{prop:cap}
Assume $Y$ has a Giroux domain $\Sigma\times S^1$, and let $Y'$ denote the manifold obtained by blowing down. There is a strong cobordism $(W,\omega)$ from $Y$ to $Y'$ obtained by attaching $\Sigma \times \D$ to $Y$ along the Giroux domain. Moreover, $\Sigma\times \{0\}\subset W$ is a codimension $2$ symplectic submanifold whose boundary is the belt of the blow-down in $Y'$. The symplectic form $\omega$ has a primitive defined outside $\Sigma\times \{0\}\subset W$ that extends the contact form on $Y$. We call $\Sigma \times \{0\}$ the non-exact co-core of the round handle.
\end{proposition}

\begin{example}[$S^1$-invariant contact structures]\label{ex:S1}
In \cite{DingGeiges:S1}, Ding and Geiges showed that any $S^1_{\tau}$-invariant contact form $\alpha$ on $S\times S^1_{\tau}$ admits a decomposition 
$(W_+\cup_{\Gamma} (-W_-))\times S^1_{\tau}$ with 
\begin{equation*}
\alpha=f\rd \tau +\beta, \quad f\in C^\infty(S), \quad \beta \in \Omega^{1}(S),
\end{equation*}
such that on the interior $W^{\circ}_{\pm}$, we have $\pm f>0$ and $\beta_{\pm} = \pm \beta/f \in \Omega^{1}(W_{\pm}^{\circ})$ defining Liouville forms. The symplectic forms $d\beta_{\pm}$ then orient $W_{+}$ (respectively $W_{-}$) the same as (opposite to) the orientation of $S$. It also follows that $\alpha$ is a contact form when restricted to $\Gamma=\{f=0, \tau=\tau_{0}\}$, and we write $\xi_{\Gamma} = \ker \alpha|_{\Gamma}$. With this contact structure, $(\Gamma, \xi_{\Gamma})$ is the ideal boundary of $(W_{\pm}, \beta_{\pm})$. In other words, $(W_{\pm}, \beta_{\pm},\ker \beta|_{\Gamma})$ are ideal Liouville domains.

Conversely, let $(W_+,\omega_+,\xi_+)$ and $(W_-,\omega_-,\xi_-)$ be two ideal Liouville domains. Given a contactomorphism $\phi:(\partial W_+,\xi_+)\to (\partial W_-,\xi_-)$, we can build an $S^1$-invariant contact structure on $(W_+\cup_{\phi}(-W_-))\times S^1$ by gluing the two Giroux domains $W_+\times S^1$ and $W_-\times S^1$ using $\phi$, see e.g.\ \cite[\S 2.2]{AZ}. We write $\DG(W_+,W_-,\phi)$ for such an $S^1$-invariant contact manifold. Then we can view it as obtained from transverse surgeries:
$$\DG(W_+,W_-,\phi) = \partial(W_+\times \D)_{\Gamma,W_-,\phi^{-1}}= \partial(W_-\times \D)_{\Gamma,W_+,\phi}.$$
The cobordism in \Cref{prop:cap} produces cobordisms from $\DG(W_+,W_-,\phi)$ to either $\partial(W_+\times \D)$ or $\partial (W_-\times \D)$.
\end{example}

In higher dimensions, we not only have greater flexibility in choosing symplectic fillings for transverse surgery, but also the extra freedom of choosing the contactomorphism $\phi$, although our knowledge of the contact mapping class group is very limited. One special case of contactomorphisms are those lifted from the symplectic mapping class groups of certain contactization pieces. The following is a prominent example.
\begin{example}[Bourgeois contact structures]\label{ex:Bourgeois}
Let $(\Sigma,\beta_{\Sigma})$ be a Liouville domain and $\phi\in \symp_c(\Sigma, \beta_{\Sigma})$ a compactly supported symplectomorphism. Then $\phi^*\beta_{\Sigma}=\beta_{\Sigma} + \eta$, where $\eta$ is a closed $1$-form that is zero near $\partial \Sigma$. The Thurston--Winkelnkemper construction endows the total space $Y=(\partial \Sigma\times \D) \cup \Sigma_{\phi}$ of the open book with a contact form $\alpha$, where $\Sigma_{\phi}=\Sigma \times [0,1]_t/(x,1)\sim (\phi(x),0)$ is the mapping torus. More precisely,
\begin{equation*}
    \alpha = \begin{cases}
    \beta_{\Sigma}|_{\partial \Sigma} + Kr^2 \rd \theta & \text{along } \partial \Sigma\times \D, \\
    \beta_{\Sigma} + b(\theta)\eta+2\pi K \rd t & \text{along } \Sigma_{\phi}, \theta=2\pi t,
    \end{cases}
\end{equation*}
for $K\gg 0$ and $b=b(\theta)$ a function equal to $1$ near $\theta=2\pi$ and $0$ near $\theta=0$. Such a contact open book is denoted by
\begin{equation*}
    (Y,\xi) = \Openbook(\Sigma, \phi)
\end{equation*}
Strictly speaking, the contact open book has a corner at $S^1 \times \partial \Sigma$, since a neighborhood with the contact form is identified with a neighborhood of the corner in the boundary of the product $(\Sigma\times \D, Kr^2\rd \theta+\lambda)$. One can round the corner following, e.g., \cite[\S 2.1]{zbMATH07673358}.

The smooth structure near the corner is inherited from the submanifold smooth structure in $\Sigma \times \D$ after smoothing the corner. Therefore, the projection to $\D$ is a smooth map; we use $\Phi_1,\Phi_2$ to denote the two coordinates. Hence $\Phi=(\Phi_{1}, \Phi_{2}): Y \rightarrow \C$ has $0$ as a regular value with inverse image $\partial \Sigma \times \{0\}$, which is the binding of the open book for $Y$. Bourgeois \cite{Bourgeois02} showed that 
$$\alpha_{\mathrm{BO}} := \alpha + \Phi_1 \rd x + \Phi_2 \rd y$$
is a contact form on $Y \times T^2= \Openbook(\Sigma,\phi)\times T^2$, where $(x,y)$ are coordinates on $T^2$. The associated contact structure is independent of choices (e.g.\ $K$, $b$, and $\phi$ up to isotopy). We refer to such contact manifolds as \emph{Bourgeois contact manifolds}, denoted by $\BO(\Sigma,\phi)$. Observe that this contact form (and hence the contact structure) is $T^2$-invariant. In particular, each $Y \times S^1_{x}\times \{y\}$ is a convex hypersurface with respect to the contact vector field $\partial_{y}$. Therefore we have \cite[Proposition 2.6]{AZ}
$$\BO(\Sigma,\phi)=\DG(W, W,\phi_{\BO}),$$
where $W=\Sigma \times D^*S^1$, and $\phi_{\BO}$ on $\partial(\Sigma \times D^*S^1)$ is given by $\phi^{-1}$ on $\Sigma \times \{1\}\times S^1$ and identity elsewhere. By \Cref{ex:S1}, we can view
$$\BO(\Sigma,\phi)=\partial (W\times \D)_{\partial(\Sigma \times D^*S^1), W, \phi_{\BO}}.$$
Moreover, we obtain a strong cobordism from $\BO(\Sigma,\phi)$ to $\partial (W\times \D)$ via \Cref{prop:cap}, which is the main topological input for the works in \cite{BGM,bowden2022tight}.
\end{example}

\subsection{Spinal open books}
Spinal open books generalize the concept of contact open books and were introduced by Lisi, Van Horn-Morris, and Wendl \cite{lisi2018symplectic}. Heuristically, spinal open books arise as the contact boundary of a Lefschetz fibration over a surface with boundary. Such contact manifolds, especially in dimension $4$, have been studied systematically in \cite{lisi2018symplectic,lisi2020symplectic,MR4278702, arXiv:2410.10697}. Moreover, spinal open books can be used to construct contact manifolds with infinitely many fillings \cite{zbMATH06537657,zhou2023note}. 

In this paper, we will only consider special cases of spinal open books. Let $S_{k}$ be $S^2$ with $k$ disks removed, viewed as a Liouville surface with Liouville form $\lambda_S$. Let $M$ be a contact manifold, let $\Sigma_1,\ldots,\Sigma_{k}$ be $k$ Liouville fillings of $M$, each equipped with a contactomorphism $\psi_i:\partial \Sigma_i\to M$, and let $\phi_i\in \pi_0(\symp_c(\Sigma_i))$. We can endow the space
$$\left(S_{k}\times M \displaystyle \bigcup_{1\le i \le k}(\Sigma_i)_{\phi_i}\right)/\sim$$
with a contact structure, where the gluing $\sim$ is given by 
$$S^1\times M \to \partial \Sigma_{\phi_i}=S^1\times \partial \Sigma_i, \quad (t,x)\mapsto (t,\psi_i^{-1}(x)),$$
and where $S^1$ is the $i$th boundary component of $S_k$. We endow this space with a contact structure via a generalization of the Thurston--Winkelnkemper construction, see \cite[\S 2.3]{lisi2018symplectic}. More precisely, let $\alpha_M$ be a contact form on $M$ and $\lambda_i$ a Liouville form on $\Sigma_i$ that restricts to $\phi_i^*\alpha_M$ on the boundary. Then the contact form is given by 
\begin{equation}\label{eqn:TW}
    \alpha = K\lambda_S+\alpha_M \text{ on } S_{k}\times M, \quad \alpha = K\rd t + \lambda_{i}+\beta(t)\eta_{i} \text{ on } (\Sigma_i)_{\phi_i}.
\end{equation}
Here we assume that $\lambda_S$ restricts to $\rd t$ on each boundary component $S^1=[0,1]_t/0\sim1$ and that $\phi_i^*\lambda_i=\lambda_i+\eta_i$. Strictly speaking, we need to smooth the corners for both the smooth structure and the contact form as before. We use $\SOB(\bm{\Sigma},\bm{\phi},\bm{\psi})$ (with bold font) to denote this contact manifold. By a standard argument using Gray stability, the contact structure up to isotopy is well-defined for all $K\gg 0$ and for $\lambda_S,\alpha_M,\lambda_i,\psi_i,\phi_i$ up to suitable homotopy. If $\Sigma_1=\ldots=\Sigma_k=\Sigma$ and $\psi_1=\ldots=\psi_k$, then we use $\SOB(\Sigma,\bm{\phi})$ to denote the special case. When $k=1$, this specializes to the contact open book $\OB(\Sigma,\phi)$. We can suppress $\psi$ here, as different $\psi$ yield contactomorphic open books.

\begin{remark}\label{rmk:spinal}
    A few remarks on spinal open books are in order.
    \begin{enumerate}
        \item $S_k$ is called a \emph{vertebra} and $\Sigma_i$ is called a \emph{page}. The region $S_k\times M$ is the \emph{spine} of the spinal open book, and $\cup \Sigma_{\phi_i}$ is called the \emph{paper}. 
        \item Unlike the special case $\SOB(\bm{\Sigma},\bm{\phi},\bm{\psi})$, there can be multiple vertebrae, and each page can be glued to several vertebrae if the page has multiple boundary components. 
        \item In dimension $3$, there is a notion of multiplicity \cite[Definition 1.3]{lisi2018symplectic} since a boundary component of the paper region is $T^2$. This flexibility leads to the notion of rational open books in \cite{BEV}. However, in higher dimensions for general pages\footnote{Analogs of rational open books would exist, if the boundary of the page admits a Reeb flow generating an $S^1$-action, e.g., prequantization circle bundles generalizing $S^1$ in dimension $3$.}, such a notion no longer exists.
        \item The presence of $\psi_i$ was implicit in \cite{lisi2018symplectic} because there are no non-trivial co-orientation preserving contactomorphisms of $S^1$ up to isotopy. However, this is no longer the case in higher dimensions. In dimension $3$, if the page has multiple boundary components, we can still have non-trivial $\psi_i$ by permutations, see e.g.\ \Cref{ex:BO=SOB}.
    \end{enumerate}
\end{remark}

\begin{example}\label{ex:BO=SOB}
    $\DG(W_+,W_-,\phi)$ can be viewed as a spinal open book $\SOB(\{W_+,W_-\},\{\Id,\Id\},\{\phi,\Id\})$ with $M=\partial W_-$. In particular, $\BO(\Sigma,\phi) = \SOB(\{\Sigma \times D^*S^1,\Sigma \times D^*S^1\},\{\Id,\Id\},\{\phi_{\BO},\Id\})$. As a special case, $(T^3,\xi_k)$ in \cite{Eli96} can be viewed as $\BO(K,\phi)$ from the open book $S^1\to S^1, z\mapsto z^k$, where $K$ is a set of $k$ points and $\phi$ is a cyclic permutation. Then $(T^3,\xi_k)=\SOB(\{K \times D^*S^1,K \times D^*S^1\},\{\Id,\Id\},\{\phi_{\BO},\Id\})$, where $\phi_{\BO}$ is the permutation on $2k$ copies of $S^1$ that makes $\SOB(\{K \times D^*S^1,K \times D^*S^1\},\{\Id,\Id\},\{\phi_{\BO},\Id\})$ connected.
\end{example}

\begin{example}\label{ex:SOB_Liouville_sum}
    Given a spinal open book $\SOB\left(\bm{\Sigma}, \bm{\phi},\bm{\psi}\right)$, without loss of generality we may assume $\partial \Sigma_1=M$ and $\psi_1=\Id$. Since $M\times \{p\}\subset M\times S_k$ is a contact submanifold, given a Liouville domain $\Sigma_{k+1}$ and $\psi_{k+1}:\partial \Sigma_{k+1}\to M$, we have 
    $$\SOB\left(\bm{\Sigma}, \bm{\phi},\bm{\psi}\right)_{M,\Sigma_{k+1},\psi_{k+1}}=\SOB\left(\bm{\Sigma}\cup \{\Sigma_{k+1}\}, \bm{\phi}\cup \{\Id\},\bm{\psi}\cup \{\psi_{k+1}\}\right).$$
    The reverse process, via \Cref{prop:cap}, produces a strong cobordism between them. A basic example is the co-sphere bundle $S^*T^2$, viewed as $\SOB(D^*S^1,\{\Id,\Id\})$, which has a strong cobordism to $\OB(D^*S^1,\Id)=S^1\times S^2$.
\end{example}

\begin{example}\label{ex:+1}
Let $\Lambda$ be an exact Lagrangian sphere in $\Sigma_i$. Up to changing the Liouville form $\lambda_i$ near $\Lambda$ by a Liouville homotopy, we may assume $\lambda_i|_{\Lambda}=0$. As a consequence, $\{0\}\times \Lambda \subset (\Sigma_i)_{\phi_i}$ is a Legendrian sphere. We can apply $(+1)$ surgery to this Legendrian sphere in $(\Sigma_i)_{\phi_i}\subset\SOB\left(\bm{\Sigma}, \bm{\phi},\bm{\psi}\right)$. By \cite[Theorem 4.6]{zbMATH06826753}, the resulting contact manifold is $\SOB\left(\bm{\Sigma}, \bm{\phi'},\bm{\psi}\right)$ with $\phi'_j=\phi_j$ for $j\ne i$ and $\phi'_i=\phi^{-1}_{\Lambda}\circ \phi_i$, where $\phi_{\Lambda}$ is the Dehn--Seidel twist around the Lagrangian sphere $\Lambda$.
\end{example}

\subsection{Liouville connected sum}
In this part, we review Avdek's Liouville connected sum \cite{Russell}. 
\begin{definition}
    Let $(Y^{2n-1},\xi)$ be a contact manifold. A Liouville hypersurface $i:\Sigma \subset Y$ is a Liouville domain $(\Sigma,\lambda)$ for which there exists a contact form $\alpha$ on $Y$ such that $i^*\alpha = \lambda$.
\end{definition}
For example, a page of a spinal open book is a Liouville hypersurface. Transverse surgery yields a Liouville hypersurface $\Sigma$ in $Y_{M,\Sigma,\phi}$. 

\begin{example}[{\cite[Example 1.5]{Russell}}]
    Let $S\subset Y^{2n-1}$ be an isotropic submanifold, and assume the symplectic normal bundle $TS^{\rd \alpha|_{\xi}}/TS \subset \xi$ is a trivial symplectic bundle. Then by Morse's trick, a neighborhood of $S$ is modeled on a neighborhood of the zero section in the contactization $(\R_z\times T^*S\times \R^{2n-2-2\dim S}, \rd z+\lambda_{\mathrm{can}}+\lambda_{\std} )$, where $\lambda_{\mathrm{can}}$ is the canonical Liouville form on $T^*S$ and $\lambda_{\std}$ is the standard Liouville form on $\R^{2n-2-2\dim S}$. This produces a Liouville hypersurface of the form $D^*S\times \D^{n-1-\dim S}$. A special case is the Liouville hypersurface $\D^{n-1}$ in any contact manifold, which is the positive half of the convex boundary after deleting a Darboux ball.
\end{example}

The main theorem of \cite{Russell} is that we can cut and paste contact manifolds along Liouville hypersurfaces. More precisely, given two disjoint Liouville hypersurfaces $i_1,i_2:(V,\lambda)\to (Y,\xi)$, we can cut along the images of $i_1$ and $i_2$ to produce two isomorphic convex hypersurface boundaries. Gluing these two convex hypersurfaces yields a new contact manifold, denoted by $\#_{((V,\lambda),(i_1,i_2))}(Y,\xi)$. This procedure is called \emph{Liouville connected sum}. The significance of this construction is the following natural cobordism:

\begin{proposition}[{\cite[Theorem 1.8]{Russell}}]
    There is a Liouville cobordism from $Y$ to $\#_{((V,\lambda),(i_1,i_2))}(Y,\xi)$. If $V$ is Weinstein, then so is the cobordism. More precisely, a $(k+1)$-Weinstein handle in the cobordism corresponds to a $k$-handle in $V$.
\end{proposition}

Many classical constructions, e.g., Weinstein handle attachment, can be viewed as Liouville connected sums, see \cite[Example 1.7]{Russell}. Another example is the natural cobordism between spinal open books.

It was shown in \cite{Russell,Baldwin,BEV,Klukas} that there is an exact cobordism from $\OB(\Sigma,\phi)\sqcup \OB(\Sigma,\psi)$ to $\OB(\Sigma,\phi\circ \psi)$, which is a Weinstein cobordism if $\Sigma$ is Weinstein. In terms of the Liouville connected sum \cite[Proposition 8.2]{Russell}, $\OB(\Sigma,\phi\circ \psi)$ is $\#_{(\Sigma,(i_1,i_2))}(\OB(\Sigma,\phi)\sqcup \OB(\Sigma,\psi))$, where $i_1,i_2:\Sigma \to \OB(\Sigma,\phi), \OB(\Sigma,\psi)$ are inclusions of pages. Similar to the proof of \cite[Proposition 8.2]{Russell}, we have the following generalization. 

\begin{proposition}\label{prop:SOB_sum}
Given two spinal open books $\SOB\left(\bm{\Sigma}, \bm{\phi},\bm{\psi}\right)$, $\SOB\left(\bm{\Sigma'}, \bm{\phi'},\bm{\psi'}\right)$, assume that $\Sigma_1=\Sigma_1'=\Sigma$ and $\psi_1=\psi'_1=\psi$, and $i,i':\Sigma \to \SOB\left(\bm{\Sigma}, \bm{\phi},\bm{\psi}\right), \SOB\left(\bm{\Sigma'}, \bm{\phi'},\bm{\psi'}\right)$ are inclusions of the first pages as Liouville hypersurfaces. Then
\begin{align*}
 \#_{\Sigma,(i,i')} & \left(\SOB\left(\bm{\Sigma}, \bm{\phi},\bm{\psi}\right) \sqcup \SOB\left(\bm{\Sigma'}, \bm{\phi'},\bm{\psi'}\right)\right)=\\ 
 &\SOB\left(\bm{\Sigma}\cup \left(\bm{\Sigma'}\backslash\{\Sigma'_1\}\right), \{\phi_1\circ \phi'_1\}\cup \left(\bm{\phi}\backslash\{\phi_1\}\right)\cup  \left(\bm{\phi'}\backslash\{\phi'_1\}\right),\bm{\psi}\cup \left(\bm{\psi'}\backslash\{\psi'_1\}\right)\right).   
\end{align*}
\end{proposition}
\begin{remark}
    We can also apply the Liouville connected sum to embeddings of $\Sigma$ into the same spinal open book, either in the same paper component or in different paper components. We may also twist the Liouville embedding by elements in $\symp_c(\Sigma)$. This will result in a spinal open book with possibly different vertebrae and different monodromy. For example, we can perform the Liouville connected sum for two pages in $\OB(\Sigma,\phi)$ with suitable twists from $\symp_c(\Sigma)$ to obtain $\SOB(\Sigma, \{\phi_1,\phi_2\})$, such that $\phi_1\circ \phi_2=\phi$. Moreover, we can twist the Liouville embedding by elements in $\symp(\Sigma)$ that restrict to a contactomorphism on the boundary. The Liouville connected sum will then also induce twists on the $\{\psi\}$ part.
\end{remark}

\subsection{Weak fillings and stable fillings}
Here we recall the notion of weak filling and stable filling from \cite{MNW,LW}.
\begin{definition}[{\cite[Definition 4]{MNW}}]
Let $\xi$ be a co-oriented contact structure on a manifold $Y$. Let $(W, \omega)$ be a symplectic manifold with $\partial W = Y$ as oriented manifolds. We say that $(W, \omega)$ is a weak filling of $(Y, \xi)$, and $\omega$ weakly dominates $\xi$ if 
$$\alpha \wedge (\rd\alpha + \omega_{\xi})^{n-1},\quad \alpha \wedge \omega_\xi^{n-1}$$
are positive volume forms on $Y$ for any positive contact form $\alpha$ and $\omega_{\xi}=\omega|_{\xi}$. 
\end{definition}
These conditions are equivalent to the condition that $t\omega + (1-t)\rd \alpha$ defines a symplectic structure on $\xi=\ker \alpha$ for any $t\in [0,1]$. In dimension $3$, they are equivalent to the compatibility of the orientations on $\xi$ induced by $\omega$ and $\rd \alpha$.

A closely related notion from the perspective of symplectic field theory is stable fillings of contact manifolds, see \cite[Definition 1.9]{LW}. A fundamental theorem relating them is that a weak filling can be deformed into a stable filling \cite[Proposition 6]{MNW}. The following simple lemma is needed when we consider weak fillings under contact surgeries and cobordisms. 
\begin{lemma}\label{lemma:weak_surgery}
Assume $W$ is an exact cobordism from $Y_-$ (concave boundary) to $Y_+$ (convex boundary) such that $H^2(W;\R)\to H^2(Y_-;\R)$ is surjective. If $Y_-$ is weakly fillable, then so is $Y_+$.
\end{lemma}
\begin{proof}
Let $(V,\omega)$ be a weak filling of $(Y,\alpha_-)$. By \cite[Lemma 1.6]{MNW}, a neighborhood of the boundary of $V$ is modeled on $(-\epsilon,0]_t\times Y_-, \rd(t\alpha_-)+\omega|_{Y_-})$. On the exact cobordism $(W,\lambda)$, we have a closed $2$-form $\omega_W$ extending $\omega|_{Y_-}$ and a collar neighborhood of $Y_-$ modeled on $([1,1+\epsilon)_r\times Y_-,\lambda = r\alpha_-)$. We may assume $\omega_W$ on this collar neighborhood is $\omega|_{Y_-}$. Now we choose a function $f:[1,1+\epsilon)_r\to \R_{\ge 0}$, such that $f=r-1$ near $1$ and $f'>0$ and $f=Kr$ for $K\gg 0$ near $1+\epsilon$. Then by setting $t=1-r$, we can glue $V$ to $W$ equipped with a symplectic form $\omega_W+\rd(\frac{f(r)}{r}\lambda).$ To see it is a symplectic form, on the collar neighborhood in $W$, the form is $\omega|_{Y_-}+\rd(f(r)\lambda)$, which is symplectic by the weak filling property. Outside the collar, the form is $K\rd \lambda+\omega_W$, which is symplectic if $K\gg 0$. Finally, $K\rd \lambda+\omega_V$ weakly dominates the contact structure on $Y_-$ if $K\gg 0$. Hence, $Y_+$ is weakly fillable.
\end{proof}

\begin{remark}
    It is clear from the proof that \Cref{lemma:weak_surgery} holds for a symplectic cobordism $(W,\omega)$ that is strong on the negative boundary and weak on the positive boundary. Moreover, the weak filling can be glued as long as $\omega|_{Y_-}$ extends to a closed $2$-form on $W$. 
\end{remark}

\begin{example}\label{ex:weak_filling}
By \cite[Theorem A]{zbMATH07162211}, if $(W,\omega)$ is a weak filling of $\OB(\Sigma,\phi)$, then $(W\times T^2, \omega\oplus \rd\mathrm{vol}_{T^2})$ is a weak filling of $\BO(\Sigma,\phi)$.
\end{example}

\subsection{Constructions of torsion cobordisms}\label{ss:construct_torsion}
In this part, we will introduce several constructions of torsion cobordisms. 
\begin{example}\label{ex:SOB}
 Let $\Sigma$ be a Liouville domain and $\Lambda\subset \Sigma$ an exact Lagrangian sphere such that $[\Lambda]\ne 0$ in $H_*(\Sigma;\Q)$. Following \Cref{ex:+1}, the surgery cobordism from $(+1)$ surgery yields an exact torsion cobordism from $\OB(\Sigma,\phi^{-1}_{\Lambda})$ to $\OB(\Sigma,\Id)$. More generally, by transverse surgery and blowing down the Giroux domains in \Cref{prop:transverse_Giroux}, we obtain a strong torsion cobordism from the spinal open book $\SOB\left(\{\Sigma,\Sigma_2,\ldots,\Sigma_k\},\{\phi^{-1}_{\Lambda},\Id,\ldots,\Id\},\{\Id, \ldots, \Id\}\right)$ to $\OB(\Sigma,\Id)$. The homology class $A$ in \Cref{def:torsion_cobordism} is the fundamental class of $\Lambda$. 
\end{example}

\begin{example}\label{ex:DG}
    Let $V_+,V_-$ be two Liouville domains, and $\phi:\partial V_+\to \partial V_-$ a contactomorphism. Assume there is a homology class $x\ne 0\in H_*(V_+;\Q)$ such that $x=0$ in $H_*(V_+\cup_{\phi}V_-;\Q)$, where $V_+\cup_{\phi}V_-$ is the space obtained from $V_+\sqcup V_-$ by identifying $x\in \partial V_+$ with $\phi(x)\in \partial V_-$. By the Mayer--Vietoris sequence, this is equivalent to the existence of a class $x\in H_*(\partial V_+;\Q)$ such that $x\ne 0\in H_*(V_+;\Q)$ and $\phi_*(x)=0$ in $H_*(V_-;\Q)$. There are two ways to produce such examples: one from the discrepancy between $V_+$ and $V_-$, and the other from non-trivial $\phi$.
    \begin{enumerate}
        \item\label{DG1} If $M$ has two exact fillings $V_{\pm}$ and there is a class $x\in H_*(M;\Q)$ such that $x\ne 0\in H_*(V_+;\Q)$ and $x=0\in H_*(V_-;\Q)$. For example, if $M=S^1\sqcup S^1$, $V_+$ is a connected surface with boundary $M$, and $V_-=S_1\cup S_2$ where each $S_*$ is a surface with one boundary component. Then any boundary component of $V_+$ represents such a homology class. Examples with algebraic $1$-torsion and no Giroux torsion from \cite[Theorem 3]{LW} belong to this class.

        More generally, in higher dimensions, let $M$ be the contact boundary of $\Sigma_{1,1}\times V$, where $\Sigma_{1,1}$ is the $2$-torus with one disk removed and $V$ is a $2n$-dimensional Weinstein domain such that $H_{n-1}(V;\Q)=0$ and $H_n(V;\Q)\ne 0$. Let $\phi\in Symp_c(\Sigma)$ be such that $\phi_*\ne \Id$ on $H_n(V;\Q)$. Then $M$ has two Weinstein fillings: $\Sigma_{1,1}\times V$ and $V_{\phi\vee \Id}$, which is a $V$-fibration over $\Sigma_{1,1}$ with monodromy around the two circles in $\Sigma_{1,1}$ given by $\phi$ and $\Id$, as in \cite{zhou2023note}. Note that the inclusion $V\subset S^1\times V \subset M$ induces an injection $H_n(V;\Q)\to H_n(M;\Q)$, since the composition $V \subset M \subset \Sigma_{1,1}\times V$ induces an injection on homology. By the exact sequence following \cite[Lemma 2.2]{zhou2023note}, we have $\dim H_n(V_{\phi\vee\Id};\Q)=\dim \ker(\phi_*-\Id)|_{H_n(V;\Q)}<\dim H_n(V;\Q)$, so we can find such a class $x\in H_n(M;\Q)$. It is easy to construct such $\phi$ from compositions of Dehn-Seidel twists, see \cite{zhou2023note}.
        \item\label{DG2} If the inclusion $\Sigma \to \OB(\Sigma,\phi)$ does not induce an injection on rational homology, then the strong cobordism from $\BO(\Sigma,\phi)$ to $\OB(\Sigma\times D^*S^1,\Id)$ in \Cref{ex:Bourgeois} is a torsion cobordism. Let $x\in H_*(\Sigma;\Q)$ be a class that vanishes in $\OB(\Sigma,\phi)$. Then $x=0$ in $\BO(\Sigma,\phi)=\OB(\Sigma,\phi)\times T^2$ via the inclusion $\Sigma\subset \Sigma \times D^*S^1 \subset \DG(\Sigma \times D^*S^1,\Sigma\times D^*S^1,\phi_{\BO})$. Then $x$, viewed as a non-trivial homology class in $H_*(\Sigma \times D^*S^1;\Q)$, is trivial in $\BO(\Sigma,\phi)$. In this case, $x\times [S^1] \ne 0 \in H_*(\Sigma\times D^*S^1;\Q)$ is also zero in $H_*(\BO(\Sigma,\phi);\Q)$. This is the topological input for the main theorem in \cite{BGM}. The key player here is the non-trivial contactomorphism $\phi_{\BO}$. We can construct more such non-trivial contactomorphisms following the construction of $\phi_{\BO}$. For $\phi \in Symp_c(\Sigma)$, define the variation of $\phi$ \cite[\S 2.4]{zbMATH07738054} by
        $$\mathrm{var}(\phi):H_*(\Sigma,\partial\Sigma;\Q)\to H_*(\Sigma), \quad [c]\mapsto [\phi_*(c)-c],$$
        where $c$ is a relative cycle representing the homology class $[c]$. By \cite[Lemma 2.10]{zbMATH07738054}, if $\Sigma$ is Weinstein, then $H_*(\Sigma;\Q)\to H_*(\OB(\Sigma,\phi);\Q)$ is injective if and only if $\mathrm{var}(\phi)=0$.
        
        Let $S$ be a connected Liouville surface with at least two boundary components, and let $\partial_1S\simeq S^1$ be one of them. Then define a contactomorphism $\phi_S:\partial(\Sigma\times S) \to \partial(\Sigma\times S)$ by
        $$\phi_S (t,x)=(t,\phi(x)) \text{ on } \partial_1 S\times \Sigma, \qquad \phi_S = \Id \text{ on the rest of the domain.}$$
        Then $\phi_{S}$ satisfies the required condition if $\mathrm{var}(\phi)\ne 0$, and $\phi_{\BO}$ is a special case of this construction. To see this, assume $\mathrm{var}(\phi)(c)=x\ne 0 \in H_*(\Sigma;\Q)$. Let $I$ be an interval in $S$ connecting a point in $\partial_1S$ to a point in a different boundary component of $S$. Then $\partial(I\times c)$ is a closed chain representing a class $[y]$ in $H_*(\partial(\Sigma\times S);\Q)$ that is zero in $H_*(\Sigma\times S;\Q)$. Then $(\phi_S)_*([y])=[y]+[x]$, where $[x]$ is viewed as a class in $H_*(\partial(\Sigma \times S);\Q)$ via the inclusion $\Sigma \subset \Sigma \times \partial_1S \subset \partial(\Sigma\times S)$, and $[x]$ is non-zero in $H_*(\Sigma \times S;\Q)$. Similarly, we can consider $\phi^{-1}_*[x]$ as a class in $H_*(\partial(\Sigma \times S);\Q)$ via the same inclusion, and $\phi^{-1}_*[x]$ is non-zero in $H_*(\Sigma \times S;\Q)$. Consequently, $[y]-\phi^{-1}_*[x]$ is non-zero in $H_*(\Sigma\times S;\Q)$, but $(\phi_S)_*([y]-\phi^{-1}_*[x])=[y]$, which is zero in $H_*(\Sigma\times S;\Q)$; hence the claim holds. When $\Sigma$ is Weinstein of dimension $2n$, such a class $x$ must have degree $n$.
    \end{enumerate}
    With such data, the strong cobordism from $\DG(V_+,V_-,\phi)$ to $\partial(V_+\times \D)$ is a torsion cobordism, where the homology class $A$ in \Cref{def:torsion_cobordism} is the class $x\in H_*(V_+;\Q)$. Since $x$ comes from $H_*(\partial V_+;\Q)$, by the Mayer--Vietoris sequence we have $x=0$ in $\DG(V_+,V_-,\phi)$. Consequently, $x=0$ in the homology of the cobordism. 
\end{example}

\begin{example}\label{ex:generalization}
   There are various generalizations of the constructions above to produce torsion cobordisms:
   \begin{enumerate}
       \item In case \eqref{DG1} of \Cref{ex:DG}, we can use the construction in \cite{zbMATH07197504} to produce many contact manifolds with multiple Weinstein fillings that satisfy the conditions in \eqref{DG1} of \Cref{ex:DG}.  
       \item In case \eqref{DG2} of \Cref{ex:DG}, we can consider more general spinal open books than $\partial(\Sigma \times S)$ to produce such $\phi_S$. We may also replace $S$ by Liouville domains with disconnected boundary as in \cite[Theorem C]{MNW}.
       \item Using surgery cobordisms and round handle attachments (see \Cref{prop:cap}), we can apply coisotropic surgeries or transverse surgeries to the concave boundary of a torsion cobordism to obtain new contact manifolds that are the concave boundary of a torsion cobordism. 
       \item\label{surgery} If we can apply a subcritical isotropic surgery to $\DG(V_+,V_-,\phi)$ outside of $V_-\times S^1$, and if we assume that $V_+$ is Weinstein, then we can switch the order of subcritical surgery and round handle attachment to $V_-\times S^1$. In many cases, this yields a torsion cobordism from $\DG(V_+,V_-,\phi)$ after the isotropic surgery to $\partial(V_+\times \D)$ after the isotropic surgery, which is still of the form $\partial(V\times \D)$ by \cite[Theorem 14.16]{CieEli}. For example, exotic spheres $(S^{2n+1},\xi_{ex})$ in \cite{bowden2022tight} for $n\ge 3$ are realized as the negative boundary of a torsion cobordism in this way by \cite[Proposition 2.13]{bowden2022tight}.
   \end{enumerate}
\end{example}

\subsection{Variants of torsion cobordism}
In this part, we introduce several variants of torsion cobordisms and their constructions. The concave boundaries of these variants also have finite algebraic planar torsion; see \S \ref{s:torsion}.

\subsubsection{Cobordisms to flexibly fillable contact manifolds}
A contact manifold is called flexibly fillable if it admits a flexible Weinstein filling, i.e., a Weinstein domain constructed from subcritical handle attachments and critical handle attachments along loose Legendrians (flexible surgeries). Such contact manifolds, when they have vanishing first Chern class, also exhibit strong uniqueness properties for their fillings \cite{zbMATH07367119,zbMATH07706508,intersection}. This leads to the following definition.

\begin{definition}
        A torsion cobordism of type II is a strong cobordism $W$ such that
        \begin{enumerate}
            \item The convex boundary $\partial_+ W=Y_0\#Y_1$ with the property that $Y_0$ is flexibly fillable with $c_1^{\Q}(Y_0)=0$. 
            \item There exists a homology class $A\in H_k(Y_0;\Q)\hookrightarrow H_k(Y_0;\Q)\oplus H_k(Y_1;\Q)=H_k(Y_0\#Y_1;\Q)$ for $k\ge 2$\footnote{This is needed in \eqref{eqn:vdim}.} that maps to zero in $H_k(W;\Q)$ but does not map to zero under $H_k(Y_0;\Q)\to H_k(W_0;\Q)$, where $W_0$ is a flexible filling of $Y_0$.
        \end{enumerate}
\end{definition}
For example, let $Y$ be a flexibly fillable contact manifold with $c_1^*(Y)=0$, and let $\Gamma \subset Y$ be a codimension $2$ contact submanifold with trivial normal bundle. Let $W$ be an exact filling of $\Gamma$ with $c_1^{\Q}(W)=0$, and suppose there is a homology class $x\in H_*(\Gamma;\Q)$ of degree $|x|\ge 2$ that is non-zero in $H_*(W_0;\Q)$ for some flexible filling $W_0$ of $Y$, but $x$ becomes zero in $Y_{\Gamma,W,\Id}$. If we also have $H^1(W;\Z)\to H^1(\Gamma;\Z)$ is surjective, then $Y_{\Gamma,W,\Id}$ admits a torsion cobordism of type II to $Y$ via the round handle attachment in \Cref{prop:cap}. 

\begin{example}\label{ex:exotic_5}
    One way to realize the above construction is to apply flexible surgeries in addition to subcritical surgeries as in case \eqref{surgery} of \Cref{ex:generalization}, outside the Giroux domain that is blown down in the construction of the torsion cobordism. As a concrete example, the exotic contact sphere $(S^5,\xi_{ot})$ in \cite{bowden2022tight} admits a torsion cobordism of type II to a flexibly fillable contact manifold by \cite[the proof of Theorem 5.1]{bowden2022tight}.
\end{example}

\subsubsection{Cobordisms to disconnected contact manifolds}
A contact manifold $Y$ is called cofillable if there exists a non-empty contact manifold $Y'$ such that $Y\sqcup Y'$ has a connected strong filling $W$. Another way to enforce finite algebraic planar torsion for the concave boundary is to consider cobordisms to a disjoint union of contact manifolds, with one of them not cofillable. By \cite[Theorem 1.2]{moreno2024rsft}, finite planarity--another RSFT invariant \cite[Definition 3.24]{moreno2024landscape},  implies non-cofillability.
\begin{remark}
    It is not known to the author if there is a contact manifold of dimension $\ge 5$ that is strongly fillable, not cofillable, and does not have finite planarity.
\end{remark}
In the following, we confine ourselves to non-cofillable contact manifolds with finite planarity, where the finite planarity is observed in a specific way, e.g., \cite[Proposition 4.5]{moreno2024rsft}. In this paper, we call such contact manifolds \emph{strongly non-cofillable} contact manifolds, and we recall the precise definition in \S \ref{s:torsion}. Examples of strongly non-cofillable contact manifolds include iterated planar open books \cite[Theorem B(5)]{moreno2024landscape}, $\partial(V\times \D)$ for a Liouville domain $V$ \cite{zbMATH07673358}, and more generally $\partial(V\times S_k)$ for $c^{\Q}_1(V)=0$ \cite[Theorem 6.6]{moreno2024landscape}\footnote{The condition $c_1^{\Q}=0$ simplifies the proof; it is expected to be unnecessary}, where $S_k$ is $S^2$ with $k$ disks removed.

\begin{definition}\label{def:torsion_III}
A torsion cobordism of type III is a \emph{connected} strong cobordism $W$ such that the convex boundary is disconnected, with one of the components strongly non-cofillable.   
\end{definition}
Eliashberg \cite{Eli96} showed that $(T^3,\xi_k)$ can be realized as the concave boundary of a torsion cobordism of type III if $k\ge 2$, where $\xi_k$ is the contact structure obtained from the $k$-fold cover of $(T^3,\xi_1)=(S^*T^2,\xi_{\std})$ via the $k$-fold cover of the cotangent circles. This was used by Eliashberg to prove that $(T^3,\xi_k)$ is not strongly fillable if $k\ge 2$. In general dimensions, it is not easy to construct a connected cobordism with a disconnected convex boundary. One construction uses transverse surgery with a Liouville domain that has disconnected boundary.
\begin{example}
Here, we recall some constructions of connected Liouville domains with a disconnected contact boundary.
\begin{enumerate}
    \item In dimension $2$, any connected Riemann surface with more than two boundary components is an example.
    \item In general, any manifold $M$ that admits a Liouville pair in the sense of \cite[Definition 1]{MNW} can be used to construct a Liouville structure on $M\times [0,1]$. Constructions of this type include $D^*S^1$, McDuff's first example in dimension $4$ \cite{McDuff91}, and examples by Geiges \cite{zbMATH00663659,zbMATH00779041} and Mitsumatsu \cite{zbMATH00842476} in low dimensions, and by Massot--Niederkr\"uger--Wendl \cite{MNW} in general dimensions. 
\end{enumerate}
\end{example}

\begin{example}
Let $V$ be a connected Liouville domain with disconnected boundary $Y_1\sqcup Y_2$. Assume that $Y_1\sqcup Y_3$ (but $Y_3$ can be $\emptyset$ here) has a connected Liouville filling $W$, and that $Y$ is a contact manifold containing $Y_2$ as a codimension $2$ contact submanifold with trivial normal bundle. Then $(\partial(W\times \D)\sqcup Y)_{Y_1\sqcup Y_2,V,\Id}$ arises as the concave boundary of a torsion cobordism of type III. 
\end{example}

\begin{example}
    In dimension $3$, the only transverse submanifolds are $S^1$, which is the disconnected boundary of a Liouville surface. This then yields many constructions of torsion cobordisms of type III. For example,  let $K_1$ be a transverse knot in $(S^3,\xi_{\std})$ and $K_2$ a transverse knot in $(Y,\xi)$. Then $(S^3\sqcup Y)_{S^1\sqcup S^1,D^*S^1,\Id}$ arises as the concave boundary of a torsion cobordism of type III. Clearly, there are many generalizations of this construction in dimension $3$.
\end{example}

\begin{example}
    Using McDuff's Liouville filling for $S^*\Sigma_g\cup \partial \cO(2-2g)$ for $g\ge 2$, we can produce a torsion cobordism of type III in dimension $5$ by finding a Legendrian genus $g$ surface $\Sigma_g\subset Y_1$, which produces a contact $S^*\Sigma_g$ in $Y_1$ by contact push-off, and then by finding a contact submanifold $\partial \cO(2-2g)$ with a trivial normal bundle in $Y_2$.
\end{example}
\section{Rational symplectic field theory}\label{s:SFT}
\subsection{Algebraic preliminaries}\label{ss:algebra}
We begin by recalling the main algebraic notions from \cite{moreno2024landscape,moreno2024rsft} that are used to formulate rational symplectic field theory (RSFT). 
\subsubsection{$BL_\infty$ Algebras}\label{ss:BL}
Let $V$ be a $\Z/2$-graded vector space over $\Q$. Consider the $\Z/2$-graded symmetric algebra $S V:=\bigoplus_{k\ge 0} S^k V$ and the non-unital symmetric algebra $\overline{S} V =\bigoplus_{k\ge 1} S^k V$, where $S^kV=\otimes^k V/\Sigma_k$ (with $\Sigma_k$ the symmetric group on $k$ letters) in the graded sense. Let $EV=\overline{S}SV$. We use $\odot$ for the product on the outer symmetric product $\overline{S}$ and $\ast$ (often omitted) for the product on the inner symmetric product $S$. Given linear operators $p^{k,l}:S^kV \to S^l V$ for $k\ge 1, l\ge 0$, we define a map $\widehat{p}^{k,l}:S^k SV \to SV$ by the following properties.
\begin{enumerate}
	\item $\widehat{p}^{k,l}|_{\odot^k V\subset S^kSV}$ is defined by $p^{k,l}$.
	\item If $w_i\in \Q$, then $\widehat{p}^{k,l}(w_1\odot \ldots \odot w_k)=0$.
	\item $\widehat{p}^{k,l}$ satisfies the Leibniz rule in each argument, i.e., 
        \begin{equation}\label{eqn:pkl}
            \widehat{p}^{k,l}(w_1\odot \ldots \odot w_k)=\sum_{j=1}^m(-1)^{\square} v_1 \ldots  v_{j-1} \widehat{p}^{k,l}(w_1\odot \ldots \odot v_j \odot \ldots \odot w_k) v_{j+1} \ldots  v_m,
        \end{equation}
	where $w_i=v_1 \ldots  v_m$ and
 \begin{equation}\label{eqn:sign}
\square = \sum_{s=1}^{i-1}|w_s|\cdot \sum_{s=1}^{j-1}|v_s|+\sum_{s=1}^{j-1}|v_s||p^{k,l}|+\sum_{s=i+1}^n |w_s|\cdot \sum_{s=j+1}^{m}|v_s|.
 \end{equation}
\end{enumerate}
Explicitly, $\widehat{p}^{k,l}$ is defined by
\begin{equation}\label{eq:p_hat_1}
    w_1\odot \ldots \odot w_k \mapsto \sum_{\substack{(i_1,\ldots,i_k) \\ 1\le i_j\le n_j}} (-1)^{\bigcirc} p^{k,l}(v^1_{i_1}\ldots v^k_{i_k}) \check{w}_1 \ldots \check{w}_k,
\end{equation}
where $w_j=v^j_1 \ldots  v^j_{n_j}$, $\check{w}_j=v^j_1 \ldots  \check{v}^j_{i_j} \ldots  v^j_{n_j}$, and $\bigcirc$ is determined by $w_1 \ldots  w_k=(-1)^\bigcirc v^1_{i_1} \ldots  v^k_{i_k} \check{w}_1 \ldots \check{w}_k$. We then define $\widehat{p}^k:S^k S V \to S V$ by $\bigoplus_{l\ge 0} \widehat{p}^{k,l}$, and $\widehat{p}:EV \to EV$ by
\begin{equation}\label{eqn:q}
    w_1\odot \ldots \odot w_n \mapsto \sum_{k=1}^n\sum_{\sigma \in Sh(k,n-k)}(-1)^{\diamond} \widehat{p}^k(w_{\sigma(1)}\odot \ldots \odot w_{\sigma(k)})\odot w_{\sigma(k+1)}\odot \ldots \odot w_{\sigma(n)},
\end{equation}
where $Sh(k,n-k)$ is the subset of permutations $\sigma$ such that $\sigma(1)<\ldots<\sigma(k)$ and $\sigma(k+1)<\ldots < \sigma(n)$, and
$$\diamond=\sum_{\substack {1\le i < j \le n\\ \sigma(i)>\sigma(j)}}|w_i||w_j|.$$

\begin{definition}[{\cite[Definition 2.3]{moreno2024landscape}}]\label{def:BL}
$(V,\{p^{k,l}\})$ is a $BL_\infty$ algebra if $\widehat{p}\circ \widehat{p}=0$ and $|\widehat{p}|=1$.
\end{definition}
This definition implicitly requires that $\widehat{p}$ is well-defined. To ensure this, we can, for example, impose that for any $v_1,\ldots,v_k\in V$, there are at most finitely many $l$ such that $p^{k,l}(v_1\ldots v_k)\ne 0$.

An economical way to explain the combinatorics of the operations is via the following graphical description, following \cite{moreno2024landscape,moreno2024rsft}. The main advantage of this graphical language is that it frees us from keeping track of signs and explicit components of compositions, as in \Cref{eqn:pkl}, which are governed by graphs.

Let $w\in S^kV$. We can represent $w$ by an element $\overline{w}$ in $\otimes^k V$, i.e., $\overline{w}=\sum_{i=1}^N c_iv^i_1\otimes \ldots \otimes v^i_k$ for $c_i\in \Q$ and $v_*^*\in V$, such that $\pi(\overline{w})=w$ under $\pi:\otimes^kV \to S^kV$. We represent it by a rooted tree with $k$ leaves (represented by $\bullet$) labeled by $\overline{w}$. The leaves are ordered from left to right to indicate the $k$ copies of $V$ in $\otimes^k V$. When $\overline{\omega} = v_1\otimes \ldots \otimes v_k$, we may label the leaves by $v_1,\ldots,v_k$ to mean the same thing. A general labeled tree is viewed as a formal linear combination of such trees with labeled leaves.
    	
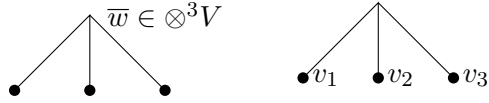
\begin{center}
        \begin{tikzpicture}
        \node at (0,0) [circle,fill,inner sep=1.5pt] {};
		\node at (1,0) [circle,fill,inner sep=1.5pt] {};
		\node at (2,0) [circle,fill,inner sep=1.5pt] {};
		\draw (0,0) to (1,1) to (1,0);
		\draw (1,1) to (2,0);
        \node at (2,1) {$\overline{w}\in \otimes^3 V$};
        \end{tikzpicture}
        \qquad
		\begin{tikzpicture}
		\node at (0,0) [circle,fill,inner sep=1.5pt] {};
        \node at (0.3,0) {$v_1$};
		\node at (1,0) [circle,fill,inner sep=1.5pt] {};
        \node at (1.3,0) {$v_2$};
		\node at (2,0) [circle,fill,inner sep=1.5pt] {};
        \node at (2.3,0) {$v_3$};
		\draw (0,0) to (1,1) to (1,0);
		\draw (1,1) to (2,0);
		\end{tikzpicture}
	\end{center}
Now let $s\in S^kSV$. We can represent $s$ by $\overline{s}\in \boxtimes^k TV$, where $TV=\oplus_{k\in \N}(\otimes^k V)$. Here we use $\boxtimes$ to differentiate it from the inner tensor $\otimes$. We write
$$\overline{s}=\sum_{i=1}^N c_i\overline{w}^i_1 \boxtimes \ldots  \boxtimes \overline{w}^i_k, \quad c_i\in \Q, \overline{w}_*^*\in \otimes^{m^*_*}V.$$
We represent $\overline{w}^i_1 \boxtimes\ldots  \boxtimes \overline{w}^i_k$ by an ordered forest of labeled trees as follows. Then $\overline{s}$ is a formal linear combination of such forests.
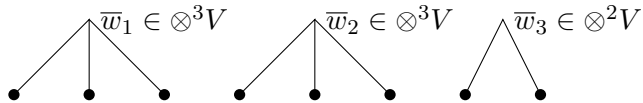
\begin{figure}[H]
    \begin{center}
		\begin{tikzpicture}
		\node at (0,0) [circle,fill,inner sep=1.5pt] {};
		\node at (1,0) [circle,fill,inner sep=1.5pt] {};
		\node at (2,0) [circle,fill,inner sep=1.5pt] {};
		\node at (3,0) [circle,fill,inner sep=1.5pt] {};
		\node at (4,0) [circle,fill,inner sep=1.5pt] {};
		\node at (5,0) [circle,fill,inner sep=1.5pt] {};
		\node at (6,0) [circle,fill,inner sep=1.5pt] {};
		\node at (7,0) [circle,fill,inner sep=1.5pt] {};
		
		\draw (0,0) to (1,1) to (1,0);
		\draw (1,1) to (2,0);
		\draw (3,0) to (4,1) to (4,0);
		\draw (4,1) to (5,0);
		\draw (6,0) to (6.5,1) to (7,0);

       \node at (2,1) {$\overline{w}_1\in \otimes^3V$};
       \node at (5,1) {$\overline{w}_2\in \otimes^3V$};
       \node at (7.5, 1) {$\overline{w}_3\in \otimes^2V$};
	\end{tikzpicture}
	\end{center}
    \caption{A forest of labeled trees}
    \label{fig:forest}
\end{figure}
We represent the operation $p^{k,l}:S^kV\to S^lV$ by a graph with $k+l+1$ vertices: $k$ top input vertices, $l$ bottom output vertices, and one middle vertex $\tikz\draw[black,fill=white] (0,-1) circle (0.4em);$ labeled by $p^{k,l}$ representing the operation type.
    
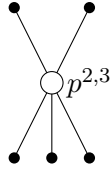
\begin{center}
        \begin{tikzpicture}
        \node at (2,0) [circle,fill,inner sep=1.5pt] {};
		\node at (3,0) [circle,fill,inner sep=1.5pt] {};
        \draw (2,0) to (2.5,-1) to (3,0);
		\draw (2,-2) to (2.5,-1) to (2.5,-2);
		\draw (2.5,-1) to (3,-2);
		\node at (2.5,-1) [circle, fill=white, draw, outer sep=0pt, inner sep=3 pt] {};
        \node at (2,-2) [circle,fill,inner sep=1.5pt] {};
	    \node at (3,-2) [circle,fill,inner sep=1.5pt] {};
	    \node at (2.5,-2) [circle,fill,inner sep=1.5pt] {};
        \node at (3,-1) {$p^{2,3}$};
        \end{tikzpicture}
    \end{center}
    
So far the discussion is completely formal without any actual content. The real content lies in the interpretation of a glued graph, whose definition will be clear from an example.
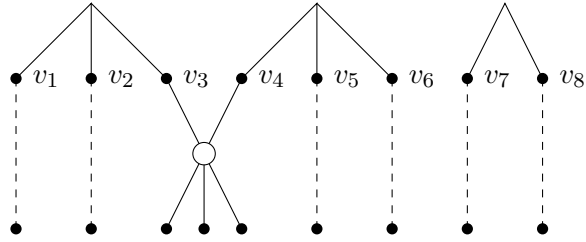
\begin{figure}[H]
	\begin{center}
		\begin{tikzpicture}
		\node at (0,0) [circle,fill,inner sep=1.5pt] {};
		\node at (1,0) [circle,fill,inner sep=1.5pt] {};
		\node at (2,0) [circle,fill,inner sep=1.5pt] {};
		\node at (3,0) [circle,fill,inner sep=1.5pt] {};
		\node at (4,0) [circle,fill,inner sep=1.5pt] {};
		\node at (5,0) [circle,fill,inner sep=1.5pt] {};
		\node at (6,0) [circle,fill,inner sep=1.5pt] {};
		\node at (7,0) [circle,fill,inner sep=1.5pt] {};
		
		\draw (0,0) to (1,1) to (1,0);
		\draw (1,1) to (2,0);
		\draw (3,0) to (4,1) to (4,0);
		\draw (4,1) to (5,0);
		\draw (6,0) to (6.5,1) to (7,0);

		\draw (2,0) to (2.5,-1) to (3,0);
		\draw (2,-2) to (2.5,-1) to (2.5,-2);
		\draw (2.5,-1) to (3,-2);
		\node at (2.5,-1) [circle, fill=white, draw, outer sep=0pt, inner sep=3 pt] {};
		
		\node at (0,-2) [circle,fill,inner sep=1.5pt] {};
		\node at (1,-2) [circle,fill,inner sep=1.5pt] {};
		\node at (2,-2) [circle,fill,inner sep=1.5pt] {};
		\node at (3,-2) [circle,fill,inner sep=1.5pt] {};
		\node at (4,-2) [circle,fill,inner sep=1.5pt] {};
		\node at (5,-2) [circle,fill,inner sep=1.5pt] {};
		\node at (6,-2) [circle,fill,inner sep=1.5pt] {};
		\node at (7,-2) [circle,fill,inner sep=1.5pt] {};
		\node at (2.5,-2) [circle,fill,inner sep=1.5pt] {};
		
		\draw[dashed] (0,0) to (0,-2);
		\draw[dashed] (1,0) to (1,-2);
		\draw[dashed] (4,0) to (4,-2);
		\draw[dashed] (5,0) to (5,-2);
		\draw[dashed] (6,0) to (6,-2);
		\draw[dashed] (7,0) to (7,-2);
		\node at (0.4,0) {$v_1$};
		\node at (1.4,0) {$v_2$};
		\node at (2.4,0) {$v_3$};
		\node at (3.4,0) {$v_4$};
		\node at (4.4,0) {$v_5$};
		\node at (5.4,0) {$v_6$};
		\node at (6.4,0) {$v_7$};
		\node at (7.4,0) {$v_8$};
		\end{tikzpicture}
	\end{center}
    \caption{Gluing forests $\Leftrightarrow$ applying operations}
    \label{fig:gluing}
\end{figure}
The above glued graph represents a forest: we first fix a representative $\overline{p}$ of $p^{2,3}(v_3v_4)$ in $\otimes^3V$. The glued forest in Figure \ref{fig:forest} represents $\pm (v_1\otimes v_2 \otimes \overline{p}\otimes v_5\otimes v_6)\boxtimes (v_7\otimes v_8)$. In the gluing, we do not create cycles in the glued graph (which in the context of SFT translates to the fact that the underlying curves of those SFT buildings, viewed as nodal curves, have arithmetic genus zero). Each dashed line represents the identity map, and each connected component represents a tree in the output. Drawing the input element as a forest of ordered trees with ordered leaves corresponds to choosing representatives from the tensor product, not the symmetric product. Finally, when we draw the glued graph on a plane as above (i.e., choosing an order of the trees and leaves, hence edges may cross), it determines a representative in the tensor product, but we view different orders as equivalent up to the obvious sign change. For example, the following is an equivalent gluing to Figure \ref{fig:gluing}, but with an extra sign when viewing it in the tensor product.
	
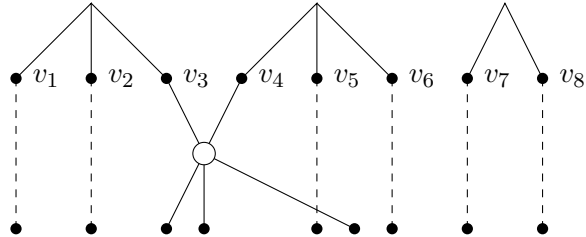
\begin{center}
		\begin{tikzpicture}
		\node at (0,0) [circle,fill,inner sep=1.5pt] {};
		\node at (1,0) [circle,fill,inner sep=1.5pt] {};
		\node at (2,0) [circle,fill,inner sep=1.5pt] {};
		\node at (3,0) [circle,fill,inner sep=1.5pt] {};
		\node at (4,0) [circle,fill,inner sep=1.5pt] {};
		\node at (5,0) [circle,fill,inner sep=1.5pt] {};
		\node at (6,0) [circle,fill,inner sep=1.5pt] {};
		\node at (7,0) [circle,fill,inner sep=1.5pt] {};
		
		\draw (0,0) to (1,1) to (1,0);
		\draw (1,1) to (2,0);
		\draw (3,0) to (4,1) to (4,0);
		\draw (4,1) to (5,0);
		\draw (6,0) to (6.5,1) to (7,0);

		\draw (2,0) to (2.5,-1) to (3,0);
		\draw (2,-2) to (2.5,-1) to (2.5,-2);
		\draw (2.5,-1) to (4.5,-2);
		\node at (2.5,-1) [circle, fill=white, draw, outer sep=0pt, inner sep=3 pt] {};
		
		\node at (0,-2) [circle,fill,inner sep=1.5pt] {};
		\node at (1,-2) [circle,fill,inner sep=1.5pt] {};
		\node at (2,-2) [circle,fill,inner sep=1.5pt] {};
		\node at (4.5,-2) [circle,fill,inner sep=1.5pt] {};
		\node at (4,-2) [circle,fill,inner sep=1.5pt] {};
		\node at (5,-2) [circle,fill,inner sep=1.5pt] {};
		\node at (6,-2) [circle,fill,inner sep=1.5pt] {};
		\node at (7,-2) [circle,fill,inner sep=1.5pt] {};
		\node at (2.5,-2) [circle,fill,inner sep=1.5pt] {};
		
		\draw[dashed] (0,0) to (0,-2);
		\draw[dashed] (1,0) to (1,-2);
		\draw[dashed] (4,0) to (4,-2);
		\draw[dashed] (5,0) to (5,-2);
		\draw[dashed] (6,0) to (6,-2);
		\draw[dashed] (7,0) to (7,-2);
		\node at (0.4,0) {$v_1$};
		\node at (1.4,0) {$v_2$};
		\node at (2.4,0) {$v_3$};
		\node at (3.4,0) {$v_4$};
		\node at (4.4,0) {$v_5$};
		\node at (5.4,0) {$v_6$};
		\node at (6.4,0) {$v_7$};
		\node at (7.4,0) {$v_8$};
		\end{tikzpicture}
	\end{center}
The sign is determined similarly to \Cref{eqn:sign}; in the case of Figure \ref{fig:gluing}, the sign is $(-1)^{(|v_1|+|v_2|)|p^{2,3}|}$. 

Writing the forest using a glued graph, as in Figure \ref{fig:gluing}, contains slightly more refined information than just labeling the forest as in Figure \ref{fig:forest}, namely, we keep track of which leaves come from $p^{k,l}$ in a representative. 
To enumerate all admissible gluings, each output leaf and tree is considered as different. However, we do not differentiate the input leaves of $p^{k,l}$. Therefore, when we glue a $p^{k,l}$ component, we pick $k$ trees (this is $Sh(k,n-k)$ in \Cref{eqn:q}) from the forest and then one leaf from each chosen tree (that is $1\le i_j\le n_j$ in \Cref{eq:p_hat_1}) to glue to $p^{k,l}$. For example, in the situation of Figure \ref{fig:gluing}, we have $3*3+3*2+3*2=21$ direct ways to glue $p^{2,3}$. The ambiguity from choosing a representative of the input is then eliminated by summing over all possible gluings.
In the following, we will use the graphical description for morphisms, augmentations, and Maurer-Cartan elements without giving explicit formulae like \Cref{eqn:pkl,eqn:q}. The explicit formulae, as well as more details of tree calculus, can be found in \cite[\S 2]{moreno2024landscape} and \cite[\S 2]{moreno2024rsft}.

\subsubsection{Morphisms} We now recall morphisms between $BL_{\infty}$ algebras. Consider a family of operators $\{\phi^{k,l}:S^k V\to S^l V'\}_{k\ge 1,l\ge 0}$; we can construct a map $\widehat{\phi}:EV\to EV'$ from the following tree description. To represent $\phi^{k,l}$, we use a graph similar to the one representing $p^{k,l}$ but replace $\tikz\draw[black,fill=white] (0,-1) circle (0.4em);$ by $\tikz\draw[black,fill=black] (0,-1) circle (0.4em);$ to indicate that they are maps of different roles. To represent a component configuration in the definition of $\widehat{\phi}$ on $S^{i_1}V\odot\ldots \odot S^{i_n}V$, we glue a family of graphs representing $\phi^{k,l}$ such that the input vertices and the output vertices of the top forest are completely paired and glued and the resulting graph has no cycles. Then $\widehat{\phi}$ is the sum of all possible configurations. Unlike the definition of $\widehat{p}$, where we need to glue exactly one $p^{k,l}$ graph, it is possible that we do not glue in any $\phi^{k,l}$ graphs. This is the case when the input is in $\odot^m\Q$ and $\widehat{\phi}$ is the identity in such a case, i.e., $\widehat{\phi}(1\odot \ldots \odot 1) = 1\odot \ldots \odot 1$.
\begin{figure}[H]
	\begin{center}
		\begin{tikzpicture}
		\node at (0,0) [circle,fill,inner sep=1.5pt] {};
		\node at (1,0) [circle,fill,inner sep=1.5pt] {};
		\node at (2,0) [circle,fill,inner sep=1.5pt] {};
		\node at (3,0) [circle,fill,inner sep=1.5pt] {};
		\node at (4,0) [circle,fill,inner sep=1.5pt] {};
		\node at (5,0) [circle,fill,inner sep=1.5pt] {};
		\node at (6,0) [circle,fill,inner sep=1.5pt] {};
		\node at (7,0) [circle,fill,inner sep=1.5pt] {};
		
		\draw (0,0) to (1,1) to (1,0);
		\draw (1,1) to (2,0);
		\draw (3,0) to (4,1) to (4,0);
		\draw (4,1) to (5,0);
		\draw (6,0) to (6.5,1) to (7,0);

		\draw (2,0) to (2.5,-1) to (3,0);
		\draw (2,-2) to (2.5,-1) to (2.5,-2);
		\draw (2.5,-1) to (3,-2);
		\node at (2.5,-1) [circle, fill, draw, outer sep=0pt, inner sep=3 pt] {};
		
		\node at (0,-2) [circle,fill,inner sep=1.5pt] {};
		\node at (1,-2) [circle,fill,inner sep=1.5pt] {};
		\node at (2,-2) [circle,fill,inner sep=1.5pt] {};
		\node at (3,-2) [circle,fill,inner sep=1.5pt] {};
		\node at (4,-2) [circle,fill,inner sep=1.5pt] {};
		\node at (2.5,-2) [circle,fill,inner sep=1.5pt] {};
	
		\draw (0,0) to (0,-2);
		\draw (1,0) to (1,-2);
		\draw (4,0) to (4,-2);
		\draw (5,0) to (5.5,-1);
		\draw (5.5,-1) to (6,0);
		\draw (7,0) to (7,-1);
		
		\node at (0,-1) [circle, fill, draw, outer sep=0pt, inner sep=3 pt] {};
		\node at (1,-1) [circle, fill, draw, outer sep=0pt, inner sep=3 pt] {};
		\node at (4,-1) [circle, fill, draw, outer sep=0pt, inner sep=3 pt] {};
		\node at (5.5,-1) [circle, fill, draw, outer sep=0pt, inner sep=3 pt] {};
		\node at (7,-1) [circle, fill, draw, outer sep=0pt, inner sep=3 pt] {};
		\end{tikzpicture}
	\end{center}
	\caption{A component of $\widehat{\phi}$ from $S^3V \odot S^3 V \odot S^2 V$ to $S^6V^\prime$}
\end{figure}
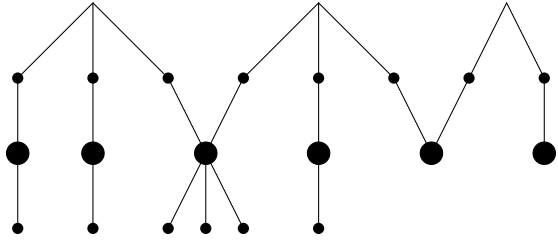

\begin{definition}[{\cite[Definition 2.10]{moreno2024landscape}}]\label{def:morphism}
 $\widehat \phi$ is a $BL_{\infty}$ morphism from $(V,\widehat p)$ to $(V',\widehat p')$ if $\widehat{\phi}\circ \widehat{p}=\widehat{p}'\circ \widehat{\phi}$ and $|\widehat{\phi}|=0$.
\end{definition}
This definition implicitly assumes that $\widehat{\phi}$ is well-defined. For example, when $\phi$ arises from counting holomorphic curves in an exact cobordism, we have for any $v_1 \ldots  v_k\in S^kV$ that there are at most finitely many $l$ such that $\phi^{k,l}(v_1 \ldots  v_k)\ne 0$. Hence $\widehat{\phi}$ is well-defined.

\subsubsection{Augmentations}
The trivial vector space $V=\{0\}$ has a unique trivial $BL_\infty$ algebra structure with $p^{k,l}=0$. We use $\mathbf{0}$ to denote this trivial $BL_\infty$ algebra.
\begin{definition}[{\cite[Definition 2.13]{moreno2024landscape}}]
A $BL_\infty$ augmentation is a $BL_\infty$ morphism $\epsilon:V \to \mathbf{0}$, i.e., a family of operators $\epsilon^k:S^kV \to \Q$ satisfying Definition \ref{def:morphism}.
\end{definition}

For a $BL_\infty$ algebra $V$, we let \begin{equation}\label{eq:B}
    E^kV=\overline{B}^kSV:=\bigoplus_{j=1}^k S^jSV,
\end{equation} which is a filtration (sentence length filtration) on $EV$ compatible with the differential $\widehat{p}$. We have $E\mathbf{0}=\Q \oplus S^2\Q \oplus +\ldots$ with $\widehat{p}=0$, and $H_*(E^k\mathbf{0})=E^k\mathbf{0}$ for all $k\ge 1$. We define $1_\mathbf{0}$ to be the generator in $E^1\mathbf{0}$; then $1_{\mathbf{0}}\ne 0\in H_*(E^k\mathbf{0})$ for all $k\ge 1$. We define $1_V\in H_*(E^kV)$ to be the image of $1_{\mathbf{0}}$ under the chain map $E^k\mathbf{0}\to E^kV$ induced by the trivial $BL_\infty$ morphism $\mathbf{0}\to V$. Note that the existence of an augmentation implies that $H_*(EV,\widehat{p})\neq 0$, as $\epsilon$ descends to homology and $\epsilon_*1_V=1_{\mathbf{0}}\ne 0\in H_*(E\mathbf{0},0)$.

\begin{definition}[{\cite[Definition 2.15]{moreno2024landscape}}]
We define the torsion of a $BL_\infty$ algebra $V$ to be
$$\mathrm{T}(V):= \min\{k-1\mid 1_V=0 \in H_*(E^kV),k\ge 1\}.$$
Here, the minimum of an empty set is defined to be $\infty$.
\end{definition}
By definition, we have that $\mathrm{T}(V)=0$ iff $1_V\in H^*(SV,\widehat{p}^1)$ is zero. Since $H^*(SV,\widehat{p}^1)$ is an algebra with $1_V$ a unit, in this case we have $H^*(SV,\widehat{p}^1)=0$. In general, a finite torsion is an obstruction to augmentations, and the hierarchy of torsion is functorial with respect to $BL_\infty$ morphisms, i.e., if there is a $BL_\infty$ morphism from $V$ to $V'$, then $\mathrm{T}(V)\ge \mathrm{T}(V')$, see \cite[\S 2.5]{moreno2024landscape} for more discussions. In the context of RSFT, this torsion invariant is the algebraic planar torsion $\APT(Y)$.

Given a $BL_\infty$ augmentation $\epsilon$, we use a change of coordinates on $EV$ to assemble a new $BL_\infty$ structure $\{p^{k,l}_{\epsilon}\}$ with the property that $p^{k,0}_\epsilon =0$ for all $k\ge 1$. We will call such a procedure the linearization w.r.t.\ $\epsilon$, \cite[Propositions 2.17, 2.18]{moreno2024landscape}. In particular, $p^{k,1}_{\epsilon}$ form an $L_\infty$ structure (after fixing suitable sign conventions; see \cite[\S 2.1, 2.2]{moreno2024landscape}) on $V$. There is a notion of pointed map on $BL_\infty$ algebras, to model the structure arising from counting rational holomorphic curves in the symplectization with a marked point passing through a closed submanifold in the contact manifold (or more generally a closed chain), with a typical case a point constraint, see \cite[Definition 2.19]{moreno2024landscape}. Along with the linearization w.r.t.\ all possible augmentations, we can derive another type of numerical invariants \cite[Definition 2.21]{moreno2024landscape}, which will turn into a functorial contact invariant, called planarity $\Pl(Y)$; see \cite[\S 3.5]{moreno2024landscape} for details. 

\subsubsection{Completions and Twisted Coefficients}
We define the Novikov field $\Lambda$ as
    $$\Lambda=\left\{ \sum_{i=1}^\infty a_i T^{\lambda_i}\left| a_i\in \Q, \displaystyle \lim_{i\to \infty} \lambda_i = +\infty \right.\right\}.$$
Let $\omega:G\to \R $ be a group homomorphism. The Novikov completion $\overline{\Q[G/\ker \omega]}$ of the group ring $\Q[G/\ker \omega]$ is
    $$\overline{\Q[G/\ker \omega ]}=\left\{ \sum_{i=1}^\infty a_i T^{g_i}\left| a_i\in \Q, g_i\in G/\ker \omega, \displaystyle \lim_{i\to \infty} \omega(g_i) = +\infty \right.\right\}.$$
These coefficient rings are equipped with a decreasing filtration and are complete with respect to the filtration.

Now let $V$ be a $\Q$-vector space with filtration degree $0$; then $V\otimes_{\Q} \Lambda$ and $V\otimes_{\Q} \overline{\Q[G/\ker \omega]}$ have induced decreasing filtrations, and so do the base changes for $SV$ and $EV$. We use $\overline{V\otimes_{\Q} \Lambda}$, $\overline{SV\otimes_{\Q} \Lambda}$, etc., to denote the corresponding completions. For example,
$$\overline{V\otimes_{\Q}\Lambda}=\left\{ \sum_{i=1}^{\infty} v_iT^{\lambda_i}\left|v_i\in V,\lambda_i\in \R, \lim_{i\to \infty} \lambda_i=+\infty \right. \right\}.$$
The filtration on $\overline{V\otimes_{\Q}\Lambda}$ is given by
$$(\overline{V\otimes_{\Q}\Lambda})_{\rho}=\left\{ \sum_{i=1}^{\infty} v_iT^{\lambda_i}\left|v_i\in V,\lambda_i\ge \rho, \lim_{i\to \infty} \lambda_i=+\infty \right. \right\},$$
so that
$$(\overline{V\otimes_{\Q}\Lambda})_{\rho} \subset (\overline{V\otimes_{\Q}\Lambda})_{\rho'}, \text{ when } \rho\ge \rho'.$$

\begin{definition}\label{def:complete}
A filtered completed $BL_\infty$ algebra structure on $V$ over $R=\Lambda,\overline{\Q[G/\ker \omega]}$ consists of $R$-linear maps $p^{k,l}:\overline{S^kV\otimes_{\Q} R} \to \overline{S^lV\otimes_{\Q} R}$
for $k\ge 1, l\ge 0$, such that the assembled map $\widehat{p}$ is well-defined on $\overline{EV\otimes_{\Q} R}$, has degree $1$, squares to zero, and preserves the filtration. The filtered completed versions for morphisms, augmentations, etc., are similar.
\end{definition}

In the context of SFT, a typical case for $G$ is the second homology of the contact manifold and $\omega$ comes from certain closed $2$-forms that are positive on the contact structure; see \S \ref{sss:twisted} for details.

\subsubsection{Maurer-Cartan Elements}\label{ss:MC_Alg}
\begin{definition}
    A Maurer-Cartan element of a filtered completed $BL_\infty$ algebra is an element $\mc$ of degree $0$ with a positive filtration degree in $\overline{SV\otimes_{\Q}R}$, such that $\widehat{p}(e^\mc-1)=0$, where $$e^{\mc}=\sum_{i=0}^{\infty} \frac{\odot^i \mc}{i!} \in \overline{SSV\otimes_{\Q}R}=\overline{R}\oplus \overline{EV\otimes_{\Q}R}.$$ It is clear that $e^\mc-1$ is the projection of $e^{\mc}$ to $\overline{EV\otimes_{\Q}R}$.
\end{definition}

We can represent an element $x$ of degree $0$ of $\overline{SV\otimes_{\Q}R}$ by a sum of labeled trees, i.e., $x=\sum T_k$ for $T_k\in \overline{S^kV\otimes_{\Q}R}$
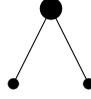
\begin{center}
\begin{tikzpicture}
		\node at (6,0) [circle,fill,inner sep=1.5pt] {};
		\node at (7,0) [circle,fill,inner sep=1.5pt] {};
		\draw (6,0) to (6.5,1) to (7,0);
  	\node at (6.5,1) [circle,fill,inner sep=3pt] {};
\end{tikzpicture}
\end{center}
Here we use a notation similar to a $BL_\infty$ morphism $\phi$, i.e., trees with $\tikz\draw[black,fill=black] (0,-1) circle (0.4em);$ representing the root, instead of trees with undotted roots in the description of elements in $SV$, since Maurer-Cartan elements behave like a ``morphism" from the trivial $BL_\infty$ algebra. We can represent the constant term of $\mathfrak{mc}$ by a tree without leaves, i.e., a single root. With such notation, $e^x-1$ in $\overline{EV\otimes_{\Q}R}$ is represented by linear combinations of forests generated by the labeled trees. As each component of $\mc$ has degree $0$, forests up to switching order are equivalent.
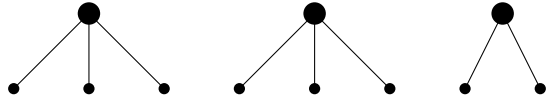
\begin{center}
\begin{tikzpicture}
		\node at (0,0) [circle,fill,inner sep=1.5pt] {};
		\node at (1,0) [circle,fill,inner sep=1.5pt] {};
		\node at (2,0) [circle,fill,inner sep=1.5pt] {};
		\node at (3,0) [circle,fill,inner sep=1.5pt] {};
		\node at (4,0) [circle,fill,inner sep=1.5pt] {};
		\node at (5,0) [circle,fill,inner sep=1.5pt] {};
		\node at (6,0) [circle,fill,inner sep=1.5pt] {};
		\node at (7,0) [circle,fill,inner sep=1.5pt] {};
    	\node at (6.5,1) [circle,fill,inner sep=3pt] {};
       \node at (1,1) [circle,fill,inner sep=3pt] {};
        \node at (4,1) [circle,fill,inner sep=3pt] {};
		
		\draw (0,0) to (1,1) to (1,0);
		\draw (1,1) to (2,0);
		\draw (3,0) to (4,1) to (4,0);
		\draw (4,1) to (5,0);
		\draw (6,0) to (6.5,1) to (7,0);
\end{tikzpicture}
\end{center}
Here the coefficient of a forest $\{\underbrace{T_{a_1},\ldots,T_{a_1}}_{i_1}, \ldots, \underbrace{T_{a_m},\ldots,T_{a_m}}_{i_m}\}$ is given by the ``isotropy order" $\frac{1}{(i_1)!\ldots (i_m)!}$.

Given $a\in \overline{SV\otimes_{\Q}R}$ of degree $0$ with a positive filtration degree, we define
$$\exp_a:\overline{EV\otimes_{\Q}R}\to \overline{EV\otimes_{\Q}R}, \qquad x\mapsto x\odot e^a=x+x\odot a +\frac{x\odot a \odot a}{2}+\ldots,$$
which preserves the filtration. The inverse of $\exp_a$ is given by $\exp_{-a}$.

Given a Maurer-Cartan element $\mathfrak{mc}$, we have a deformed $BL_\infty$ structure $p^{k,l}_{\mathfrak{mc}}$ defined by trees from stacking the forests representing $e^{\mathfrak{mc}}-1$ over $\widehat{p}$, namely:
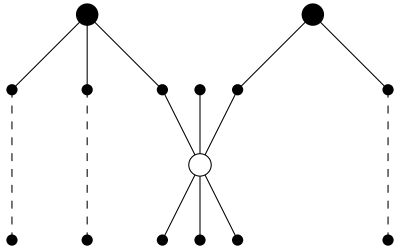
\begin{figure}[H]
\begin{center}
		\begin{tikzpicture}
		\node at (0,0) [circle,fill,inner sep=1.5pt] {};
		\node at (1,0) [circle,fill,inner sep=1.5pt] {};
		\node at (2,0) [circle,fill,inner sep=1.5pt] {};
		\node at (3,0) [circle,fill,inner sep=1.5pt] {};
		\node at (5,0) [circle,fill,inner sep=1.5pt] {};
       \node at (1,1) [circle,fill,inner sep=3pt] {};
        \node at (4,1) [circle,fill,inner sep=3pt] {};

		\draw (0,0) to (1,1) to (1,0);
		\draw (1,1) to (2,0);
		\draw (3,0) to (4,1);
		\draw (4,1) to (5,0);
        \draw (2.5,0) to (2.5,-1);
  
		\draw (2,0) to (2.5,-1) to (3,0);
		\draw (2,-2) to (2.5,-1) to (2.5,-2);
		\draw (2.5,-1) to (3,-2);
		\node at (2.5,-1) [circle, fill=white, draw, outer sep=0pt, inner sep=3 pt] {};
		
		\node at (0,-2) [circle,fill,inner sep=1.5pt] {};
	    \node at (1,-2) [circle,fill,inner sep=1.5pt] {};
	    \node at (2,-2) [circle,fill,inner sep=1.5pt] {};
	    \node at (3,-2) [circle,fill,inner sep=1.5pt] {};
	    \node at (5,-2) [circle,fill,inner sep=1.5pt] {};
	    \node at (2.5,-2) [circle,fill,inner sep=1.5pt] {};
	    \node at (2.5,0) [circle,fill,inner sep=1.5pt] {};
	    \draw[dashed] (0,0) to (0,-2);
		\draw[dashed] (1,0) to (1,-2);
		\draw[dashed] (5,0) to (5,-2);

		\end{tikzpicture}
\end{center}
    \caption{A component of $p^{1,6}_{\mathfrak{mc}}$}\label{fig:pmc}
\end{figure}
In a more formal way, we have
\begin{equation}\label{eqn:pmc}
    p_{\mathfrak{mc}}^{k,l}(v_1\ldots v_k) = \pi_{1,l}\circ\widehat{p}(v_1\odot \ldots \odot v_k\odot e^{\mathfrak{mc}})
\end{equation}
where $\pi_{1,l}$ is the completed version of the projection from $EV$ to $S^lV \subset SV \subset EV$. Since $\mathfrak{mc}$ has degree $0$, it does not matter where to put the Maurer-Cartan element in \Cref{fig:pmc}. The same construction holds for any $a\in \overline{SV\otimes_{\Q}R}$ of degree $0$, and we have the following.

\begin{lemma}[{\cite[Lemma 2.20]{moreno2024rsft}}]\label{lem:pmc}
For $s\in \overline{EV\otimes_{\Q}R}$ and $a\in \overline{SV\otimes_{\Q}R}$ such that $\deg(a)=0$, we have $$\widehat{p}(s\odot e^a)=\widehat{p}_a(s)\odot e^a+(-1)^{|s|} s\odot \widehat{p}(e^a-1),$$
assuming that $s$ is of pure degree.
\end{lemma}

Then $p^{k,l}_{\mc}$ is the structural map under the change of coordinates by $\exp_{\mc}$. More precisely, we have the following.
\begin{proposition}[{\cite[Proposition 2.21]{moreno2024rsft}}]
    For a Maurer-Cartan element $\mc$ and a filtered completed $BL_\infty$ algebra $p^{k,l}$, we have that $p^{k,l}_{\mc}$ forms a filtered completed $BL_\infty$ algebra. Moreover, the deformed structure does not depend on the constant term of $\mc$.
\end{proposition}

Another immediate corollary of \Cref{lem:pmc} is the following.
\begin{proposition}\label{prop:chain_mc}
Let $\mathfrak{mc}$ be a Maurer-Cartan element; then $\exp_{\mc}$ defines a chain map from $(\overline{EV\otimes_{\Q}R}, \widehat{p}_{\mathfrak{mc}})$ to $(\overline{EV\otimes_{\Q}R}, \widehat{p})$.
 \end{proposition}
However, since $\exp_{\mc}$ does not preserve the sentence length filtration $\overline{E^kV\otimes_{\Q}R}\subset \overline{EV\otimes_{\Q}R}$, we cannot guarantee that the torsion of $p$ is smaller than that of $p_{\mathfrak{mc}}$. There are similar deformations to the pointed maps by Maurer-Cartan elements; see \cite[\S 2.7]{moreno2024rsft} for details.

\subsubsection{$\Q$-grading}
The above discussion is carried out under the $\Z/2$-grading assumption for $V$ and hence all the variants derived from it. In practice, sometimes it is convenient to work with $\Z$-grading or, more generally, $\Q$-grading. 
\begin{definition}\label{def:Q-grading}
    Let $(V,\{p^{k,l}\})$ be a $BL_\infty$ algebra. We say it is $\Q$-graded of dimension $2n$ if 
    \begin{enumerate}
        \item $V$ is also $\Q$-graded.
        \item $p^{k,l}$ in $\Q$-grading has degree $-2(n-3)(k-1)-1$.
    \end{enumerate}
    A morphism $\phi$ between $\Q$-graded $BL_\infty$ algebras is a $BL_\infty$ morphism such that $\phi^{k,l}$ in $\Q$-grading has degree $-2(n-3)(k-1)$. A Maurer-Cartan element $\mc$ is $\Q$-graded if it is contained in $\overline{(SV)_{2(n-3)}\otimes_{\Q} R}$, where $(SV)_{2(n-3)}$ is the grading $2(n-3)$ piece of $SV$. 
\end{definition}
We do not require that the $\Z/2$-grading can be deduced from the $\Q$-grading in some way. The $\Z/2$-grading is used to determine all the signs, while $\Q$-grading is just an additional piece of data. 

\subsection{Rational symplectic field theory}
\subsubsection{Punctured holomorphic curves in completions}
Let $(Y,\xi)$ be a co-oriented contact manifold and $\alpha$ a contact form such that all Reeb orbits are non-degenerate. We fix an $\R$-translation invariant almost complex structure $J$ on the symplectization $\widehat{Y}=\R_t \times Y$, such that $J\partial_t=R$ (the Reeb vector field) and $J$ preserves $\xi$ and is $\rd \alpha$-compatible on $\xi$. Let $\Gamma_+,\Gamma_-$ be two ordered multisets (i.e., allowing duplicates of elements) of good Reeb orbits of $\alpha$. $\cM_{Y,A}(\Gamma_+,\Gamma_-)$ is defined to be the moduli space of rational $J$-holomorphic curves in homology class $A\in H_2(Y,\Gamma_+\cup\Gamma_-)$ with $|\Gamma_+|$ ordered positive punctures and $|\Gamma_-|$ ordered negative punctures, all equipped with asymptotic markers, such that the curve converges to orbits in $\Gamma_+/\Gamma_-$ near the positive/negative punctures and converges to pre-fixed base points on the image of Reeb orbits along the direction of asymptotic markers, modulo reparameterization and $\R$-translation; see \cite[\S 3.2]{moreno2024landscape} for details. The following are compactifications of the moduli spaces above, as well as variants of them; all of them can be found in \cite{moreno2024landscape,moreno2024rsft}. 
\begin{enumerate}
    \item We use $\overline{\cM}_{Y,A}(\Gamma_+,\Gamma_-)$ to denote the compactification of $\cM_{Y,A}(\Gamma_+,\Gamma_-)$; for formal discussions, one can either use the original SFT compactification in \cite{zbMATH02062477} or the reduced version by throwing out trivial cylinders used in \cite{zbMATH07085531,zbMATH07656377,moreno2024landscape}. However, the definition of virtual count differs in these two settings, and it is currently carried out only using the reduced compactification. We could have $\Gamma_-=\emptyset$, but $\Gamma_+$ has to be non-empty to have a non-empty moduli space.
    \item Let $W$ be a cobordism, not necessarily Liouville; we use  $\overline{\cM}_{W,A}(\Gamma_+,\Gamma_-)$ to denote the compactified moduli space of rational curves of class $A$ with asymptotics $\Gamma_+,\Gamma_-$ in the completion $\widehat{W}$, i.e., similar to $\overline{\cM}_{Y,A}(\Gamma_+,\Gamma_-)$ except that there is no $\R$-translation symmetry anymore. When $W$ is not exact, we can have $\Gamma_-=\emptyset$ as well. 
    \item Let $o$ be a point in $Y$; we use $\overline{\cM}_{Y,A}(\Gamma_+,\Gamma_-,o)$ to denote the compactified moduli space of rational curves of class $A$ in $\widehat{Y}$, with asymptotics $\Gamma_+,\Gamma_-$ and a free marked point going through $\{(0,o)\}\in \widehat{Y}$. More generally, let $\mu:M\to Y$ be a map from an oriented closed manifold $M$; we use $\overline{\cM}_{Y,A}(\Gamma_+,\Gamma_-,\mu)$ to denote the compactified moduli space of rational curves of class $A$ in $\widehat{Y}$, with asymptotics $\Gamma_+,\Gamma_-$ and a free marked point going through $\{(0,\Ima \mu)\}\subset \widehat{Y}$, i.e., a fiber product of the moduli space of curves with a marked point and $\mu:M\to \{0\}\times Y\subset \widehat{Y}$. Let $W$ be a cobordism and $\mu:(N,\partial N)\to (W,\partial W)$ be a smooth map from an oriented manifold $N$ possibly with boundary; we use $\overline{\cM}_{W,A}(\Gamma_+,\Gamma_-,\mu)$ to denote the compactified moduli space of rational curves of class $A$ in $\widehat{W}$, with asymptotics $\Gamma_+,\Gamma_-$ and a free marked point going through $\Ima \widehat{\mu}$, where $\widehat{\mu}$ is the completion $\widehat{N}\to \widehat{W}$ by adding an infinite collar of $\partial N$ to $N$.
    \item We use $\overline{\cM}_Y,\overline{\cM}_W$ to denote the union of $\overline{\cM}_{Y,A}$ or $\overline{\cM}_{W,A}$ for all $A$.
\end{enumerate}
The virtual dimension of $\overline{\cM}_{Y,A}(\Gamma^+,\Gamma^-)$ is given by
$$(n-3)(2-s^+-s^-)+\sum_{i=1}^{s^+}\mu_{\CZ}(\gamma_i^+) - \sum_{i=1}^{s^-}\mu_{\CZ}(\gamma_i^-)-1,$$
where the Conley-Zehnder index is computed using a trivialization of $u^*\det_{\C}\xi$ for $u\in \cM_{Y,A}(\Gamma_+,\Gamma_-)$. Alternatively, we fix a trivialization of $\det_{\C}\xi$ along $\Gamma_+\cup \Gamma_-$ to define Conley-Zehnder indices and pick up $2c_1(A)$ from the relative first Chern class defined using such boundary trivializations. The virtual dimension for other moduli spaces can then be derived from it. 

Let $V$ be the $\Q$ vector space generated by formal variables $q_\gamma$ for each good orbit $\gamma$ of $(Y,\alpha)$. We grade $q_\gamma$ by the $\Z/2$-SFT degree $|q_{\gamma}|:=\mu^{\Z/2}_{\CZ}(\gamma)+n-3\in \Z/2$. Now if $c^{\Q}_1(Y)=0$, we can trivialize $\det_{\C}\oplus^N \xi$ for some $N\in \N_+$; with this trivialization, we can obtain a $\Q$-valued Conley-Zehnder index $\mu^{\Q}_{\CZ}(\gamma)$ (depending on the trivialization) for each Reeb orbit $\gamma$, see \cite[\S 3.3.2]{arXiv:2108.12247}. Then the $\Q$-grading $|q_{\gamma}|^{\Q}$ is given by $\mu^{\Q}_{\CZ}(\gamma)+n-3$. If $c_1(Y)=0$, then we can trivialize $\det_{\C}\xi$, which determines a $\Z$-valued Conley-Zehnder index and hence a $\Z$-valued SFT degree, whose mod $2$ reduction is the $\Z/2$-SFT degree.

\subsubsection{RSFT and Contact homology}
We can define operations 
$$p^{k,l}:S^kV \to S^lV, \quad q^{\Gamma_+} \mapsto \sum_{[\Gamma_-],|\Gamma_-|=l} \frac{1}{\mu_{\Gamma_-}\kappa_{\Gamma_-}}\#\overline{\cM}_Y(\Gamma_+,\Gamma_-)q^{\Gamma_-}$$
for $k\ge 1$ and $l\ge 0$.
The sum is over all multisets $[\Gamma_-]$, i.e., sets with duplicates. And $\Gamma_-$ is an ordered representation of $[\Gamma_-]$, e.g., $$\Gamma_-=\{\underbrace{\gamma_1,\ldots,\gamma_1}_{i_1}, \ldots, \underbrace{\gamma_m,\ldots,\gamma_m}_{i_m}\}$$ is an ordered set of good orbits with $\gamma_i\ne \gamma_j$ for $i\ne j$. We write $\mu_{\Gamma_-}=i_1!\ldots i_m!$ and $\kappa_{\Gamma_-}=\kappa^{i_1}_{\gamma_1}\ldots \kappa^{i_m}_{\gamma_m}$ is the product of multiplicities of the Reeb orbits, and $q^{\Gamma_-}=q_{\gamma_1} \ldots  q_{\gamma_m}$. 

The virtual count $\#\overline{\cM}$ depends on auxiliary choices in the virtual setup, e.g., Pardon's VFC; see \cite[\S 3.6]{moreno2024landscape} for details. The combinatorics of the (reduced) SFT compactness implies that $\{p^{k,l}\}$ forms a $BL_\infty$ structure on $V$. Moreover, if $c_1^{\Q}(Y)=0$ and we trivialize $\det_{\C}\oplus^N \xi$ for some $N\in \N_+$, $V$ is also $\Q$-graded by the $\Q$-valued SFT degree and $p^{k,l}$ also respects the $\Q$-grading in the sense of \Cref{def:Q-grading}. 

Given an exact cobordism $(X,\lambda)$ from $Y_-$ to $Y_+$, let $V_{\pm}$ denote the $BL_\infty$ algebra associated to $Y_{\pm}$ after fixing some auxiliary data. We can define maps $\phi^{k,l}$ from $S^kV_+$ to $S^lV_-$ by
$$\phi^{k,l}(q^{\Gamma_+})= \sum_{[\Gamma_-],|\Gamma_-|=l} \frac{1}{\mu_{\Gamma_-}\kappa_{\Gamma_-}} \# \overline{\cM}_{X}(\Gamma_+,\Gamma_-)q^{\Gamma_-},$$
for $|\Gamma_+|=k$.
The boundary configuration of $\overline{\cM}_{X}(\Gamma_+,\Gamma_-)$ with virtual dimension $1$ shows that $\phi^{k,l}$ is a $BL_\infty$ morphism.

\begin{definition}[{\cite[Definition 3.12, Proposition 3.13]{moreno2024landscape}}]
    The algebraic planar torsion $\APT(Y)$ is the torsion $\mathrm{T}(V)$ of the $BL_\infty$ algebra $(V,\{p^{k,l}\})$ of the contact manifold $Y$ after fixing a contact form and some auxiliary data.
\end{definition}
Following \cite[\S 3.6]{moreno2024landscape}, Pardon's VFC provides enough foundation for $\APT(Y)$ to be a well-defined functor from the exact symplectic cobordism category to the ordered set $\N\cup \{\infty\}$.

From the $BL_\infty$ structure, we have that $(SV,\widehat{p}^1)$ is a differential graded algebra (dga), whose homology is the contact homology $\CH(Y)$. More specifically, the differential is defined as follows,
\begin{equation}\label{eqn:partial}
\partial_{\CH}(q_{\gamma}) = \sum_{[\Gamma]} \#\overline{\cM}_{Y}(\{\gamma \},\Gamma) \frac{1}{\mu_{\Gamma}\kappa_{\Gamma}}q^{\Gamma}.
\end{equation}
The restriction of a $BL^{\infty}$ morphism from an exact cobordism to $SV_+$ gives a dga morphism. On the homology level, those constructions define a (contravariant) functor from the (exact) symplectic cobordism category to the category of (super)commutative algebras. Such construction was established by Pardon \cite{zbMATH07085531} using the implicit atlas and VFC, and by Bao-Honda \cite{zbMATH07656377} using Kuranishi perturbation theory. 

\begin{proposition}\label{prop:APT}
    Here we summarize some of the properties of $\APT$ from \cite{moreno2024landscape,moreno2024rsft}.
    \begin{enumerate}
        \item $\APT(Y)=0$ if and only if $\CH(Y)=0$;
        \item If $W$ is an exact cobordism from $Y_+$ to $Y_-$, then $\APT(Y_+)\le \APT(Y_-)$;
        \item If $\APT(Y)<\infty$, then $Y$ has no strong filling;
        \item\label{APT4} If $W$ is a strong cobordism from $Y_+$ to $Y_-$ and $\APT(Y_+)<\infty$, then $\APT(Y_-)<\infty$, see \cite[Theorem 1.1]{moreno2024rsft}.
    \end{enumerate}
\end{proposition}

\subsubsection{Twisted coefficient and weak fillings}\label{sss:twisted}
Given a closed $2$-form $\omega$ on $Y$, we get a group homomorphism $H_2(Y;\R)\to \R$ by pairing with $\omega$. This allows us to define a completion of the group ring $\overline{\Q[H_2(Y;\R)/\ker \omega]}$. Following \cite{zbMATH01643843,LW}, after choosing a basis for $H_1(Y;\R)$, we can define the contact homology, as well as other versions of SFT, with coefficient in the group ring as well as the completed group ring, by tracking the homology class of the curves. We use $\CH_{\omega}$ to denote the homology of $\overline{SV\otimes \Q[H_2(Y;\R)/\ker \omega]}$. With a bit of abuse of terminology, we call this contact homology with $\overline{\Q[H_2(Y;\R)/\ker \omega]}$ coefficient. We use $\CH_{tw}$ to denote the homology of $SV\otimes \Q[H_2(Y;\R)]$ without the completion, called the fully twisted contact homology. It is clear that we have algebra maps $\CH_{tw}(Y)\to \CH_{\omega}(Y),\CH(Y)$, hence $\CH_{tw}(Y)=0$ implies that both $\CH(Y)=0$ and $\CH_{\omega}(Y)=0$ for any $\omega$.

Overtwistedness in contact manifolds translates via SFT to the vanishing of contact homology by the work of Bourgeois and van Koert \cite{zbMATH05709738}, which leads to the concept of algebraically overtwisted manifolds (i.e., those with vanishing contact homology) \cite{zbMATH05658836}. 
\begin{theorem}[\cite{zbMATH05709738,+1}]\label{thm:AO}
    If $Y$ is overtwisted, then $\CH_{tw}(Y)=0$.
\end{theorem}
\cite{zbMATH05709738,+1} studied the usual contact homology. But as the proofs start by showing $\CH(S^{2n-1},
\xi_{ot})=0$, where $H_2(S^{2n-1})=0$. The connected sum argument in \cite{zbMATH05709738,+1} leads to the above version, see \Cref{prop:twisted} for analogues. The functoriality of contact homology tells us that if we have a strong filling of $Y$, then $\CH(Y)\ne 0$ by \Cref{prop:APT}. If $Y$ has a weak filling $(W,\omega)$, such that $\omega|_Y$ is rational\footnote{As weak filling is an open condition for $\omega$, by a small perturbation of $\omega$, we may assume $\omega|_Y$ is rational. This condition is needed in the deformation to stable fillings constructed in \cite[Proposition 6]{MNW}.} in $H^2(Y;\R)$, then we have an algebra map $\CH_{\omega|_Y}(Y)\to \overline{\Q[H_2(W;\R)/\ker \omega]}$ by considering holomorphic curves in the stable filling from the weak filling via \cite[Proposition 6]{MNW}, see \cite[Proposition 2.5 and discussions before it]{LW}, which implies that $\CH_{\omega|_Y}(Y)\ne 0$. In other words, if $\CH_{\omega}(Y)=0$ for rational $\omega$, then $Y$ does not admit a weak filling $(W,\eta)$ such that $[\omega]=[\eta|_{Y}] \in H^2(Y;\R)$. Therefore if $Y$ is overtwisted,  $Y$ has no weak filling. This fact was also established in \cite{schmaltz2023nonfillability}.  

Similar to the contact homology case, we define the $BL_{\infty}$ algebra using the coefficient $\overline{\Q[H_2(Y;\R)/\ker \omega]}$, which leads to the definition of $\APT_{\omega}(Y)$ as well as the fully twisted version $\APT_{tw}(Y)$. Similarly, we have
\begin{proposition}
    If $\APT_{\omega}(Y)<+\infty$ for rational $\omega$, then $Y$ has no weak filling $(W,\eta)$ such that $[\omega]=[\eta|_Y]\in H^2(Y;\R)$. If $\APT_{tw}(Y)<\infty$, then $Y$ has no weak filling
\end{proposition}

\subsubsection{Deformation by Maurer-Cartan element}
Let $(X,\omega)$ be a strict strong cobordism from $(Y_-,\alpha_-)$ to $(Y_+,\alpha_+)$; it induces a Maurer-Cartan element:
$$\mc= \sum_{A,[\Gamma_-]} \# \overline{\cM}_{X,A}(\emptyset, \Gamma_-)\frac{T^{\int_A \overline{\omega}}}{\mu_{\Gamma_-}\kappa_{\Gamma_-}}q^{\Gamma^-}$$
where $\overline{\omega}$ defined on $\widehat{X}$ is $\omega$ on $X$ and $\rd \alpha_-,\rd\alpha_+$ on the cylindrical ends of $X$. Even though $\overline{\omega}$ is only continuous but not smooth on $\widehat{X}$ (as $\overline{\omega}$ does not vary smoothly in the cylindrical direction along the boundary of $X$), it is easy to show that $\int_A \overline{\omega}$ only depends on the relative homology class by Stokes' theorem.

By \cite[Proposition 3.9]{moreno2024rsft}, along with $\phi^{k,l}$ from counting  $\overline{\cM}_{X,A}(\Gamma_+, \Gamma_-)$ in the strong cobordism, we get a filtered $BL_\infty$ morphism from $V_+$ to $V_-$ deformed by $\mc$. Then \eqref{APT4} of \Cref{prop:APT} is actually a consequence of \Cref{prop:chain_mc}. More precisely, one key point from the proof of \cite[Theorem 1.1]{moreno2024rsft}, using the information of contact action, is that if $(V,p^{k,l}_{\mc})$ has finite torsion for a filtered Maurer-Cartan element $\mc$ (not necessarily geometrically defined from a strong cobordism) and $p^{k,l}$ is a $BL_\infty$ algebra geometrically defined via a contact manifold (in particular, $p^{k,l}$ in $\Lambda$-coefficient preserves the weight \cite[Definition 3.3]{moreno2024rsft}), then $(V,p^{k,l})$ also has finite torsion.

\subsection{$IBL_\infty$ structures on SFT and other torsions}
By counting holomorphic curves with arbitrary genera and punctures, we obtain operations
$$p^{s_+,s_-,g}:S^{s_+}V\to S^{s_-}V, \quad q^{\Gamma_+} \mapsto \sum_{[\Gamma_-],|\Gamma_-|=s_-} \frac{1}{\mu_{\Gamma_-}\kappa_{\Gamma_-}}\#\overline{\cM}_{Y,g}(\Gamma_+,\Gamma_-)q^{\Gamma_-} $$
where $\overline{\cM}_{Y,g}(\Gamma_+,\Gamma_-)$ is the SFT compactification of the moduli space of genus $g$ holomorphic curves with asymptotic markers and asymptotic Reeb orbits $\Gamma^+=\{\gamma_1^+,\ldots,\gamma_{s^+}^+ \}$, $\Gamma^-=\{\gamma_1^-,\ldots,\gamma_{s^-}^- \}$ near positive and negative punctures respectively, modulo reparameterization and the $\R$-translation, in the symplectization $(\R_t\times Y, \rd(e^t\alpha))$. The expected dimension is 
$$(n-3)(2-2g-s^+-s^-)+\sum_{i=1}^{s^+}\mu_{CZ}(\gamma_i^+) - \sum_{i=1}^{s^-}\mu_{CZ}(\gamma_i^-)-1$$
where the Conley-Zehnder index is defined using a symplectic trivialization of $u^*\det_{\C}\xi$ for $u\in \overline{\cM}_{Y,g}(\Gamma^+,\Gamma^-)$. We say $p^{k,l,g}$ form a (cocurved) $IBL_\infty$ algebra, if they assemble to $\widehat{p}$ on $EV[[\hbar]]$ such that $\widehat{p}^2=0$ following the explanation before \cite[Definition 5.9]{moreno2024rsft}.
\begin{remark}
    Our version of $IBL_\infty$ algebra does not keep track of the $\tau$-variable in \cite[Definition 2.3]{CFL} (by setting $\tau=1$), moreover we allow $p^{k,0,g}$. \cite[Definition 2.3]{CFL} is equivalent to $\widehat{p}$ restricted to $\overline{S}V[[\hbar]]$ squaring to $0$. With the presence of  $p^{k,0,g}$, one can define a $\widehat{p}$ on $SV[[\hbar]]$ and $\widehat{p}^2=0$ on $EV[[\hbar]]$ is equivalent to $\widehat{p}^2=0$ on $SV[[\hbar]]$  by \cite[Proposition 5.13]{moreno2024rsft}. The chain complex $(SV[[\hbar]],\widehat{p})$ gives rise to a $BV_\infty$ algebra introduced in \cite{CL:string}, which was used to define algebraic torsions in \cite{LW}. 
\end{remark}

\begin{definition}[{\cite[Definition 5.17]{moreno2024rsft}}]
    We say $Y$ has an $(n,m)$ torsion if 
    $$[\hbar^n]=0\in H_*(E^{m+1}V[[\hbar]])$$
    The algebraic torsion $\AT(Y)$ in \cite{LW} is defined to be the minimum $n$ such that $Y$ has an $(n,0)$ torsion, i.e.
    $$\AT(Y):=\min\left\{k \left| [\hbar^k]=0 \in H_*(SV[[\hbar]],\widehat{p}) \right.\right\}.$$
\end{definition}
There is another notion of algebraic torsion which takes the genus filtration into consideration, which can recover the algebraic planar torsion discussed before; see \cite[\S 5.2.2]{moreno2024rsft} for their relations. Vanishing of contact homology is equivalent to both $\AT=0$ and $\APT=0$, see \cite{zbMATH05658836,LW,moreno2024landscape}. Moreover, we have 
\begin{proposition}[{\cite[Proposition 5.21]{moreno2024rsft}}]\label{prop:0-k}
    If $Y$ has a $(0,k)$-torsion, then both $\AT(Y)\le k$ and $\APT(Y)\le k$.
\end{proposition}

In this paper, we will only give sketches of proofs for the analogous theorems for $\AT$,  as the full SFT has not yet been defined to the extent that $\AT$ is a well-defined contact invariant. In contrast with the $\APT$ case, we will not be precise about the definitions of such concepts here, and refer readers to \cite{LW,moreno2024rsft} for all details.

$IBL_\infty$ structure has a similar theory of twisting by Maurer-Cartan elements. Assuming the foundation of full SFT, then similar to \cite[Theorem 1.1]{moreno2024rsft}, we have
\begin{proposition}\label{prop:IBL_torsion}
If $W$ is a strong cobordism and $\partial_+W$ has a $(n,m)$ torsion, then $\partial_-W$ has a $(n,k)$-torsion with $k<\infty$. 
\end{proposition}
Briefly speaking, the argument is that $\partial_-W$ has a $(n,m)$ torsion for the Maurer-Cartan deformed $\widehat{p}_{\mc}$ by the deformed $IBL_\infty$ functoriality, with $\mc$ the $IBL_\infty$ Maurer-Cartan element from counting curves without positive punctures in $\widehat{W}$. Then, using the weight argument in \cite[Theorem 1.1]{moreno2024rsft}, we see that $\partial_+W$ has a $(n,k)$ torsion for $k<\infty$, possibly larger than $m$. 

\subsection{Linearized contact homology}
Here we recall some basics of linearized contact homology, which will provide a convenient language in \S \ref{s:torsion}. Let $\epsilon:(SV,\partial_{\CH}=\widehat{p}^1) \to (\Q,0)$ be a DGA augmentation. By \cite[Proposition 2.18]{moreno2024landscape}, $p^{1,1}_{\epsilon}$, defined by capping $p^{1,k}$ with $\epsilon$ for all but one output in $p^{1,k}$, is a differential on $V$. We use $\LCH_*(Y,\epsilon)$ to denote the homology of $(V,p^{1,1}_{\epsilon})$. Note that $\epsilon$ is a DGA augmentation, which depends on the choice of contact structure, almost complex structure, and any additional auxiliary choices in the virtual theory needed to define $\{p^{1,k}\}_{k\ge 0}$. We use $\LCH^{<D}_*(Y,\epsilon)$ to denote the filtered linearized contact homology generated by Reeb orbits of period smaller than $D$. Given a DGA map $\phi:SV\to SV'$ and an augmentation $\epsilon:SV'\to \Q$, then $\epsilon\circ \phi$ is an augmentation of $SV$, and $\phi^{1,1}_{\epsilon}$ induces a chain map $(V,p^{1,1}_{\epsilon\circ \phi})\to (V,p^{1,1}_{\epsilon})$. Now given $\mu:M\to Y$ from a closed oriented manifold, by counting $\overline{\cM}_Y(\gamma,\Gamma,\mu)$ weighted by $\epsilon(q^{\Gamma})$, we define a map $\tau_{\mu}:\LCH^{<D}_*(Y,\epsilon)\to \Q$, or precisely a map $\tau:\LCH^{<D}_*(Y,\epsilon)\to H^*(Y;\Q)$ such that $\tau_{\mu}(x)=\la \tau(x),\mu_*[M] \ra$. 

Given an exact cobordism $X$ from $Y_-$ to $Y_+$, it was only shown in \cite{zbMATH07656377,zbMATH07085531} that the DGA map from virtual counting of curves in $X$ up to chain homology is well-defined. To get functorial contact invariants (independent of contact structures, almost complex structures, and auxiliary choices) from the linearized contact homology $\LCH_*(Y,\epsilon)$, we need to update the chain homotopy in \cite{zbMATH07656377,zbMATH07085531} to a DGA homotopy\footnote{This will be done in \cite{DG}.}. Nevertheless, the set of all linearized contact homology is a contact invariant by the work of Chaidez \cite{zbMATH08118928} using a version of Whitehead theorem for DGAs. On the other hand, one way to avoid touching the foundations of those subtle invariant and functorial questions regarding linearized contact homology is to relate the linearized contact homology to another invariant that enjoys better invariant and functorial properties. This was carried out by Bourgeois and Oancea \cite{zbMATH05533194} that the linearized contact homology with an augmentation from a Liouville filling $W$ is isomorphic to the positive $S^1$-equivariant symplectic cohomology $SH^*_{+,S^1}(W;\Q)$ (with the convention in \cite{zbMATH07567794}). The following is a quantitative version of the main theorem in  \cite{zbMATH05533194}, where $SH^{*,<D}_{+,S^1}(W;\Q)$ is the filtered positive $S^1$-equivariant symplectic cohomology generated by Reeb orbits of period smaller than $D$. The tautological long exact sequence gives a map $SH^{*,<D}_{+,S^1}(W;\Q) \to H^{*+1}(W;\Q)\otimes Q[u,u^{-1}]/\la u \ra$, we can compose it with the projection to $H^*(W;\Q)$ and then to $H^*(Y;\Q)$ to obtain a map $SH^{*,<D}_{+,S^1}(W;\Q) \to H^{*+1}(Y;\Q)$. 

\begin{theorem}[\cite{zbMATH05533194}]\label{thm:BO}
Let $W$ be an exact filling of the contact manifold $Y$ and $n=\dim_{\C}W$. For any augmentation $\epsilon_W$ from virtual counting holomorphic planes in $W$ (hence contact forms etc., are fixed), we have an isomorphism $\rho_{BO}$ and a commutative diagram:
    $$
    \xymatrix{
    \LCH^{<D}_{2n-3-*}(Y,\epsilon_W)\ar[d]^{\rho_{BO}} \ar[r]^{\tau} & H^{*+1}(Y;\Q)\ar[d]^{=}\\
    SH^{*,<D}_{+,S^1}(W;\Q) \ar[r] & H^{*+1}(Y;\Q)
    }
    $$
\end{theorem}
In contact homology, a Reeb orbit $\gamma$ is graded by $\mu_{\CZ}+n-3$ in $\Z,\Z/2$ or $\Q$ depending on the context. While in symplectic cohomology, the grading is $n-\mu_{\CZ}$. The proof can be decomposed into two steps: (1) Proving an isomorphism between $\LCH^{<D}_{2n-3-*}(Y,\epsilon_W)$ and $S^1$-equivariant symplectic cohomology on $\widehat{Y}$ with augmentation $\epsilon_{W}$ as in \cite[\S 4.2]{AOT}, this was done in \cite[\S 6]{zbMATH05533194}. (2) Proving the latter cohomology is isomorphic to the positive $S^1$-equivariant symplectic cohomology of $W_0$ by neck-stretching along the contact boundary, this was done in \cite[\S 5.2]{zbMATH05533194}.

\section{From torsion cobordism to torsion}\label{s:torsion}
\subsection{Holomorphic curves in $\partial(\Sigma\times \D)$}\label{ss:curve}
We first recall a contact form on $\partial(\Sigma \times \D)$ following \cite[\S 2.1]{zbMATH07673358}. Let $f$ be an auxiliary Morse function on $\Sigma$, satisfying the following properties: 
\begin{itemize}
    \item it is ``self-indexing'', by which we mean that $f(p)>f(q)$ if and only if $\ind (p) > \ind (q)$ for every pair of critical points $p,q$;
    \item $f(p)$ is approximately $1$ for every critical point $p$ of $f$;
    \item $f<1$ on the interior of $\Sigma$ and $f\to 1$ when approaching the boundary of $\Sigma$, whose precise meaning will be clear from the discussion below;
    \item $f$ has a unique local minimum.
\end{itemize}
We use such an $f$ to smooth the corners of the hypersurface $\partial(\Sigma \times \D)$ inside the completion $(\widehat{\Sigma}\times \C, \widehat{\lambda}_\Sigma+\frac{r^2}{2\pi}\rd \theta)$ by using the graph of $r^2=1/f$ to bump up the $\Sigma \times S^1$ part. It is clear that we can choose $f$ such that $\rd f$ is $C^1$-small outside a collar neighborhood of $\partial \Sigma$ in $\Sigma$, and $f$ only depends on the Liouville coordinate near $\partial \Sigma$, with derivatives diverging when approaching the boundary, so that the perturbed hypersurface is smooth and of contact type. For an $f$ chosen as above, the perturbed contact form $\alpha$ is given by $\lambda_\Sigma+\frac{r^2}{2\pi}\rd \theta$ on $\partial \Sigma \times \D$, and by $\lambda_\Sigma +\frac{1}{2\pi f} \rd \theta$ on $\Sigma \times S^1 = \partial(\Sigma \times \D)\setminus  \partial \Sigma \times \D$. These constructions are explained in detail in \cite[\S 2.1]{zbMATH07673358}. Moreover, the contact form satisfies the following properties.
\begin{enumerate}
	\item Each critical point $p$ of $f$ corresponds to a simple non-degenerate Reeb orbit $\gamma_p$, which is the circle over $p$ in the region $\Sigma\times S^1\subset \partial(\Sigma \times \D)$. 
    In particular, the Reeb orbits $\gamma_{p}$'s wind around the binding $\partial \Sigma \times \{0\} \subset \partial \Sigma \times \D$ exactly once\footnote{This also includes the case when $\partial \Sigma$ is disconnected. In this case, $\gamma_{p_i}$ bounds multiple disks, which are not homologous to each other, each has intersection one with $\partial \Sigma \times \{0\}$ at some component.}. 
	\item The period of $\gamma_p$ is $1/f(p)$, which is approximately $1$ by our choice of $f$, and hence $\gamma_p$ has longer period than $\gamma_q$ if and only if $f(p)<f(q)$.
    \item The set of all Reeb orbits with period $<2$ is just $\{\gamma_p \; \vert \; p\in \mathrm{Crit}(f)\}$, see \cite[Proposition 2.2]{zbMATH07673358}. 
\end{enumerate}	

The key inputs from holomorphic curves in $\partial(\Sigma \times \D)$ are the following properties, which are essential for the main theorem in \cite{zbMATH07673358}. Proofs in \cite{bowden2022tight,zbMATH07673358} were phrased using symplectic cohomology, hence holomorphic curves with Hamiltonian perturbations. In this paper, we rephrase everything with curves in SFT. 
\begin{lemma}\label{lemma:curve_on_product}
    With the contact form above, there exist almost complex structures such that: 
    \begin{enumerate}
        \item\label{curve1} For $p\in \mathrm{Crit}(f)$, $\overline{\cM}_{\partial(\Sigma \times \D)}(\gamma_p,\Gamma)=\emptyset$ if $|\Gamma|\ge 2$. For $p,q\in \mathrm{Crit}(f)$, $\overline{\cM}_{\partial(\Sigma \times \D)}(\gamma_p,\gamma_q)$ has virtual dimension $\ind(q)-\ind(p)-1$ (i.e.\ potential components with a different virtual dimension must be empty). We may assume those with virtual dimension $\le 1$ are cut out transversely. The count of those with virtual dimension $0$ defines a chain complex generated by the $\gamma_p$, whose homology is isomorphic to $H^*(\Sigma;\Q)$.
        \item\label{curve2} For $\alpha \ne 0\in H^k(\Sigma\times \D;\Q)$, we choose a smooth map $\mu:M^k\to \partial(\Sigma\times \D)$ from an oriented closed $k$-dimensional manifold $M^k$ such that $\la \mu_*[M],\alpha|_{\partial(\Sigma \times \D)}\ra\ne 0$. Such $\mu$ exists by Thom's solution to the rational Steenrod problem. Then there exist $a_i\in \Q$ and $p_i$ with $\ind(p_i)=k$ for $1\le i \le \ell$, such that $\sum a_i\gamma_{p_i}$ is a closed class in the above chain complex and represents $\alpha$. Moreover, the following holds:
        \begin{enumerate}
            \item\label{3a} $\overline{\cM}_{\partial(\Sigma \times \D)}(\gamma_{p_i},\Gamma,\mu)=\emptyset$ for any non-empty $\Gamma$ and $1\le i \le \ell$.
            \item\label{3b} $\overline{\cM}_{\partial(\Sigma \times \D)}(\gamma_{p_i},\emptyset,\mu)$ has virtual dimension $0$ and is cut out transversely for $1\le i \le \ell$.  Moreover, $\sum a_i\# \overline{\cM}(\gamma_{p_i},\emptyset,\mu)\ne 0$, which is precisely  $\la \mu_*[M],\alpha_{\partial(\Sigma \times \D)}\ra$.
            \item\label{3c} Each curve in $\overline{\cM}_{\partial(\Sigma \times \D)}(\gamma_{p_i},\emptyset,\mu)$ intersects the symplectization of $\partial \Sigma \times \{0\}$ inside the symplectization $\widehat{\partial (\Sigma \times \D)}$ once positively.
        \end{enumerate}
    \end{enumerate}
\end{lemma}
\begin{proof}
    This is a version of \cite[Proposition 4.5]{bowden2022tight} without the $Y$ component. It is essentially contained in \cite{zbMATH07673358} but phrased using moduli spaces with Hamiltonian perturbation in symplectic cohomology, e.g.\ \cite[Proposition 4.5]{bowden2022tight}. Here, we choose to use purely SFT curves, which is cleaner in setup and can be viewed as using (linearized) contact homology. 

     The emptiness of moduli spaces in \eqref{curve1} follows from energy considerations: all Reeb orbits $\gamma_p$ have period approximately $1$. The Conley-Zehnder index of $\gamma_p$ computed using the obvious bounding disk from $\D$-factor in $\Sigma \times \D$ is $\dim_{\C}\Sigma-\ind(p)+2$ by \cite[Theorem 6.3]{zbMATH07367119}. Then the claim regarding virtual dimensions follows directly if $c_1(\Sigma)=0$. In general, for a sufficiently flat contact form on $\Sigma \times S^1\subset \partial(\Sigma \times \D)$, curves in $\overline{\cM}_{\partial(\Sigma \times \D)}(\gamma_p,\gamma_q)$ cannot pick up nontrivial homology classes in $\Sigma$ by the compactness argument in adiabatic limit in \cite[Proposition 3.1]{zbMATH07673358} or alternatively the small area argument in \cite[Proposition 4.6]{arXiv:2506.06807} without appealing to compactness results. Therefore, we can compute the virtual dimension using the Conley-Zehnder indices from the natural bounding disks. The count of those rigid curves defines the linearized contact homology $\mathrm{LCH}^{<1+\epsilon}_*(\partial(\Sigma\times \D),\epsilon_{\Sigma \times \D})$, generated by Reeb orbits of period at most $1+\epsilon$ (i.e., $\{\gamma_p\}_{p\in \mathrm{Crit}(f)}$), with augmentation $\epsilon_{\Sigma \times \D}$ from $\Sigma \times \D$. \cite[Theorem 4.3]{arXiv:2506.06807} (a specialization of \Cref{thm:BO} whose proof does not require any virtual techniques) states that $\tau:\mathrm{LCH}^{<1+\epsilon}_{2\dim_{\C}\Sigma-1-*}(\partial(\Sigma\times \D),\epsilon_{\Sigma \times \D}) \to H^{*+1}(\Sigma\times \D;\Q)\to H^{*+1}(\Sigma\times\{1\};\Q)$ is an isomorphism, respecting the Morse grading in an obvious way. We pick $a_i$ and $p_i$ such that $\sum a_i\gamma_{p_i}$ represents a closed class in $\mathrm{LCH}^{<1+\epsilon}_{2\dim_{\C}\Sigma-k}(\partial(\Sigma\times \D),\epsilon_{\Sigma \times \D})$ that corresponds to $\alpha$ under the isomorphism $\tau$ to $H^*(\Sigma;\Q)$. Hence we have $\ind(p_i)=k$. The pairing of $\tau(\sum a_i\gamma_{p_i})$ with $\mu_*([M])$ counts curves in $\overline{\cM}_{\partial(\Sigma \times \D)}(\gamma_{p_i},\emptyset,\mu)$ as well as in $\overline{\cM}_{\partial(\Sigma \times \D)}(\gamma_{p_i},\Gamma,\mu)$, with the latter weighted by the augmentation $\prod_{\gamma \in \Gamma}\epsilon_{\Sigma \times \D}(\gamma)$. For $\overline{\cM}_{\partial(\Sigma \times \D)}(\gamma_{p_i},\Gamma,\mu)$ to have positive energy, we must have $\Gamma=\{\gamma_q\}$ and $\ind(q)>\ind(p_i)$. By \cite[Proposition 4.6]{arXiv:2506.06807}, curves in $\overline{\cM}_{\partial(\Sigma \times \D)}(\gamma_{p_i},\gamma_q,\mu)$ must have trivial homology class in the $\Sigma$-direction, hence the virtual dimension is $\ind(q)-2n<0$. Therefore we may assume that $\overline{\cM}_{\partial(\Sigma \times \D)}(\gamma_{p_i},\Gamma,\mu)=\emptyset$ if $\Gamma \ne \emptyset$. By the nontriviality of the pairing $\la \tau(\sum a_i \gamma_{p_i}),\mu_*([M]) \ra$, we have $\sum a_i\# \overline{\cM}(\gamma_{p_i},\emptyset,\mu)\ne 0$. Finally, the last claim follows from the fact that $\gamma_p$ has linking number $1$ with the null-homologous $\partial \Sigma \times \{0\}$, and we can arrange that the symplectization of $\partial \Sigma \times \{0\}$ is a holomorphic hypersurface inside $\widehat{\partial(\Sigma \times \D)}$.   
\end{proof}

\subsection{Around the connected sum}
To upgrade \Cref{lemma:connected_sum} to \Cref{prop:curve}, we also need to recall the effect of contact connected sum on Reeb dynamics, which was worked out by Yau \cite{Yau} (see also \cite{Lazarev}). The following lemma can be proved by combining \cite[\S 4]{Yau} and \cite[Proof of Proposition 4.4]{Lazarev}. 
\begin{lemma}\label{lemma:connected_sum}
    Let $(Y_1,\alpha_1)$ and $(Y_2,\alpha_2)$ be two $(2n+1)$-dimensional contact manifolds with contact forms whose Reeb orbits do not cover the whole manifolds $Y_1$ and $Y_2$. Then, by picking two points $p_1\in Y_1$ and $p_2\in Y_2$ that do not lie on any Reeb orbits, we can attach a Weinstein $1$-handle to obtain a contact connected sum $(Y_1\#Y_2,\alpha_1\#_D\alpha_2)$ for $D\gg 0$ with the following properties:
    \begin{enumerate}
        \item All Reeb orbits of $\alpha_1 \#_D \alpha_2$ of period at most $D$ are either from $(Y_1,\alpha_1)$ and $(Y_2,\alpha_2)$ outside the handle attachment region, or are multiple covers $\gamma_{h,i}^k$ of the Reeb orbit $\gamma_{h,i}$ for $1\le i \le n$ contained in a standard contact sphere of dimension $2n-1$ in the belt of the $1$-handle, with $k\le k_0=\lfloor D/ \int \gamma_{h,i}^*(\alpha_1\#_D\alpha_2) \rfloor$ (for any $i$, since $\gamma_{h,i}$ has approximately the same small period). The orbits $\gamma^k_{h,i}$ are non-degenerate.
        \item The Conley-Zehnder index of $\gamma_{h,i}^k$ is $n-1+2(k-1)n+2i \ge n+1$, computed using bounding discs contained in the handle. That is, the Conley-Zehnder indices of these orbits range over $\{n+1,n+3, \ldots, 2k_0n+n-1\}$ with multiplicity one. It is the canonical global Conley-Zehnder index if $c_1(Y_1)=c_1(Y_2)=0$.
    \end{enumerate}
    The handle grows thinner as $D\to +\infty$.
\end{lemma}
The assumption that the base points $p_1,p_2$ are not on closed Reeb orbits allows us to find sufficiently small Darboux neighborhoods such that the Reeb flow starting from these Darboux balls does not return within time $D$. The orbits $\gamma_{h,i}^k$ come from a standard contact sphere of codimension $2$ contained in the co-core of the $1$-handle, with an ellipsoid contact form close to a round contact form; hence there are precisely $n$ simple Reeb orbits with approximately the same period. Their Conley-Zehnder index is precisely the Conley-Zehnder index on the ellipsoid of dimension $2n-1$, since the linearized flow in the symplectic normal direction is positive hyperbolic; see \cite[Lemma 3.1, Theorem 3.1]{Yau} for the computation in the handle model.

In the connected sum construction, the belt sphere is a convex hypersurface $B$, with both the positive and negative regions $B_{\pm}=\D^n$ and dividing set $\Gamma=(S^{2n-1},\xi_{std})$. The orbits $\gamma_{h,i}^k$ are contained in the dividing set. The following proposition, using intersection theory, will be useful to confine curves, especially for the case of torsion cobordism of type II in \S \ref{ss:torsion_II}.
\begin{proposition}[\cite{arXiv:2307.09068,connected_sum}]\label{prop:wall}
    There exists a complex structure on the symplectization of a neighborhood of the convex hypersurface $\Sigma$ in a connected sum $Y_1\#Y_2$, such that $\widehat{B}$ is foliated by two $\R$-family of $\widehat{B}_{\pm}$ and on $\widehat{B}$ in the middle, and each of the leaf is holomorphic. By the intersection theory in \cite{arXiv:2506.06807}, holomorphic curves with asymptotic Reeb orbits from those in \Cref{lemma:connected_sum} will be contained in one side of $\widehat{B}$ (containing $\widehat{B}$) in $\widehat{Y_1\#Y_2}$.
\end{proposition}
The holomorphic foliation was constructed in \cite{zbMATH07699409} for general convex hypersurfaces and was used in \cite{connected_sum} to confine curves. Holomorphic curves with asymptotic Reeb orbits from those in \Cref{lemma:connected_sum} will have zero intersection number \cite[\S 2]{arXiv:2506.06807} with all leaves in the wall $\widehat{\Sigma}$, unless there are some positive punctures asymptotic to $\gamma_{h,i}^k$, in which case the intersection number is negative. Hence, the curve is contained in one side of $\widehat{B}$, possibly contained in a leaf in $\widehat{B}$.

We now recall the ``thin" cobordism from the handle attachment from \cite{+1}. We view the $1$-handle attachment as a degenerate exact cobordism $(W,\lambda)$ from $Y_1\sqcup Y_2$ to $Y_1\# Y_2$, i.e. the cobordism has width $0$ on $Y_1\sqcup Y_2$ away from the surgery region. We call such a part the thin part of the cobordism. 
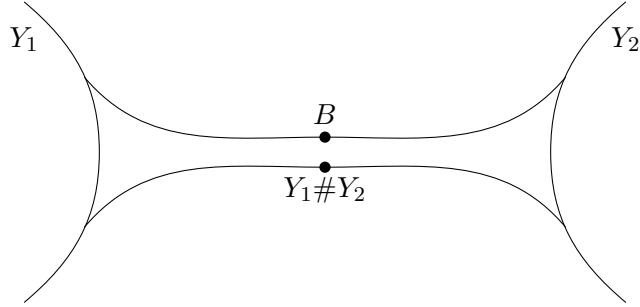
\begin{figure}[H]
    \centering
    \begin{tikzpicture}
    \draw (-4,2) [out=-40,in=90] to (-3,0) [out=-90,in=40] to (-4,-2);
    \draw (4,2) [out=220,in=90] to (3,0) [out=-90,in=140] to (4,-2);
    \draw (-3.2,1) [out=-50,in=180] to (0,0.2) [out=0,in=230] to  (3.2,1);
    \draw (-3.2,-1) [out=50,in=180] to (0,-0.2) [out=0,in=-230] to  (3.2,-1);
    \draw node at (-4,1.5) {$Y_1$};
    \draw node at (4,1.5) {$Y_2$};
    \draw node at (0,-0.5) {$Y_1\# Y_2$};
    \node at (0,0.2) [circle,fill,inner sep=1.5pt] {};
    \node at (0,-0.2) [circle,fill,inner sep=1.5pt] {};
    \node at (0,0.5) {$B$};
    \end{tikzpicture}
\end{figure}
Let $\widehat{W}$ be the completion of $W$ w.r.t.\ the Liouville vector field. We use $\overline{\lambda}$ to denote $1$-form that is the Liouville form $\lambda$ on $W$ and is $\lambda|_{\partial_+W}$ on the positive end $\partial_+ W\times \R_+$ and $\lambda|_{\partial_-W}$ on the negative end $\partial_-W\times \R_-$, this form is not smooth along the boundary of the handle, but smooth everywhere else. We will be using almost complex structures $J$, such that $\rd \overline{\lambda} (\cdot,J\cdot)\ge 0$ on where $\rd \overline{\lambda}$ is defined.  A neighborhood of $W$ in $\widehat{W}$ can be colored as follows:
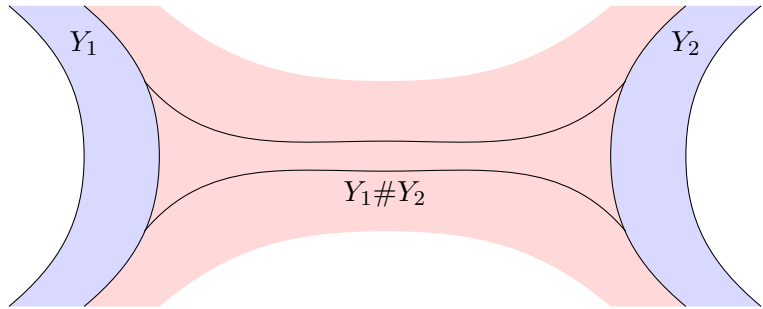
\begin{figure}[H]
    \centering
    \begin{tikzpicture}
    \path [fill=blue!15] (-5,2) to [out=-40, in=90]  (-4,0) to [out=-90,in=40] (-5,-2) to (-4,-2) to [out=40, in=-90] (-3,0) to [out=90,in=-40] (-4,2);
    \path [fill=blue!15] (5,2) to [out=220, in=90]  (4,0) to [out=-90,in=140] (5,-2) to (4,-2) to [out=140, in=-90] (3,0) to [out=90,in=220] (4,2);
    \path [fill=red!15] (-4,2) [out=-40,in=90] to (-3,0) [out=-90,in=40] to (-4,-2) to [out=0,in=180] (-3,-2) to [out=40,in=180] (0,-1) to [out=0,in=140] (3,-2) to [out=0,in=180] (4,-2) to [out=140,in=-90] (3,0) to [out=90,in=220] (4,2) to [out=180,in=0] (3,2) to [out=220,in=0] (0,1) to [out=180,in=-40] (-3,2) to [out=180,in=0] (-4,2);
    \draw (-5,2) [out=-40,in=90] to (-4,0) [out=-90,in=40] to (-5,-2);
    \draw (-4,2) [out=-40,in=90] to (-3,0) [out=-90,in=40] to (-4,-2);
    \draw (4,2) [out=220,in=90] to (3,0) [out=-90,in=140] to (4,-2);
    \draw (5,2) [out=220,in=90] to (4,0) [out=-90,in=140] to (5,-2);
    \draw (-3.2,1) [out=-50,in=180] to (0,0.2) [out=0,in=230] to  (3.2,1);
    \draw (-3.2,-1) [out=50,in=180] to (0,-0.2) [out=0,in=-230] to  (3.2,-1);
    \draw node at (-4,1.5) {$Y_1$};
    \draw node at (4,1.5) {$Y_2$};
    \draw node at (0,-0.5) {$Y_1\# Y_2$};
    \end{tikzpicture}
\end{figure}
Here, the concave boundary of the blue region and the convex boundary of the red region are slices of the symplectization of $Y_1\cup Y_2$ and $Y_1\#Y_2$, respectively. The union of them is the usual surgery cobordism from $Y_1\sqcup Y_2$ to $Y_1\# Y_2$. The significance of $\overline{\lambda}$ is that $\int u^*\rd \overline{\lambda}\ge 0$ for any holomorphic curve $u$ in $\widehat{W}$. In particular, if $u$ has positive asymptotics $\Gamma_+$ and negative asymptotics $\Gamma_-$ such that all Reeb orbits in $\Gamma_+$ are outside of the surgery region. Then we have 
\begin{equation}\label{eqn:contact}
    \sum_{\gamma \in \Gamma_+} \int \gamma^*\overline{\lambda} - \sum_{\gamma \in \Gamma_-}\int \gamma^*\overline{\lambda} = \sum_{\gamma \in \Gamma_+} \int \gamma^*\alpha - \sum_{\gamma \in \Gamma_-}\int \gamma^*\alpha\ge 0,
\end{equation}
where $\alpha$ is the contact form on $Y_1\sqcup Y_2$ outside the surgery region. We also have a trivial cylinder over $\gamma$ from $Y_1\sqcup Y_2$ to $Y_1\#Y_2$ for a Reeb orbit $\gamma$ outside of the surgery region. The trivial cylinder is cut out transversely. 

Following the construction in \cite{connected_sum}, there is $2n+1$ dimensional ball (co-core) filling $C$ of the belt in the surgery cobordism $W$, then $\widehat{C}$ admits a foliation of $1$-family of holomorphic hypersurfaces diffeomorphic to $\C^n$. By the intersection theory in \cite{arXiv:2506.06807}, holomorphic curves in $\widehat{W}$ with positive asymptotics from Reeb orbits in \Cref{lemma:connected_sum} are contained in one side of $\widehat{C}$ in $\widehat{W}$, similar to \Cref{prop:wall}.

\begin{proposition}\label{prop:curve}
    Let $\Sigma$ be a $2n$-dimensional Liouville domain with $c_1^{\Q}(\Sigma)=0$, and let $Y$ be a $(2n+1)$-dimensional contact manifold. Then there exist a contact form $\alpha_\#$ on $Y_+:=Y\#\partial(\Sigma \times \D)$ and an almost complex structure $J$ such that the following holds:
    \begin{enumerate}
        \item\label{homology1} All Reeb orbits of period smaller than $2$ are either the $\gamma_{p}$ on $\partial(\Sigma\times \D)$ or the $\gamma^k_{h,i}$ from \Cref{lemma:connected_sum} on the handle, with $k<k_0$ for some $k_0$. 
        \item\label{homology2} For $p\in \mathrm{Crit}(f)$ with $\ind(p)>0$, we have $\overline{\cM}_{Y_+}(\gamma_p,\Gamma)=\emptyset$ if $\Gamma$ contains some $\gamma^k_{h,i}$ or if $|\Gamma|\ge 2$. 
        \item\label{homology3} For $p,q\in \mathrm{Crit}(f)$ with $\ind(p),\ind(q)>0$, the moduli space $\overline{\cM}_{Y_+}(\gamma_p,\gamma_q)$ has virtual dimension $\ind(q)-\ind(p)-1$. Those with virtual dimension at most $1$ are cut out transversely, and the count of those with virtual dimension $0$ defines a chain complex generated by the $\gamma_p$, whose homology is isomorphic to $\oplus_{*>0}H^*(\Sigma;\Q)$.
        \item\label{homology4} Let $\mu:M^k\to \partial(\Sigma \times \D)$ be a smooth map from an oriented closed $k$-dimensional manifold, for $1\le k \le 2n$, avoiding the connected sum region, such that there exists $\alpha \in H^k(\Sigma \times \D;\Q)$ with $\la \mu_*[M], \alpha|_{\partial (\Sigma \times \D)}\ra\ne 0$. Then there exist $a_i\in \Q$ and $p_i$ with $\ind(p_i)= k$ for $1\le i \le \ell$, such that $\sum a_i\gamma_{p_i}$ is a closed class in the homology from \eqref{homology3} representing $\alpha$, and the following holds:
        \begin{enumerate}
            \item\label{4a} $\overline{\cM}_{Y_+}(\gamma_{p_i},\Gamma,\mu)=\emptyset$ for any non-empty $\Gamma$ and $1\le i \le \ell$.
            \item\label{4b} $\overline{\cM}_{Y_+}(\gamma_{p_i},\emptyset,\mu)$ has virtual dimension zero and is cut out transversely for $1\le i \le \ell$.  Moreover, $\sum a_i\# \overline{\cM}(\gamma_{p_i},\emptyset,\mu)=\la \mu_*[M], \alpha|_{\partial (\Sigma \times \D)}\ra\ne 0$.
            \item\label{4c} Each curve in $\overline{\cM}_{Y_+}(\gamma_{p_i},\emptyset,\mu)$ intersects the symplectization of $\partial \Sigma \times \{0\}$ inside the symplectization $\widehat{\partial (\Sigma \times \D)}$ once positively.
        \end{enumerate}
    \end{enumerate}
\end{proposition}
\begin{proof}
    This proposition was stated and proved in \cite[Proposition 4.5]{bowden2022tight}. The statements here are slightly different; we now explain briefly. \eqref{homology1} follows from \Cref{lemma:connected_sum} applied to the contact form on $\partial(\Sigma \times \D)$ explained at the beginning and to a sufficiently large contact form on $Y$. \eqref{homology2} follows from the proof of the subcritical case of Proposition 4.5 in \cite{bowden2022tight} by choosing a suitable almost complex structure $J$. 

    By \Cref{prop:wall}, curves in $\overline{\cM}_{Y_+}(\gamma_p,\gamma_q)$ and $\overline{\cM}_{Y_+}(\gamma_{p_i},\Gamma,\mu)$ do not enter the $Y$-component. In particular, they cannot pick up nontrivial homology from $Y$, and their virtual dimensions can be computed purely on $\partial(\Sigma \times \D)$, where the Conley-Zehnder index of $\gamma_p$ is $\dim_{\C}\Sigma -\ind(p)+2$ because $c^{\Q}_1(\partial(\Sigma\times \D))=0$ by assumption. As a consequence, \eqref{4a} holds because the virtual dimension of $\overline{\cM}_{Y_+}(\gamma_{p_i},\Gamma,\mu)$ is negative when $\Gamma \ne \emptyset$ and $\ind(p_i)=k$.  From this perspective, \eqref{homology2} follows again by dimension reasons. 

    Now, the $a_i$ and $p_i$ are the same as in \Cref{lemma:curve_on_product}. To ensure they have similar properties on $Y\#\partial(\Sigma \times \D)$, we consider the `thin' cobordism from $Y\sqcup \partial(\Sigma\times \D)$ to $Y\#\partial(\Sigma \times \D)$ above. Curves between $\gamma_{p}$-type orbits on $Y\sqcup \partial(\Sigma\times \D)$ and $\partial(\Sigma \times \D)$ in the cobordism also do not enter the $Y$-component by intersection theory. Then by the energy argument in \cite[\S 4.2]{+1}, e.g.\ \eqref{eqn:contact}, and the virtual dimension counting (since $c^{\Q}_1(\partial(\Sigma\times \D))=0$), the cobordism map between $\gamma_{p}$ orbits on $\partial(\Sigma\times \D)$ and $Y\sqcup \partial(\Sigma\times \D)$ is the identity map coming from trivial cylinders. Then \eqref{homology3} follows from \eqref{curve1} of \Cref{lemma:curve_on_product} and \eqref{homology4} follows from \eqref{curve2} of \Cref{lemma:curve_on_product}.
\end{proof}

\begin{remark}
   Without the $c_1^{\Q}(\Sigma)=0$ assumption, we cannot argue the virtual dimension of $\overline{\cM}_{Y_+}(\gamma_p,\gamma_q)$ and $\overline{\cM}_{Y_+}(\gamma_{p_i},\Gamma,\mu)$ using the Morse index of $p,q$, even though they do not enter into $Y$. This is different from the $\partial(\Sigma \times \D)$ case in \Cref{lemma:curve_on_product}, as we can apply \cite[Proposition 4.6]{arXiv:2506.06807} to argue that the homology of the curve is trivial in the $\Sigma$-direction. The analogue is probably true here, but the proof breaks down due to the complication of the connected sum with $Y$. The key feature we need here is that the cobordism map from $Y\sqcup \partial(\Sigma \times \D)$ to $Y\#\partial(\Sigma \times \D)$ is the identity from trivial cylinders. This would also hold if $k$ is the maximal index of $f$; this is the case in \cite[Proposition 4.5]{bowden2022tight}. Since those $\gamma_{p}$ with $\ind(p)=k$ have the minimal period, the cobordism map on them must be the identity by energy reasons.
\end{remark}

\subsection{From torsion cobordism to torsion}\label{ss:torsion}
Now assume that $W$ is a torsion cobordism from $Y_-$ to $Y_+:=Y\#\partial(\Sigma \times \D)$. By assumption, there is a nontrivial homology class $A$ in $\oplus_{*>0}H_*(\Sigma;\Q)\subset H^*(\partial(\Sigma \times \D);\Q)$ that vanishes in $H_*(W;\Q)$. By Thom's solution to the Steenrod problem, we can find an oriented manifold $N$ with boundary and a smooth map $\mu:(N,\partial N) \to (W,Y_+)$ such that $\mu_*[\partial N]=\ell A$ for $\ell \ne 0 \in \Q$. We first assume that $\mu(\partial N)$ is contained in $\partial(\Sigma \times \D)$, away from the connected sum region.

\begin{proof}[Proof of \Cref{thm:torsion}]
When the torsion cobordism is exact, then $\CH(Y_-)=0$; this is the main theorem in \cite{+1}. \Cref{thm:torsion} describes the same phenomenon, with deformation by Maurer-Cartan elements due to the non-exactness of the cobordism, similar to \cite[Theorem 1.1]{moreno2024rsft}. We first briefly recall the proof of \cite[Theorem 1.1]{+1}. There exists a class $\alpha \in H^k(\Sigma;\Q)$ such that $\la A, \alpha \ra \ne 0$. Then we find $a_i$ and $\gamma_{p_i}$ as in \Cref{lemma:curve_on_product} or \Cref{prop:curve}. Then by analyzing the boundary of the virtual dimension $1$ moduli spaces $\overline{\cM}_W(\gamma_{p_i},\Gamma, \mu)$, we get 
$$\partial_{\CH} \left(x:=\sum_i a_i \sum_{|\Gamma|\ge 1} \frac{1}{\mu_{\Gamma}\kappa_{\Gamma}}\#\overline{\cM}_W(\gamma_{p_i},\Gamma, \mu)q^{\Gamma}\right)\ne 0 \in \Q. $$
This follows from 
\begin{enumerate}
    \item There are two types of degenerations (of virtual dimension $0$) in the boundary of $\overline{\cM}_W(\gamma_{p_i},\emptyset, \mu)$: one from \eqref{3b} of \Cref{lemma:curve_on_product} or \eqref{4b} of \Cref{prop:curve}, which has a nontrivial count; the other from the constant term of $\partial_{\CH}(x)$. Therefore $\partial_{\CH}(x)$ has nontrivial constant terms.
    \item There are two types of degenerations (of virtual dimension $0$) in the boundary of $\overline{\cM}_W(\gamma_{p_i},\Gamma, \mu)$ for $\Gamma \ne \emptyset$: one from \eqref{curve1} and \eqref{3a} of \Cref{lemma:curve_on_product}, or from \eqref{homology3} and \eqref{4a} of \Cref{prop:curve}, whose algebraic count is zero; the other is the $q^{\Gamma}$ coefficient of $\partial_{\CH}(x)$.
\end{enumerate}
Therefore $\partial_{\CH}(x)\ne 0 \in \Q$, i.e., the contact homology vanishes. 

Now, if the torsion cobordism is strong rather than exact. By the same analysis of degeneration as above, we have 
$$\widehat{p}_{\mc} \left(x:=\sum_i a_i \sum_{A,|\Gamma|\ge 1} \frac{T^{\int_A \overline{\omega}}}{\mu_{\Gamma}\kappa_{\Gamma}}\#\overline{\cM}_{A,W}(\gamma_{p_i},\Gamma, \mu)q^{\Gamma}\right)\ne 0,$$ 
where $\mc$ is the Maurer-Cartan element associated to $(W,\omega)$. In other words, the algebraic torsion of $Y_-$ with respect to the deformed $BL_\infty$ structure $p_{\mc}$ is zero, i.e., the Maurer-Cartan deformed contact homology is zero. Then, by the proof of \cite[Theorem 1.1]{moreno2024rsft}, the undeformed algebraic planar torsion of $Y_-$ is finite. 

The argument above requires virtual techniques to make sense of the virtual counts $\#\overline{\cM}$ and $\#\partial \overline{\cM}$; for this we use Pardon's VFC as in \cite{moreno2024landscape,moreno2024rsft}. The above argument makes sense in Pardon's VFC because the geometric count of transversal $0$-dimensional moduli spaces in \Cref{lemma:curve_on_product} and \Cref{prop:curve} equals the virtual count. 
\end{proof}

We now discuss several aspects of the $\mu(\partial N) \subset \partial(\Sigma \times \D)$ assumption in \Cref{prop:curve}:
\begin{enumerate}
    \item The only place in this paper where we use the case $Y\ne \emptyset$ is in the proofs of \Cref{thm:sphere} and \Cref{thm:all}, where we can assume that $\mu(\partial N) \subset \partial(\Sigma \times \D)$, avoiding the connected sum region. 
    \item We use maps $\mu$ from manifolds as in \S \ref{ss:curve} to simplify the notation for moduli spaces and their compactification. The count of such moduli spaces gives the pairing of $\mu_*[M]$ with a map from the (truncated) linearized contact homology to the cohomology of the contact manifold. To make sense of the map on (co)homology, we can count moduli spaces with constraints given by singular chains, pseudo-cycles, or stable manifolds of Morse functions. For our purposes, we can use a model that localizes well with respect to the connected sum, e.g., Morse theory. This is the perspective taken in \cite{bowden2022tight}. The proofs of \Cref{thm:torsion} can be modified accordingly. 
    \item Alternatively, after rescaling by nonzero integers, we may assume there is a smooth map $\eta:M\to \partial(\Sigma \times \D)$, avoiding the connected sum region, such that $\eta_*[M]=\mu_*[\partial N]$ in $H_*(Y\#\partial(\Sigma \times \D);\Q)$. The emptiness in \eqref{4a} of \Cref{prop:curve} follows from the same argument, since it relies on dimension counting. For \eqref{4b} of \Cref{prop:curve}, since $\eta_*[M]=\mu_*[\partial N]$, we can find a pseudo-cycle $\beta$ whose boundary is $\mu|_{\partial N}$ and $\eta$, and we can count $\partial\overline{\cM}(\gamma_{p_i},\emptyset,\beta)$. The boundary contributions come from $\overline{\cM}(\gamma_{p_i},\emptyset,\mu|_{\partial N})$, $\overline{\cM}(\gamma_{p_i},\emptyset,\eta)$, breakings containing rigid curves from $\overline{\cM}(\gamma_{p_i},\gamma_p)$, and breakings containing curves from $\overline{\cM}(\gamma_{p_i},\Gamma,\beta)$. The count of the third case is zero because $\sum a_i\gamma_{p_i}$ is closed. The last moduli space has virtual dimension $\ind(q)+1-\dim_{\R}\Sigma$, where $\Gamma=\{\gamma_q\}$. For this dimension to be non-negative, we must have $\ind(q)=\dim_{\R}\Sigma -1$; then $\gamma_q$ has the maximum index and minimal period among those $\gamma_p$-type orbits. We may simply assume there are no such critical points; for example, if $\Sigma$ is Weinstein and $\dim \Sigma \ge 4$. In general, the SFT degree of $\gamma_q$ is $1$; the other component of the breaking must be rigid holomorphic planes asymptotic to $\gamma_q$. The count of such holomorphic planes must be zero. Otherwise, the contact homology of $\partial(\Sigma \times \D)$ would be zero, contradicting the existence of an exact filling. As a consequence, we have $\#\overline{\cM}(\gamma_{p_i},\emptyset,\mu|_{\partial N})=\# \overline{\cM}(\gamma_{p_i},\emptyset,\eta)$. Hence \Cref{thm:torsion} holds without the assumption that $\mu(\partial N) \subset \partial(\Sigma \times \D)$.
\end{enumerate}
From the proof of \Cref{thm:torsion}, we see that the emptiness of moduli spaces in \Cref{lemma:curve_on_product} and \Cref{prop:curve} can be replaced by a weaker claim that virtual counts of the moduli spaces (of virtual dimension $0$) are zero. 

\begin{remark}
    We can relax \Cref{def:torsion_cobordism} to that $A_1\oplus A_2\in  H_k(\partial(\Sigma \times \D);\Q)\oplus H_k(Y;\Q) = H_k(\partial(\Sigma\times \D)\# Y;\Q)$ is mapped to zero in $W$, such that $A_1$ is from $\Sigma \times \{(0,1)\}$. This is because there are no holomorphic planes asymptotic to $\gamma_{p_i}$ positively and intersecting $A_2$ after representing it by a map from a closed manifold to $Y$, by  \Cref{prop:wall}. 
\end{remark}

\subsection{A quantitative upper bound}
In \cite[Question 2]{LW}, Latschev and Wendl asked whether the finiteness of algebraic (planar) torsion is preserved under strong cobordisms and whether there are quantitative estimates. The former question was answered affirmatively in \cite{moreno2024rsft}. In this section, we provide an answer to the quantitative part for a special class of cobordisms. 
\begin{definition}\label{def:k_torsion_cobordism}
    A torsion cobordism $W$ from $Y_-$ to $Y_+=Y\#\partial(\Sigma\times \D)$ is of order $k$ if there exist $k$ disjoint symplectic hypersurfaces $H_i\subset W$ such that 
    \begin{enumerate}
        \item $H_i\cap \partial (\Sigma \times \D) = \partial \Sigma \times \{p_i\}$ for distinct points $p_i\in \D$. 
        \item $H_i\cap Y_- =\emptyset$.
        \item The contact form on $Y_-$ can be extended to a Liouville form on $W\setminus \bigcup_i H_i$.
        \item The symplectic form restricted to $H_i$ is exact.
    \end{enumerate}
\end{definition}

\begin{proposition}\label{prop:mc_intersection}
    For a torsion cobordism $W$ of order $k$, if we pick almost complex structures $J$ such that each $\widehat{H}_i$ is $J$-holomorphic, then the Maurer-Cartan element $\mc$ from $W$ has a decomposition $\mc = \oplus_{v\in \N^k\backslash\{0\}}\mc_v$, where $\mc_v$ counts holomorphic caps in $W$ whose topological intersection number with $\widehat{H}_i$ is $v_i$.
\end{proposition}
\begin{proof}
    Since each $H_i$ is exact, there are no $J$-holomorphic spheres in $\widehat{H}_i$. As a consequence of positivity of intersection, any $J$-holomorphic curve that contributes to the Maurer-Cartan element must have non-negative intersection with each $\widehat{H_i}$. Moreover, since the contact form extends to a Liouville form on the complement of the $H_i$, such curves must intersect positively with at least one of the $\widehat{H_i}$. Therefore, the unperturbed non-empty compactified moduli spaces for Maurer-Cartan elements decompose according to indices $v\in \N^k\backslash\{0\}$, where $v_i\ge 0$ is the topological intersection number with $\widehat{H}_i$. Since the virtual count is nonzero only if the unperturbed compactified moduli space is non-empty, the claim follows.
\end{proof}

\begin{theorem}\label{thm:finite_APT_k}
    Let $Y$ be the concave boundary of a $k$-torsion cobordism. Then $\APT(Y)\le k$.
\end{theorem}
\begin{proof}
The almost complex structure used in \Cref{lemma:curve_on_product} and \Cref{prop:curve} on $\partial \Sigma \times \D$ was described in the proof of \cite[Lemma 4.6]{bowden2022tight}. It is straightforward to check that the symplectization of $\partial \Sigma \times \{p_i\}$ is a holomorphic hypersurface. Therefore, we may assume that the almost complex structure $J$ on $\widehat{W}$ is such that each $\widehat{H}_i$ is $J$-holomorphic. 

For the moduli space $\overline{\cM}_W(\gamma_{p_i},\Gamma, \mu)$ in the proof of \Cref{thm:torsion}, since any possible breaking from the top level cannot have punctures asymptotic to Reeb orbits in $\partial \Sigma \times \D$, which have large period, we can argue as in \Cref{prop:mc_intersection} that $\overline{\cM}_W(\gamma_{p_i},\Gamma, \mu)$ decomposes into components $\overline{\cM}_{W,v}(\gamma_{p_i},\Gamma, \mu)$ with non-negative intersection numbers $v\in \N^k$ with the $\widehat{H}_i$. We introduce formal variables $t_1,\ldots,t_k$ to record the intersections with $\widehat{H}_1,\ldots,\widehat{H}_k$, and write $\hat{\mc}= \oplus_{v\in \N^k\backslash\{0\}}\mc_vt^v$. Let $x$ be the coefficient of $\prod_{i=1}^k t_i$ in 
$$\left(\sum_i a_i \sum_{v, A,|\Gamma|\ge 1} \frac{T^{\int_A \overline{\omega}}}{\mu_{\Gamma}\kappa_{\Gamma}}\#\overline{\cM}_{v,A,W}(\gamma_{p_i},\Gamma, \mu)q^{\Gamma}t^v\right)\odot e^{\hat{\mc}}. $$
Since the contact form extends to a Liouville form on the complement of $\cup H_i$, the exponent $\int_A\overline{\omega}$ is determined by the intersection numbers $v$. To see this, by the assumption of $k$-torsion cobordism, we have a Liouville form $\lambda_0$ on $W\backslash \cup_{i=1}^kH_i$ extending the contact form on $\partial_-W$. There are also $1$-forms $\lambda_i$ near $H_i$, such that $\rd \lambda_i = \omega$. Let $UH_i$ be the tubular neighborhood near $H_i$ where $\lambda_i$ is well-defined. As a consequence, $\lambda_0-\lambda_i$ defines a closed $1$-form on $UH_i\backslash H_i$, and we use $c_i$ to denote the fiber circle around $H_i$. Finally, $\lambda_0-\lambda_+$ is closed on $\partial_+W \backslash \cup_{i=1}^k H_i$, where $\lambda_+$ is the restriction of the local Liouville form of $W$ near $\partial_+W$. Then, by Stokes' theorem, we have 
$$\int_A \overline{\omega} = \int \gamma_{p_i}^*\lambda_0 -\sum_{\gamma \in \Gamma}\int \gamma^*\lambda_0 + \sum_{i=1}^kv_i\cdot \la [\lambda_0-\lambda_i],c_i \ra. $$
As a consequence, $x$ is a finite sum and can be specialized at $T=1$. Moreover, since $\hat{\mc}$ has no constant term by \Cref{prop:mc_intersection}, we have $\bm{x}:=x|_{T=1}\in E^{1+k}V$. 
We claim that $\widehat{p}(\bm{x})\ne 0 \in \Q$, where $\widehat{p}:E^iV \to E^iV$ comes from the $BL_\infty$ algebra on $Y$. This follows from the fact that
$$\left(\sum_i a_i \sum_{v, A} \frac{T^{\int_A \overline{\omega}}}{\mu_{\Gamma}\kappa_{\Gamma}}\#\partial\overline{\cM}^{\virdim =1}_{v,A,W}(\gamma_{p_i},\Gamma, \mu)q^{\Gamma}t^v\right)\odot e^{\hat{\mc}}=0.$$
We look at the coefficient of $\prod_{i=1}^k t_i$ in the above expression and specialize at $T=1$, i.e.,
\begin{equation*}
    \left(\left(\sum_i a_i \sum_{v, A, [\Gamma]} \frac{T^{\int_A \overline{\omega}}}{\mu_{\Gamma}\kappa_{\Gamma}}\#\partial\overline{\cM}^{\virdim =1}_{v,A,W}(\gamma_{p_i},\Gamma, \mu)q^{\Gamma}t^v\right)\odot e^{\hat{\mc}}\right)_{\prod t_i, T=1}=0\in EV.
\end{equation*}
There are two types of degenerations (of virtual dimension $0$) in the boundary of $\overline{\cM}^{\virdim =1}_{v,A,W}(\gamma_{p_i},\Gamma, \mu)$: 
\begin{enumerate}
    \item One type comes from breaking from the top-level. When $\Gamma=\emptyset$, we have contributions from \eqref{3b} of \Cref{lemma:curve_on_product} or \eqref{4b} of \Cref{prop:curve}, which have nontrivial counts and, by \eqref{3c} of \Cref{lemma:curve_on_product} or \eqref{4c} of \Cref{prop:curve}, intersect each $H_i$ exactly once. Hence they contribute to the constant term in $EV$. All other breakings from above are either empty or have virtual count $0$ by \eqref{curve1} and \eqref{3a} of \Cref{lemma:curve_on_product}, or by \eqref{homology3} and \eqref{4a} of \Cref{prop:curve}.
    \item The other type, breaking from the bottom-level, corresponds precisely to $\widehat{p}(\bm{x})\in EV$. 
\end{enumerate}
Hence $\widehat{p}(\bm{x})\ne 0 \in \Q\subset EV$, and thus $\APT(Y)\le k$.
\end{proof}

\subsection{Torsion from torsion cobordisms of type II}\label{ss:torsion_II}

\begin{proposition}\label{prop:curve_flexible}
Let $W_0$ be a flexible Weinstein domain of dimension $2n+2$ with $c_1^{\Q}(W_0)=0$ and $Y$ is a contact manifold. There exists $D>0$ and a contact form $\alpha$ on $Y_+=Y\#\partial W_0$ such that the following holds.
\begin{enumerate}
    \item All Reeb orbits of period smaller than $D$, are either from $\partial W_0$, which has Conley-Zehnder index at least $1$ or $\gamma^m_{h,i}$ from the handle, whose Conley-Zehnder index is at least $n+1$ by \Cref{lemma:connected_sum}. Here, the Conley-Zehnder indices are computed using a trivialization of $\det\oplus^N\xi$ over $\partial W_0$ only.
    \item Let $\mu:M^k\to Y\#\partial W_0$ be a map from an oriented closed $k$-dimensional manifold for $k\ge 2$, such that $\mu_*[M]$ represents a homology class in $\partial W_0$ and there exists $\alpha \in H^*(W_0;\Q)$ with $\la \mu_*[M], \alpha|_{\partial W_0} \ra \ne 0$.
    \begin{enumerate} 
        \item There exists $a_i\in \Q$ and Reeb orbit $x_i$ of Conley-Zehnder index $n+2-k$ for $1\le i \le \ell$, $\sum a_i\#\overline{\cM}_{Y_+}(x_i,\emptyset,\mu)\ne 0$. For $\Gamma \ne \emptyset$, we have $\#\overline{\cM}_{Y_+}(x_i,\Gamma,\mu)=0$.
        \item$\sum a_i\#\overline{\cM}_{Y_+}(x_i,\Gamma)=0$, for any $\Gamma \ne \emptyset$.
    \end{enumerate}
    Here $\#\overline{\cM}$ only counts the components with virtual dimension $0$.
\end{enumerate}
\end{proposition}
\begin{proof}
Following the beginning of the proof of \cite[Proposition 4.5]{bowden2022tight}, there exists $D>0$ and contact forms $\alpha,\alpha_{\#}$ on $\partial W_0$ and $Y\#\partial W_0$ respectively, such that the following holds:
\begin{enumerate}
    \item Reeb orbits of $\alpha$ of period smaller than $D$ on $\partial W_0$ are non-degenerate and the Conley-Zehnder indices are at least $1$. Reeb orbits of $\alpha_{\#}$ of period smaller than $D$ on $Y\#\partial W_0$ are either those on $\partial W_0$ or those $\gamma^m_{h,i}$ from the handle in \Cref{lemma:connected_sum}. In particular, we choose the contact form on $Y$ to be sufficient large, i.e.\ all Reeb orbits on $Y$ have long periods $\gg D$.  Here, the Conley-Zehnder indices are computed using a trivialization of $\det\oplus^N\xi$ over $\partial W_0$ only. In particular, the SFT degrees are positive. 
    \item Using the contact form $\alpha$, we have $SH_{S^1,+}^{*,<D}(W_0;\Q)\to H^{*+1}(W_0;\Q)$ is surjective. This is consequence of the vanishing of symplectic cohomology for flexible domains by \cite{zbMATH06035983} or \cite[Theorem 3.2]{zbMATH06864340} and the Gysin exact sequences \cite{zbMATH06247539} relating $S^1$-equivariant and non-equivariant theories.
\end{enumerate}    
By \Cref{thm:BO}, we have $SH_{S^1,+}^{*,<D}(W_0;\Q)$ is isomorphic to the filtered linearized contact homology $\mathrm{LCH}^{<D}_{2n-1-*}(\partial W_0, \epsilon_{W_0})$ with augmentation $\epsilon_{W_0}$ induced from the filling $W_0$. Since $\mathrm{LCH}^{<D}_*(\partial W_0, \epsilon_{W_0})$ only involves augmentations to Reeb orbit of period smaller than $D$, the relevant augmentation in $\epsilon_{W_0}$ is zero as the SFT degrees are all positive. As a consequence, the constant term of the contact homology differential $\partial_{\CH}$ is zero on those orbits\footnote{This can be seen alternatively as follows:}. Now by \Cref{thm:BO}, the surjection  $SH_{S^1,+}^{*,<D}(W_0;\Q)\to H^{*+1}(W_0;\Q)$ implies that $\mathrm{LCH}^{<D}_{2n-1-*}(\partial W_0, \epsilon_{W_0}) \to H^{*+1}(W_0;\Q)\to H^{*+1}(\partial W_0;\Q)$ is surjective onto the image of $H^*(W_0;\Q)\to H^*(\partial W_0;\Q)$. The map $\tau: \mathrm{LCH}^{<D}_{2n-1-*}(\partial W_0, \epsilon_{W_0})\to  H^{*+1}(\partial W_0;\Q)$ paired with $\mu_*[M]$ counts $\overline{\cM}_{\partial W_0}(x,\Gamma,\mu)$ weighted by augmentation $\epsilon_{W_0}$ to $\Gamma$. Hence in our case, it counts $\overline{\cM}_{\partial W_0}(x,\emptyset,\mu)$ only. Therefore when $Y=\emptyset$, the claims follows from that there exist $a_i\in \Q$ and Reeb orbits $x_i$ from $\partial W_0$ with Conley-Zehnder index $n+2-k$ for $1\le i \le \ell$, such that $\sum a_ix_i$ represents a closed class in the linearized contact homology and $\la\tau(\sum a_i x_i), \mu_*[M]\ra\ne 0$. This means
\begin{enumerate}
    \item $\sum a_i\#\overline{\cM}_{\partial W_0}(x_i,\mu,\emptyset)\ne 0$;
    \item $\sum a_i\#\overline{\cM}_{\partial W_0}(x_i,y)=0$, where $y$ necessarily has Conley-Zehnder index $n+1-k$;
\end{enumerate}
The vanishing of virtual count for $\overline{\cM}_{\partial W_0}(x_i,\mu, \Gamma)$ with $\Gamma\ne \emptyset$ follows from negativity of virtual dimension, as orbits in $\Gamma$ have positive SFT degree. For $\overline{\cM}_{\partial W_0}(x_i,\Gamma)$ with $|\Gamma|\ge 2$, since all Conley-Zehnder index is at least $1$, we have 
\begin{equation}\label{eqn:vdim}
    \virdim \overline{\cM}_{\partial W_0}(x_i,\Gamma) \le n-k-(|\Gamma|-1)(n-1)<0,
\end{equation}
as $k\ge 2$. 

Now we consider the case with nontrivial $Y$. We use $C_{\le n+2-k}(W_0)$ to denote the sub-complex generated by Reeb orbits with Conley-Zehnder index at most $n+2-k$ and period small than $D$. As a consequence $H_{2n-1-*}(C_{\le n+2-k}(W_0))\to \mathrm{LCH}^{<D}_{2n-1-*}(\partial W_0, \epsilon_{W_0}) \to H^{*+1}(W_0;\Q)$ is surjective onto $H^k(W_0;\Q)$, in particular contains $\alpha$ in the image. Now we write $V$ as the gluing of the thin cobordism $W$ with $W_0$, which is an exact cobordism from $Y$ to $Y\#\partial W_0$. We want to define $\mathrm{LCH}^{<D}(Y\#\partial W_0,\epsilon_V)$ using an augmentation from $V$. In general, this can not be done as $V$ has a concave boundary $Y$. We can not define an augmentation from counting holomorphic planes in $V$, as planes can develop negative punctures asymptotic to Reeb orbits on $Y$ in the SFT compactification. However, as the Reeb orbits on $Y$ have long periods $\gg D$, if we only consider augmentations to those with period small than $D$, we do not have such problematic degeneration, and $\epsilon_V$ can be defined. For our purpose of functoriality, we need to stretch the boundary $\partial W_0$ and consider the augmentation $\epsilon_V=\epsilon_{W_0}\circ \phi_W$, where $\phi_W$ is from counting curves in the thin cobordism $W$. Similar to before, when the input of $\phi_W$ has a period smaller than $D$, the output of $\phi_W$ must consist of Reeb orbits on $\partial W_0$, and hence $\epsilon_V=\epsilon_{W_0}\circ \phi_W$ is augmentation to define $\mathrm{LCH}^{<D}(Y\#\partial W_0,\epsilon_V)$. Similarly, we use $C_{\le n+2-k}(V)$ to denote the sub-complex of those orbits on $\partial W_0$ with Conley-Zehnder index at most $n+2-k$ and period smaller than $D$. It is indeed a sub-complex as those $\gamma^{k}_{h,i}$ have Conley-Zehnder index at least $n+1\ge n+2-k$. The reason that we still have a filtration from Conley-Zehnder indices computed in the $\partial W_0$-side is that all the curves with asymptotics from $\partial W_0$ or the $\gamma_{h,i}^m$ care contained in the one side of the wall $\widehat{B}$ by \Cref{prop:wall}. Similarly, the augmentation curves in $\epsilon_V$, i.e.\ constant term of $\phi_W$, also do not enter the $Y$-part by intersection theory, and hence $\epsilon_V=0$ by degree reasons for relevant low-period orbits. Using the thin cobordism, we can argue as in \cite[\S 4.2]{+1} using contact action in the thin cobordism that $C_{\le n+2-k}(V)\to C_{\le n+2-k}(W_0)$ induces a cochain isomorphism on homology through the cobordism from the trivial cylinders. We have the following commutative diagram:
$$
\xymatrix{
H_{2n-k}(C_{\le n+2-k}(V)) \ar[r]\ar[d]_{\simeq}^{\phi^{1,1}_{W,\epsilon_{W_0}}} & \mathrm{LCH}_{2n-k}^{<D}(Y\#\partial W_0,\epsilon_V=\epsilon_{W_0}\circ \phi_W) \ar[r]\ar[d]^{\phi^{1,1}_{W,\epsilon_{W_0}}} &  H^{k}(Y\#\partial W_0;\Q)\ar[d]^{\pi}\\
H_{2n-k}(C_{\le n+2-k}(W_0)) \ar[r] & \mathrm{LCH}_{2n-k}^{<D}(\partial W_0,\epsilon_{W_0}) \ar[r] &  H^k(\partial W_0;\Q)
}
$$
where $\pi:H^k(Y\#\partial W_0;\Q) =H^k(Y;\Q)\oplus H^k(\partial W_0;\Q)\to H^k(\partial W_0;\Q)$. As a consequence, we have a linear combination $\sum a_i x_i\in C_{\le n+2-k}(V)$ with $x_i$ of Conley-Zehnder index $n+2-k$, which represents a closed class and maps to $\alpha|_{\partial W_0}$ in the above diagram. In particular, we have $\sum a_i\#\overline{\cM}_{Y_+}(x_i,y)=0$ for any Reeb orbit $y$ of Conley-Zehnder index $n+1-k$ and $\sum a_i\#\overline{\cM}_{Y_+}(x_i,\emptyset,\mu)\ne 0$. The other moduli spaces in the claim has virtual count zero as they have negative virtual dimension as the $Y=\emptyset$ case, since they do not enter the $Y$-part to pick up non-trivial first Chern classes by \Cref{prop:wall}.
\end{proof}

From \Cref{prop:curve_flexible}, we have the following result by the exactly same argument in \Cref{thm:torsion}.
\begin{theorem}\label{thm:torsion_II}
    The negative boundary of a torsion cobordism of type II has finite $\APT$.
\end{theorem}

Now, we discuss situations where we can derive the finiteness of the fully twisted algebraic planar torsion.
\begin{proposition}\label{prop:twisted}
    For a ($k$-)torsion cobordism $W$ of type I/II from $Y_-$ to $Y_+$ for $Y_+=Y\#\partial W_0$, where $W_0$ is either $\partial(\Sigma\times \D)$ or a flexible domain with $c_1^{\Q}(W_0)=0$ in either case. Assume that (1) $H_2(Y_-;\R)\to H_2(W;\R)$ is injective, and (2) $\ker(H_2(\partial W_0;\R)\to H_2(W_0;\R))$ is contained in $\ker(H_2(\partial W_0;\R)\to H_2(W;\R))$ via the inclusion $\partial W_0\backslash \{*\} \subset Y_-\subset W$. Then $\APT_{tw}(Y_-)$ is finite (or $\le k$).
\end{proposition}
\begin{proof}
We simply modify the proofs of Theorems \ref{thm:torsion}, \ref{thm:finite_APT_k} and \ref{thm:torsion_II} by keeping track of the homology classes of curves in $W$. Moreover, we need to modify \Cref{prop:curve,prop:curve_flexible} to 
$a_i\in \Q[H_2(Y_+;\R))]=\{\sum \alpha_i t^{A_i}, \alpha_i\in \Q, A_i\in H_2(Y_+;\R)\}$, such that 
\begin{eqnarray}
(\sum_i\sum_A a_i\cdot t^A \#\overline{\cM}_{A}(\gamma_{p_i},\Gamma))\otimes_{\Q[H_2(Y_+;\R)]}\Q[H_2(W;\R)] & = & 0\in \Q[H_2(W;\R)] \label{twisted1},\\
(\sum_i\sum_A a_i\cdot t^A \#\overline{\cM}_{A}(\gamma_{p_i},\emptyset,\mu))\otimes_{\Q[H_2(Y_+;\R)]}\Q[H_2(W;\R)] &\ne &0\in \Q\subset \Q[H_2(W;\R)]. \label{twisted2}
\end{eqnarray}
Suppose they hold, arguments in Theorem \ref{thm:torsion} and \ref{thm:torsion_II} works for $\Q[H_2(W;\R)]$ coefficient, we can conclude that Therefore, $(EV\otimes \Q[H_2(Y_-,\R)])\otimes_{\Q[H_2(Y_-,\R)]} \Q[H_2(W;\R)] = EV\otimes \Q[H_2(W;\R)]$ has finite $\APT$, where $V$ is the space generated by good Reeb orbits on $Y_-$. Since $H_2(Y_-;\R)\to H_2(W;\R)$ is injective, we can recover $EV\otimes \Q[H_2(Y_-,\R)]$ from $EV\otimes \Q[H_2(W;\R)]$ by a base change via a projection $H_2(W;\R)\to H_2(Y;\R)$; the claim follows. To see \eqref{twisted1} and \eqref{twisted2}, note that the symplectic cohomology of $W_0:=\Sigma \times \D$ or flexible domain vanishes twisted coefficient by \cite[Theorem 3.2]{zbMATH06864340}. We know that we can find such $a_i$ if we do a base change to $\Q[H_2(W_0;\R)]$. Now the second assumption implies that such $a_i$'s work for base change with $\Q[H_2(W;\R)]$ as well.
\end{proof}
The second condition is satisfied automatically if $W_0$ Weinstein and of complex dimension at least $4$, or $W_0=\Sigma \times \D$ is subcritical Weinstein with complex dimension at least $3$. In both cases, $\ker(H_2(\partial W_0;\R)\to H_2(W_0;\R))$ is trivial.

\begin{proof}[Proof of Theorems \ref{thm:sphere} and \ref{thm:all}]
    The strong cobordisms in \cite[Theorem 5.1]{bowden2022tight} from $(S^{2n+1},\xi_{ex})$ and strong cobordisms in \cite[Theorem A]{bowden2022tight} from $Y\#(S^{2n+1},\xi_{ex})$ are torsion cobordisms of type I if $n\ge 3$ and of type II if $n=2$. Let $W$ denote the torsion cobordism, then $\partial_-W\hookrightarrow W$ induces an injection on second homology by construction. For the positive end, the second condition of \Cref{prop:twisted} holds automatically if $n\ge 3$ by the discussion above. When $n=2$, $\partial_+W=Y\#\partial W_0$ with $W_0$ a flexible domain constructed by attaching one flexible handle to a subcritical domain. The same flexible handle attachment is contained in the torsion cobordism as well. As only the flexible handle can kill a second homology class in $\partial W_0\hookrightarrow W_0$, we see that the second condition of \Cref{prop:twisted} holds. Then by \Cref{prop:twisted}, we have $\APT_{tw}((S^{2n+1},\xi_{ex}))$, $\APT_{tw}(Y\#(S^{2n+1},\xi_{ex}))<\infty$. To see \Cref{thm:all}, if $Y$ is already tight and not weakly fillable, then there is nothing to prove. If $Y$ has a weak filling, then $\CH_{tw}(Y)\ne 0$. As the cobordism from $Y\sqcup (S^{2n+1},\xi_{ex})$ to $Y\#(S^{2n+1},\xi_{ex})$ induces an isomorphism on the second homology, we have a unital algebra map from $\CH_{tw}(Y\#(S^{2n+1},\xi_{ex}))\to \CH_{tw}(Y)\otimes \CH_{tw}((S^{2n+1},\xi_{ex}))$. By \cite[Theorem 5.1]{bowden2022tight}, $\CH((S^{2n+1},\xi_{ex}))=\CH_{tw}(S^{2n+1},\xi_{ex})\ne 0$. As a consequence, we have $\CH_{tw}(Y\#(S^{2n+1},\xi_{ex}))\ne 0$ and $1\le \APT_{tw}(Y\#(S^{2n+1},\xi_{ex}))<\infty$. In particular, $Y\#(S^{2n+1},\xi_{ex})$ is tight, but not weakly fillable. 

    Finally, to see $\APT(S^{2n+1},\xi_{ex})=1$ when $n\ge 3$. The torsion cobordism $W$, is similar to a $1$-torsion cobordism, but with two $1$-handles and one $2$-handle attached on the $\partial(\Sigma \times \D)$ part. By \cite[\S 4]{Yau}, similar to \Cref{lemma:connected_sum}, those handle attachments will create Reeb orbits with possibly small period of Conley-Zehnder index at least $n$. Since the annihilated homology class the torsion cobordism from \cite{bowden2022tight} has degree $n\ge 3$, in particular, the Reeb orbits $\gamma_{p_i}$ from \Cref{prop:curve} have Conley-Zehnder index $2$. Hence by degree reasons, we may assume those new orbits can not make a difference to \Cref{prop:curve}. Hence the proof of \Cref{thm:finite_APT_k} implies that $\APT(S^{2n+1},\xi_{ex})\le 1$.
\end{proof}

\subsection{Strongly non-cofillable manifolds}

\begin{definition}\label{def:strongly_not_cofillable}
    A contact manifold $Y$ is strongly non-cofillable, if there exists a contact form $\alpha$ such that the following holds. 
    \begin{enumerate}
    \item All Reeb orbits of periods $<D$ are non-degenerate;
    \item There exist good Reeb orbits $\gamma_1,\ldots,\gamma_k$ of period $<D/k$, such that for any subset $S\subset \{ 1,\ldots,k\}$, we have $\#\overline{\cM}(\{\gamma_i\}_{i\in S}, \Gamma)=0$;
    \item For any point $o$, $\#\overline{\cM}(\{\gamma_i\}_{1\le i \le k},\emptyset,o)\ne 0$ and $\#\overline{\cM}(\{\gamma_i\}_{1\le i \le k},\Gamma,o)=t$ for $\Gamma \ne \emptyset$;
\end{enumerate}
Here $\#\overline{\cM}$ is the virtual count of compactified moduli spaces of virtual dimesnion $0$.
\end{definition}
Those conditions implies that planarity \cite[Definition 3.24]{moreno2024landscape} of $Y$ is at most $k$ and was used in \cite[Proposition 4.7]{moreno2024rsft}. 

\begin{example}
    Many examples with finite planarity from \cite{moreno2024landscape} are in this category. Typical examples include $\partial(\Sigma \times \D)$, iterated planar open book, etc. Moreover, not only are the virtual counts of those moduli spaces in \Cref{def:strongly_not_cofillable} zero, they are actually empty in those cases.
\end{example}

\begin{theorem}\label{thm:torsion_III}\label{thm:APT_III}
    The negative boundary of a torsion cobordism of type III has finite $\APT$.
\end{theorem}
\begin{proof}
    Let $W$ be the torsion cobordism with $\partial_+W=Y_1\sqcup Y_2$, where $Y_1$ is strongly not cofillable. We consider a path $\eta$ in $W$ from $o\in Y_1$ to $o'\in Y_2$. We define 
    $$\phi^{k,l}_{\bullet}(q^{\Gamma^+})=\sum_{A,|\Gamma^-|=l} \frac{T^{\int_A\overline{\omega}}}{\mu_{\Gamma^-}\kappa_{\Gamma^+}}\#\overline{\cM}_{W,A}(\Gamma^+,\Gamma^-,\eta)q^{\Gamma^-},$$
    and 
    $$\phi^{k,l}(q^{\Gamma^+})=\sum_{A,|\Gamma^-|=l} \frac{T^{\int_A\overline{\omega}}}{\mu_{\Gamma^-}\kappa_{\Gamma^+}}\#\overline{\cM}_{W,A}(\Gamma^+,\Gamma^-)q^{\Gamma^-}.$$
    Then we have 
    $$\widehat{p}_{\mc}\circ \widehat{\phi}_{\bullet}(q_{\gamma_1}\odot\ldots \odot q_{\gamma_k})\ne 0 \in \Lambda \in \overline{EV\otimes \Lambda}$$
    where $\widehat{\phi}_{\bullet}$ is assembled from $\phi,\phi_{\bullet}$ as in \cite[before Definition 2.22]{moreno2024landscape}. The is from analyzing the boundary of $\overline{\cM}_{W,A}(\{\gamma_1,\ldots,\gamma_k\},\Gamma,\eta)$ of virtual dimension $1$. Most of the top-level breaking has virtual count zero, except for $\#\overline{\cM}(\{\gamma_i\}_{1\le i \le k},\emptyset,o)$, as $\#\overline{\cM}(\{\gamma_i\}_{1\le 1\le k},\emptyset,o')$ must be empty on $\widehat{Y}_1\cup \widehat{Y}_2$. The counting of other type of breakings, i.e.\ breakings from the bottom-level, is precisely $\widehat{p}_{\mc}\circ \widehat{\phi}_{\bullet}(q_{\gamma_1}\odot\ldots \odot q_{\gamma_k})$, and hence $\widehat{p}_{\mc}\circ \widehat{\phi}_{\bullet}(q_{\gamma_1}\odot\ldots \odot q_{\gamma_k})=T^{\sum_{i=1}^k \int \gamma_i^*\alpha}\#\overline{\cM}(\{\gamma_i\}_{1\le i \le k},\emptyset,o)\ne 0 \in \Lambda$. That is $\widehat{p}_{\mc}$ for $\partial_-W$ has finite algebraic planar torsion, and by \cite[Proof of theorem 1.1]{moreno2024rsft}, $\partial_+W$ has finite algebraic planar torsion for the undeformed $BL_\infty$ structure. 
\end{proof}

\begin{remark}
  There are several generalizations of \Cref{def:torsion_III}, where analogs of \Cref{thm:torsion_III} hold.
  \begin{enumerate}
      \item We can replace the point class in \Cref{def:strongly_not_cofillable} by a non-trivial rational homology class $A$ in $Y$, or with tangency conditions, considered in \cite[\S 5.1]{moreno2024rsft}. Then an analogous torsion cobordism is a connected cobordism $W$ with convex boundary $Y\sqcup Y'$ and a rational homology class $A'$ of $Y'$, such that $A+A'$ is zero in the homology of $W$.
      \item One may replace $Y$ in \Cref{def:strongly_not_cofillable} by stable Hamiltonian manifolds that are uniruled in a similar way. Note that holomorphic curves in stable Hamiltonian structures can have no positive punctures, and one has to be careful with the maximum principle for holomorphic curves. A typical case is $(M\times \CP^1,\alpha_M,\rd \alpha_M +\omega_{\CP^1})$ for a contact manifold $(M,\xi=\ker \alpha_M)$. In the analogous torsion cobordism, we have $\partial_+W=M\times \CP^1 \sqcup Y'$, where $Y'$ could be (1) contact manifolds, (2) stable Hamiltonian structures without genus zero curve with no positive punctures, e.g. $(M'\times \Sigma_{g},\alpha_{M'},\rd \alpha_{M'}+\omega_{\Sigma_g})$ for genus $g>0$ closed surface $\Sigma_g$ and contact manifold $(M',\alpha_{M'})$. The former case applies to the higher-dimensional Giroux torsion considered in \cite{MNW} and the latter case applies to the examples in \cite[Theroem 1.3]{moreno2019algebraic}. We will review and discuss them briefly in \S \ref{ss:dim5}. 
  \end{enumerate}
\end{remark}

\subsection{Producing algebraic torsion from cobordism}
For this, we will only consider torsion cobordism in the sense of \Cref{def:torsion_cobordism} with $Y=\emptyset$, and analogous stories for torsion cobordisms of type II and III can be derived similarly with suitable assumptions.

\begin{proposition}\label{prop:curve_genus}
Under the assumptions and notations in \Cref{lemma:curve_on_product}, we also assume that $c^{\Q}_1(\Sigma)=0$, $\Sigma$ is Weinstein, $\dim_{\C}\Sigma \ge 3$, and $\ind(p_i)\ge \dim_{\C}\Sigma-1$, then for $g\ge 1$, we have
\begin{enumerate}
    \item $\overline{\cM}_g(\gamma_{p_i},\Gamma)=\emptyset$, when $\Gamma\ne \emptyset$;
    \item $\overline{\cM}_g(\gamma_{p_i},\Gamma,\mu)=\emptyset$, for any $\Gamma$. 
\end{enumerate}
\end{proposition}
\begin{proof}
If $\ind(p_i)=\dim_{\C}\Sigma$, then $\gamma_{p_i}$ has the minimal period on $\partial(\Sigma \times \D)$. As a consequence, $\Gamma$ must be $\emptyset$ by action reasons. For $\overline{\cM}_{g}(\gamma_{p_i},\emptyset,\mu)$, the virtual dimension $-2(\dim_{\C}\Sigma-2)g<0$ and can be assumed empty for generic $J$. If $\ind(p_i)=\dim_{\C}\Sigma-1$, then $\Gamma$ can only be $\{\gamma_q\}$ with $\ind(q)=\dim_{\C}\Sigma$ if not empty. In this case, 
$$\virdim \overline{\cM}_g(\gamma_{p_i},\gamma_q)=-2(\dim_{\C}\Sigma-2)g<0;$$
$$\virdim \overline{\cM}_g(\gamma_{p_i},\gamma_q,\mu)=-2(\dim_{\C}\Sigma-2)g-\dim_{\C}\Sigma<0.$$
Hence, we can assume that they are empty. 
\end{proof}
When $\dim_{\C} \Sigma = 2$ and Weinstein, we may assume $\ind(p_i)=2$. Since $\gamma_{p_i}$ has the minimal period, we must have $\Gamma=\emptyset$. However, we may not be able to rule out $\overline{\cM}_g(\gamma_{p_i},\emptyset,\mu)$, whose virtual dimension is zero. But the appearance of such moduli spaces will not affect the proof of finite algebraic torsion. 

\begin{remark}
    Using the adiabatic limit argument in \cite[\S 2]{moreno2019algebraic}, one may argue that $\overline{\cM}_{g}(\gamma_{p_i},\Gamma)=\emptyset$ for $g\ge 1$ and $\ind(p_i)=\dim_{\C}\Sigma-1$ even for the $\dim_{\C}\Sigma=2$ case. But for our purpose, the simple index and energy considerations above are sufficient. 
\end{remark}

\begin{theorem}\label{thm:finite_AT_k}
    Let $W$ be a $k$-torsion cobordism from $Y_-$ to $\partial(\Sigma \times \D)$, where $\Sigma$ is Weinstein and $c_1^{\Q}(\Sigma)=0$. We assume one of the following conditions holds:
    \begin{enumerate}
        \item\label{AT1} $\dim_{\C}\Sigma \ge 3$ and $A\in H_{\dim_{\C}}(\Sigma;\Q)$ or $H_{\dim_{\C}-1}(\Sigma;\Q)$;
        \item\label{AT2} $\dim_{\C}\Sigma = 2$, and $A\in H_{2}(\Sigma;\Q)$.
    \end{enumerate}
    Then $Y_-$ has a $(0,k)$ torsion. In particular, $\APT(Y_-),\AT(Y_-)\le k$ by \Cref{prop:0-k}.
\end{theorem}
\begin{proof}
     Let $x$ be the coefficient of $\prod_{i=1}^k t_i$ in 
    $$\left(\sum_i a_i \sum_{g, v, A,|\Gamma|\ge 1} \frac{T^{\int_A \overline{\omega}}}{\mu_{\Gamma}\kappa_{\Gamma}}\#\overline{\cM}_{v,g,A,W}(\gamma_{p_i},\Gamma, \mu)q^{\Gamma}t^v\hbar^g\right)\odot e^{\hat{\mc}},$$
    where $\mc$ is the $IBL_\infty$ Maurer-Cartan element from the torsion cobordism and $\hat{\mc}=\oplus_{v\in \N^k\backslash \{0\}}\mc_vt^v$. Then, similar to the proof of \Cref{thm:finite_APT_k}, by looking at the coefficient of $\prod_{i=1}^kt_i$ and specializing at $T=1$ of 
    $$\left(\sum_i a_i \sum_{g,v, A, [\Gamma]|} \frac{T^{\int_A \overline{\omega}}}{\mu_{\Gamma}\kappa_{\Gamma}}\#\partial\overline{\cM}^{\virdim =1}_{v,g,A,W}(\gamma_{p_i},\Gamma, \mu)q^{\Gamma}t^v\hbar^g\right)\odot e^{\hat{\mc}}=0\in \overline{EV[[\hbar]]\otimes \Lambda}.$$
    Combining with \Cref{prop:curve_genus} and the discussion afterwards when $\dim_{\C}\Sigma =2$, we see that 
    $$\widehat{p}(x|_{T=1})=\sum a_i \sum_{g\ge 0}\#\overline{\cM}_g(\gamma_{p_i},\emptyset,\mu)\hbar^g=A_0+\hbar\cdot A_1$$
    with $A_0=\sum a_i \#\overline{\cM}(\gamma_{p_i},\emptyset,\mu)\ne 0\in \Q$. By the intersection number argument as in the proof of \Cref{thm:finite_APT_k}, we have $x|_{T=1}\in E^{k+1}V[[\hbar]]$ and $Y_-$ has a $(0,k)$ torsion. 
\end{proof}

\subsection{A unified perspective}\label{ss:uni}
Finally, we sketch a unified perspective for all of the three torsion cobordisms. Following \cite[\S 5.1.1]{moreno2024rsft}, we can count holomorphic curves passing through a submanifold in $\{0\}\times Y$, or more generally the stable manifold of an axillary Morse function on $\{0\}\times Y$. This gives a chain map that descends to homology is 
$$\eta: H_*(E^kV)\to H_*(E^kV)\otimes H^{-*}(Y)$$
Here we use homology grading and a suitable grading shift is needed but omitted. \cite[\S 5.1.1]{moreno2024rsft} considered the linearized version of this map. Now given an exact cobordism $W$ from $Y_-$ to $Y_+$, by considering a Morse function $f$ on $W$ whose gradient flow is transverse to $\partial W$ and points outward, we may consider curves in $W$ with a marked point passing through the stable manifolds of $f$ to obtain a map
$$H_*(E^kV_+)\to H_*(E^kV_-)\otimes H^{-*}(W)$$
Then they form a commutative diagram on homology 
$$\xymatrix{
H_*(E^kV_+) \ar[r]^{\eta} \ar[d] & H_*(E^kV_+)\otimes H^{-*}(Y_+)\ar[d] \\
H_*(E^kV_-)\otimes H^{-*}(W)\ar[r] & H_*(E^kV_-)\otimes H^{-*}(Y_+)
}$$
where the bottom map is the restriction of $H^*(W)$ to $H^*(Y)$ and the right map is the cobordism morphism $E^kV_+ \to E^kV_-$. To see the diagram commutes, following a similar argument to \cite[\S 3.1]{zbMATH07367119}, the composition $H_*(E^kV_+) \to H_*(E^kV_-)\otimes H^{-*}(W) \to H_*(E^kV_-)\otimes H^{-*}(Y_+)$, on the chain level, is homotopic to counting curves in $W$ with a marked point passing through the stable manifolds of a Morse function on $Y_+\subset W \subset \widehat{W}$. On the other hand, the composition $H_*(E^kV_+) \to H_*(E^kV_+)\otimes H^{-*}(Y_+) \to H_*(E^kV_-)\otimes H^{-*}(Y_+)$ is obtained from applying neck-stretching along $Y_+$ to the same moduli space above. Hence the claim should follow. When $X$ is a strong cobordism, we have an analogous diagram for Maurer-Cartan deformed homologies:
$$\xymatrix{
H_*(\overline{E^kV_+},\widehat{p}_+) \ar[r]^{\eta} \ar[d] & H_*(\overline{E^kV_+},\widehat{p}_+)\otimes H^*(Y_+)\ar[d] \\
H_*(\overline{E^kV_-},\widehat{p}_{-,\mc})\otimes H^*(W)\ar[r] & H_*(\overline{E^kV_-},\widehat{p}_{-,\mc})\otimes H^*(Y_+)
}$$
\begin{proposition}
    If $1\otimes \alpha\in \Ima \eta$ and $\alpha$ is not in the image of $H^*(W)\to H^*(Y)$, then $Y_-$ has finite $\APT$. 
\end{proposition}
\begin{proof}
    By the diagram $1\otimes \alpha $ can not be in the image of $H_*(\overline{E^kV_+},\widehat{p}_+)\to H_*(\overline{E^kV_-},\widehat{p}_{-,\mc})\otimes H^*(Y_+)$, unless $1=0$ in $H_*(\overline{E^kV_-},\widehat{p}_{-,\mc})$. In context of RSFT, by the proof of \cite[Theoerem 1.1]{moreno2024rsft}, $\APT(Y_-)<\infty$. If $X$ is exact, then $\APT(Y_-)\le k-1$. 
\end{proof}

\begin{example}
    For torsion cobordism of type I or II, \Cref{prop:curve,prop:curve_flexible} show that there exists $\alpha\in H^k(\partial_+W)$, such that $1\otimes \alpha$ is in the image of $\eta$, while $\alpha$ is not from $H^*(W)$. For torsion cobordism of type III, the assumption implies that $1\otimes (1,0)$ is in the image of $\eta$, while $(1,0)\in H^0(Y_1\sqcup Y_2)$ is not from $H^0(W)$ when $W$ is connected.  
\end{example}

\section{Lower bounds}\label{s:Lower}
The upper bound for $\APT$ derived in \S \ref{s:torsion} uses the existence of certain curves from the functoriality of RSFT. To obtain lower bounds for $\APT$, we need the non-existence of certain curves. This is typically achieved through constraints from homology and virtual dimensions. In this part, we focus only on special cases for the proofs of Theorems \ref{thm:finite_APT} and \ref{thm:finite_AT}.

\subsection{Reeb dynamics on a spinal open book}
Given a spinal open book (with identical pages) $\SOB(\Sigma,\bm{\phi},\bm{\psi})$, we assume that $\dim \Sigma =2n\ge 4$ and $c_1^{\Q}(\Sigma)=0$. 
\begin{proposition}\label{prop:c1}
    Assume that there exists a trivialization $\Psi: \det_{\C}\oplus^N T\Sigma\to \underline{\C}$ for $N\ge 1 \in \N$, which induces a trivialization $\Psi|_{\partial \Sigma}$ of $\det_{\C}\oplus^N \xi_{\partial \Sigma}$ on $\partial \Sigma$ (where $\xi_{\partial \Sigma}$ is the contact structure on $\partial \Sigma$), such that the trivialization $(\psi_i^{-1}\circ \psi_1)_*\Psi|_{\partial \Sigma}$ can be extended to $\Sigma$ for all $1\le i \le k=|\bm{\phi}|$. Then we have $c_1^{\Q}(\SOB(\Sigma,\bm{\phi},\bm{\psi}))=0$.
\end{proposition}
\begin{proof}
Over the paper region, the contact structure $\xi$ is homotopic to the tangent space of the pages as an almost complex vector bundle. By assumption, we can pick trivializations $\Psi_i=\det_{\C}\oplus^N T\Sigma \to \underline{\C}$ such that $(\psi_i)_*\Psi_i|_{\partial \Sigma} = (\psi_j)_*\Psi_j|_{\partial \Sigma}$ and $\Psi_1=\Psi$. In the spine region, the contact structure is homotopic to $TS_{k}\oplus \xi_{M}$, where $\xi_{M}$ is the contact structure on $M$, which is contactomorphic to $\partial \Sigma$ via $\psi_i^{-1}$. In particular, $(\psi_i)_*\Psi_i$ is a trivialization of $\det_{\C}\oplus^N \xi_M$ independent of $i$. Therefore $\det_{\C}\xi^N$ can be trivialized using a trivialization of $TS_{k}$ and $(\psi_i)_*\Psi_i$ on the spine region and $\Psi_i$ on the paper region. However, these two trivializations may not match over the boundary. The possible mismatch comes from trivializations of $\det_{\C}\oplus^N(\xi/\xi_M)$, i.e., the $TS_k$ part in the spine region. Using a suitable gauge equivalence determined by $\Sigma_{\phi_i}\stackrel{\pi}{\to} S^1 \stackrel{\rho}{\to} U(1)$ to change the trivialization over the mapping tori, we can always compensate the discrepancy to obtain a global trivialization of $\det_{\C}\oplus^N\xi$, i.e., $c_1^{\Q}(\SOB(\Sigma,\bm{\phi},\bm{\psi}))=0$.
\end{proof}

\begin{remark}
    The proof above uses $\dim \Sigma \ge 4$ as well. The fact that $\Psi$ induces $\Psi|_{\partial \Sigma}$ uses the decomposition $T\Sigma|_{\partial \Sigma}=\xi_{\partial \Sigma}\oplus \la X, R \ra$, where $X$ is the Liouville vector field, $R$ is the Reeb vector field, and $\la X,R \ra$ is a framed complex sub-bundle with $JR=X$. This makes sense if $\xi_{\partial \Sigma}\ne 0$, i.e., $\dim \Sigma \ge 4$. When $\dim \Sigma = 2$, there is no trivialization of $\det_{\C}T\Sigma$ such that the restriction to the boundary is the framing from the Liouville/Reeb vector fields. Consequently, the proof of \Cref{prop:c1} fails, and indeed we have many contact $3$-manifolds with $c_1^{\Q}\ne 0$, even though $c_1^{\Q}(\Sigma)=0$ for any open book decomposition. 
\end{remark}

Given a spinal open book $\SOB(\Sigma,\bm{\phi},\bm{\psi})$, we have a natural continuous map $\pi_S:\SOB(\Sigma,\bm{\phi},\bm{\psi})\to S_k$ which maps the spine region to $S_k$ and the paper region to $\partial S_k$. 

\begin{proposition}\label{prop:CZ}
    Let $\alpha_M$ be a contact form on $M$ and let $D>0$ be a number not in the Reeb spectrum of $\alpha_M$ (the set of periods of Reeb orbits of $\alpha_M$). There exists a contact form on $\SOB(\Sigma,\bm{\phi},\bm{\psi})$, which is a slight perturbation and smoothing of \eqref{eqn:TW}, such that the Reeb orbits of period $<D$ are either
    \begin{enumerate}
        \item contained in the paper or part of the spine over the collar neighborhoods of $\partial S_k$, which project to the $m$-fold cover of a boundary component of $S_k$ through $\pi_S$, for $m>0$;
        \item or contained in the spine outside of the collar neighborhoods above, in the form of $\hat{\gamma},\check{\gamma}_1,\ldots,\check{\gamma}_{k-1}$ for each Reeb orbit $\gamma$ of $M$ with period $<D$, where $k=|\bm{\phi}|$. Those orbits are copies of $\gamma \subset M$ over a point in $S_k$.
    \end{enumerate}
    Assume the conditions in \Cref{prop:c1} hold. Then we have $\mu_{\CZ}(\hat{\gamma})=\mu_{\CZ}(\gamma)+1$ and $\mu_{\CZ}(\check{\gamma}_i)=\mu_{\CZ}(\gamma)$, where $\mu_{\CZ}(\gamma)$ is the Conley-Zehnder index computed using $\Psi|_{\partial \Sigma}$ (which are rational in general).
\end{proposition}
\begin{proof}
    We perturb the contact form \eqref{eqn:TW} to $K\lambda_S+f\alpha_M$, where $f$ is a continuous function on $S_{k}$ that is smooth on the interior, such that
    \begin{enumerate}
        \item $f|_{\partial S_{k}}=1$. On the collar $((1-\epsilon,1]_p\times S^1_q, p\rd q)$, i.e., a neighborhood of a component of $\partial S_{k}$, we have that $f$ depends only on $p$ and is strictly decreasing, with $\displaystyle \lim_{p\to 1} \frac{\rd ^k f}{\rd p^k}=-\infty$. 
        \item $\rd f$ is $C^1$-small outside those collar neighborhoods, and Morse with one maximum $o_{\max}$ and $k-1$ saddle points $o_{i,\text{sad}}$ for $1\le i \le k-1$.
    \end{enumerate}
    In the paper region, namely the mapping torus $\Sigma_{\phi_i}$, since $\rd \alpha = \rd \lambda_{i}+\rd \beta(t) \wedge \eta_i$, it is straightforward to check that any vector tangent to the fiber is not contained in the kernel of $\rd \alpha$. Hence, the Reeb vector field projects to the positive direction of $S^1$ via $\pi_S$. Consequently, Reeb orbits in the paper region project to the $m$-fold cover of a boundary component of $S_k$ through $\pi_S$, for $m>0$. 
    
    In the spine region, the computation of the Reeb vector field was carried out in \cite[\S 6.2]{zbMATH07367119}. In particular, the Reeb vector field in the spine region is a rescaling of the Reeb vector field on $(M,\alpha_M)$ plus the Hamiltonian vector field of $1/f$. Reeb orbits in the collar neighborhoods are non-contractible, as in the paper region, because the Hamiltonian vector field of $1/f$ wraps the Reeb direction near the boundary of $\partial S_k$. Finally, outside the collar neighborhoods, by the $C^1$-small assumption, up to an arbitrarily high ($>D$) period threshold depending on the $C^1$ norm of $\rd f$, all Reeb orbits are the lifts of a Reeb orbit $\gamma$ of $M$ over the critical points of $f$. We use $\hat{\gamma}$ to denote the lift of $\gamma$ over $o_{\max}$ and $\check{\gamma}_i$ the lift over $o_{i,\text{sad}}$. 

    Finally, using $\Psi|_{\partial \Sigma}$ to define $\mu_{\CZ}(\gamma)$, computations in \cite[\S 6.2]{zbMATH07367119} imply that $\mu_{\CZ}(\hat{\gamma})=\mu_{\CZ}(\gamma)+1$ and $\mu_{\CZ}(\check{\gamma}_i)=\mu_{\CZ}(\gamma)$ using the framing from the proof of \Cref{prop:c1}, where the extra term comes from the Conley-Zehnder index of the Hamiltonian vector field of $1/f$ at the critical points. 
\end{proof}

\subsection{Lower bounds of $\APT$}
\begin{proposition}\label{prop:filtration}
    Let $\alpha^1_M>\alpha^2_M>\ldots >\alpha^\infty_M$ be a sequence of contact forms on $M$. Then the construction in \Cref{prop:CZ} yields contact forms $\alpha_1>\alpha_2>\ldots$ on $\SOB(\Sigma,\bm{\phi},\bm{\psi})$, with period threshold $D_i$ converging to $\infty$.
\end{proposition}
\begin{proof}
    Without loss of generality, we may assume $\phi_i\in \symp_c(\Sigma_i,\rd \lambda_i)$, where the Liouville form $\lambda_i$ restricted to $\partial \Sigma_i$ is $\psi_i^*\alpha^{\infty}_M$. We fix $K\gg 0$ such that the following holds:
    \begin{enumerate}
        \item $K\rd t + \beta(t)\eta_i+\lambda_i$ is a contact form on $\Sigma_{\phi_i}$;
        \item $K>\beta'(t)\eta(X_i(t))$ for all $t\in S^1$, where $X_i(t)\in T\Sigma$ is defined by $\iota_{X_i(t)}\rd \lambda_i=\lambda_i+\beta(t)\eta_i$ with $\phi_i^*\lambda_i=\lambda_i+\eta_i$.
    \end{enumerate}
    Consider the following domain:
    $$\left(\widehat{S}_k \times M\times [1,\infty)_r \bigcup_{i=1}^k [1,\infty)_p \times \Sigma_{\phi_i}, \quad \lambda:=\left(\widehat{\lambda}_S+r\alpha_M \bigcup_{i=1}^k Kp\rd t + \beta(t)\eta_i+\lambda_i\right)\right).$$
    With the conditions above, $\lambda$ is a Liouville form, and the Liouville vector field $X_\lambda$ is defined by 
    $$\left(X_S+r\partial_r\right) \bigcup_{i=1}^k \left((Kp-\beta'(t)\eta(X_i(t)))\partial_p+X_i(t)\right),$$
    where $X_S$ is the Liouville vector field on the Liouville completion $\widehat{S}_k$. The second condition guarantees that $X_{\lambda}$ has no zeros on the domain above. Note that the contact forms in \Cref{prop:CZ} are the restrictions of $\lambda$ to a smoothing and perturbation of the hypersurface given by $r=h_i>h_{i+1}>1$ for $\alpha^i_M=h_i\alpha^{\infty}_M$ and $p=1+\delta_i$ with $\delta_1>\delta_2>\ldots >0$. The Liouville vector field $X_\lambda$ is transverse to those hypersurfaces, making the region between them a Liouville cobordism, homotopic to the trivial cobordism. Consequently, the restriction of $\lambda$, namely $\alpha_i$, satisfies $\alpha_i>\alpha_{i+1}$ under a suitable identification of the hypersurfaces (i.e., using the flow of $X_{\lambda}$), such that the period threshold $D_i$ converges to $\infty$ by flattening the function $f$ in the perturbations.
\end{proof}

\begin{remark}
    Ideally, we would like to drop the assumption of the existence of $\alpha_M^{\infty}$, as in the concept of asymptotically dynamically convex contact manifolds in \cite{Lazarev}. However, since the monodromy $\phi_i$ cannot concentrate on the skeleton of $\Sigma$, we have to choose representatives $\phi_i^j\in \symp_c(\Sigma_i,\rd \lambda_i^j)$ with Liouville forms $\lambda_i^j$ that restrict to $\alpha^j_M$ under the identification $\psi_j$. Then a suitable asymptotic condition on $(\phi_i^j)^*\lambda_i^j-\lambda_j^i$ is needed to build the trivial cobordism as in the proof of \Cref{prop:filtration}.
\end{remark}

Let $D_i$ be a sequence of increasing positive real numbers with $D_i\to \infty$. By the functoriality of RSFT, to show $\SOB(\Sigma,\bm{\phi},\bm{\psi})$ has algebraic (planar) torsion $\le m$, it suffices to show the torsion exists for $\alpha_i$ from Reeb orbits of period at most $D_i$, where $D_i,\alpha_i$ are from \Cref{prop:CZ,prop:filtration}.

\begin{proposition}\label{prop:lower_bound}
    Under the assumption of \Cref{prop:c1} for $\SOB(\Sigma,\bm{\phi},\bm{\psi})$, let $k=|\bm{\phi}|$.
    Assume that there exists a sequence of contact forms $\alpha_1>\alpha_2>\ldots >\alpha_\infty$ on $M$ and numbers $D_1<D_2<\ldots$ with $D_i\to \infty$, such that all Reeb orbits of $\alpha_i$ with period at most $D_i$ are non-degenerate and have Conley-Zehnder index $<2-\dim_{\C}\Sigma$ using the trivialization $\Psi|_{\partial \Sigma}$ from \Cref{prop:c1}. Then the algebraic planar torsion and algebraic torsion of $\SOB(\Sigma,\bm{\phi},\bm{\psi})$ are at least $k-1$.
\end{proposition}
\begin{proof}
It suffices to consider the contact form on $\SOB(\Sigma,\bm{\phi},\bm{\psi})$ constructed from $\alpha_i$ and Reeb orbits of period at most $D_i$ by \Cref{prop:filtration}. Having algebraic planar torsion or algebraic torsion $m<k$ means that there exists a rigid holomorphic curve with genus $g$ and at most $m+1-g$ positive punctures and no negative punctures. By considering the map $\pi_S:\SOB(\Sigma,\bm{\phi},\bm{\psi})\to S_{k}$, we see that if there is an asymptotic orbit in the paper region or the collar region, then to have the homology of all positive asymptotics sum to zero in $H_1(S_k)$, we must have $m+1-g\ge k$ by \Cref{prop:CZ}. Hence we only need to consider the case where all asymptotic orbits are contained in the spine region outside the collar region. In this case, the expected dimension is 
    $$(n-3)(2-g-(m+1))+\sum_{i=0}^{m-g}\mu_{\CZ}(\gamma_i)-1.$$
By assumption and \Cref{prop:CZ}, the above expected dimension is necessarily negative. This concludes that the algebraic planar torsion and algebraic torsion are at least $k-1$.
\end{proof}

\subsection{Proof of Theorems \ref{thm:finite_APT} and \ref{thm:finite_AT}}
Now, it is straightforward to produce examples with $\APT$ or $\AT$ precisely $k$ in dimensions $2n+1 \ge 5$. Let $\Sigma$ be the Brieskorn variety given by $z_0^{a}+\ldots+z_n^a =\epsilon \ne 0$ for $a\gg 0$. For $a\gg 0$, the Conley-Zehnder indices of Reeb orbits on the boundary of $\Sigma$ are sufficiently negative by \cite[(14)]{zbMATH06571536}. Therefore, let $L\subset \Sigma$ be a Lagrangian sphere with non-trivial rational homology (which exists in abundance when $a\gg 0$). Then $\SOB(\Sigma, \{\phi_L^{-1},\Id, \ldots, \Id\}, \{\Id,\ldots,\Id\})$ with $k+1$ paper regions has $\APT\le k$ by \Cref{thm:finite_APT_k} and $\AT\le k$ by \Cref{thm:finite_AT_k} ($\Sigma$ is Weinstein with trivial first Chern class and the homology class $A=[L]$ is of middle degree). By \Cref{prop:lower_bound}, we have $\APT,\AT\ge k$. Therefore, $\APT$ and $\AT$ are precisely $k$. Moreover, using the exact embedding $S_k\subset S_{k+1}$, in any odd dimension $\ge 5$, we obtain a sequence of contact manifolds $Y_0,Y_1,\ldots$ such that $\APT(Y_k)=\AT(Y_k)=k$ and there is an exact cobordism from $Y_{k-1}$ to $Y_k$, i.e., the higher-dimensional analogs of \cite[Theorem 1, Remark 1.6]{LW}.

To prove Theorems \ref{thm:finite_APT} and \ref{thm:finite_AT}, we need to cook up examples that are weakly fillable. For this we need the Bourgeois construction from \Cref{ex:Bourgeois}. We also need to discuss dimension $5$ and dimensions $\ge 7$ separately.
\subsubsection{Dimensions $\ge 7$ case}\label{sss:dim7}
Let $f=z_0^{a}+\ldots+z_n^a$ for $n\ge 2$ and $a\gg 0$. Then the Milnor fibration gives $(S^{2n+1},\xi_{\std})$ an open book with page $\Sigma$ given by the Brieskorn variety $z_0^{a}+\ldots+z_n^a=\epsilon\ne 0$ and monodromy $\phi$. By \eqref{ex:BO=SOB}, we have $$\BO(\Sigma,\phi)=\SOB(\Sigma \times D^*S^1, \{\Id,\Id\},\{\phi_{\BO},\Id\}).$$ 
By taking the Liouville connected sum with $\BO(\Sigma,\Id)$ as in \Cref{prop:SOB_sum}, we obtain a Weinstein cobordism $W$ from 
$\BO(\Sigma,\phi) \sqcup_{i=1}^{k-1} \BO(\Sigma,\Id)$
to $\SOB(\Sigma \times D^*S^1, \{\Id,\ldots,\Id\}, \{\phi_{\BO},\Id,\ldots,\Id\})$ with $k+1$ paper regions. 

\begin{claim}\label{claim}
    The map $H^2(W;\Q)\to H^2(\BO(\Sigma,\phi)=S^{2n+1}\times T^2;\Q)$ is surjective.
\end{claim}
\begin{proof}
Dually, it is equivalent to show that the fundamental class of the $T^2$-factor in $\BO(\Sigma,\phi)$ maps to a non-trivial class in $H_2(W;\Q)$. We consider the case $k=2$; the general case is similar. By \cite[\S 4.1]{Russell}, topologically the symplectic cobordism is obtained by attaching $\Sigma \times D^*S^1 \times [0,1]^2$ to $\BO(\Sigma,\phi)\sqcup \BO(\Sigma,\Id)$ along neighborhoods of two Liouville hypersurfaces $\Sigma \times D^*S^1\times [0,1]\times \{0,1\}$. Write $V:=\Sigma \times D^*S^1$. By the Mayer–Vietoris sequence, we have the exact sequence
$$H_2(V \times [0,1]\times \{0,1\};\Q)\to H_2(\BO(\Sigma,\phi)\sqcup \BO(\Sigma,\Id)\sqcup V\times [0,1]^2;\Q)\to H_2(W;\Q).$$
Since $H_2(V \times [0,1]\times \{0,1\};\Q)\to H_2(\BO(\Sigma,\phi);\Q)$ is zero, the claim follows.
\end{proof}

By \cite[Theorem A]{zbMATH07162211}, $\BO(\Sigma,\phi)$ is weakly fillable such that the symplectic form restricted to $\BO(\Sigma,\phi)$ is the positive generator of $H^2(\BO(\Sigma,\phi);\Q)=H^2(T^2;\Q)$. Therefore, by the claim above, $\SOB(\Sigma \times D^*S^1, \{\Id,\ldots,\Id\}, \{\phi_{\BO},\Id,\ldots,\Id\})$ is also weakly fillable by \Cref{lemma:weak_surgery}. Moreover, by \Cref{thm:finite_APT_k}, the algebraic planar torsion is at most $k$.

To obtain a lower bound, we wish to apply \Cref{prop:lower_bound}. The page $\Sigma \times D^*S^1$ is Weinstein of dimension at least $6$. In particular, $H^1(\Sigma \times D^*S^1;\Z)\to H^1(\partial(\Sigma \times D^*S^1);\Z)$ is an isomorphism; hence all trivializations along the boundary extend to the interior. Therefore \Cref{prop:c1} applies. However, although the Reeb orbits on the boundary of $\Sigma$ have sufficiently small Conley-Zehnder indices, this is not the case for the $D^*S^1$-factor. In particular, $\partial(\Sigma\times D^*S^1)$ cannot have arbitrarily small Conley-Zehnder indices for all Reeb orbits. To counteract this issue, we need to modify the construction so that Reeb orbits with non-small Conley-Zehnder indices have non-trivial homology.

Let $\phi_{\ell}$ be a contactomorphism on $\partial(\Sigma \times S_{\ell})$ such that $\phi_{\ell}$ equals $\phi^{-1}$ on $\Sigma \times S^1\subset \partial(\Sigma \times S_{\ell})$ and is the identity elsewhere. Then $\phi_{\BO}=\phi_2$. 

\begin{claim}
     For $\ell\ge 2$, the manifold $\SOB(\Sigma \times S_{\ell}, \{\Id,\ldots,\Id\}, \{\phi_{\ell},\Id,\ldots,\Id\})$ with $k+1$ paper components is weakly fillable. 
\end{claim}
\begin{proof}
We start with the $k=2$ case. When $\ell=2$, it is the Bourgeois construction, so assume $\ell\ge 3$. We have an exact embedding $\Sigma \times S_2\subset \Sigma \times S_{\ell}$. Let $V$ denote the cobordism from deleting the interior of $\Sigma \times S_2$. We require the inclusion $S_2\subset S_{\ell}$ to fix the boundary component where $\phi_{\BO}$ is non-trivial. Then it is straightforward to check that $\phi_{\BO}$ extends to a symplectomorphism $\phi_V$ such that $\phi_V|_{\partial(\Sigma\times S_{\ell})}$ is isotopic to $\phi_{\ell}$. Let $\widehat{V}_+$ denote the completion of $(V,\lambda_V)$ in the \emph{positive} direction. Consider the following domain $U$:
$$\widehat{D^*S^1}\times \widehat{V}_+ \cup (1,\infty)_p \times S^1_t \times (\Sigma\times S_2) \cup_{\phi_{\BO}} (1,\infty)_p \times S^1_t \times (\Sigma\times S_2),$$
with Liouville form
$$(\lambda_{D^*S^1}+\lambda_V)\cup (p\rd t + \lambda_{\Sigma\times S_2})\cup_{\phi_{\BO}}(p\rd t +\lambda'_{\Sigma\times S_2}),$$
where $\lambda_{\Sigma\times S_2}$ is a Liouville form on $\Sigma \times S_2$ extending $\lambda_V$ restricted to the negative boundary of $V$, and $\lambda'_{\Sigma\times S_2}$ is a Liouville form on $\Sigma \times S_2$ such that its restriction to the boundary is $\phi_{\BO}^*(\lambda_{\Sigma \times S_2}|_{\partial(\Sigma \times S_2)})$. Then a smoothing of the interior boundary of $U$ is precisely the Bourgeois construction $\SOB(\Sigma \times S_{2}, \{\Id,\Id\}, \{\phi_{\BO},\Id\})$. The outer ideal boundary of $U$ is contactomorphic to $\SOB(\Sigma \times S_{\ell}, \{\Id,\Id\}, \{\phi_{\ell},\Id\})$ by the extension property explained at the beginning. Therefore, $U$ is an exact cobordism from $\SOB(\Sigma \times S_{2}, \{\Id,\Id\}, \{\phi_{\BO},\Id\})$ to $\SOB(\Sigma \times S_{\ell}, \{\Id,\Id\}, \{\phi_{\ell},\Id\})$. Note that the $T^2$-factor of $\BO(\Sigma,\phi)=\SOB(\Sigma \times S_{2}, \{\Id,\Id\}, \{\phi_{\BO},\Id\})$ maps to a non-trivial class in $H_2(U;\Q)$ by the Mayer–Vietoris sequence. As a consequence of \Cref{lemma:weak_surgery}, $\SOB(\Sigma \times S_{\ell}, \{\Id,\Id\}, \{\phi_{\ell},\Id\})$ is weakly fillable. The cobordism from the Liouville connected sum, i.e., gluing $(\Sigma \times S_{\ell})\times D^*I$, does not eliminate the homology class from the $T^2$-factor, as in the proof of Claim \ref{claim}. In particular, \Cref{lemma:weak_surgery} implies that $\SOB(\Sigma \times S_{\ell}, \{\Id,\ldots,\Id\}, \{\phi_{\ell},\Id,\ldots,\Id\})$ is weakly fillable as before.
\end{proof}

\begin{claim}
For $\ell > k+1$ and $a\gg 0$, the manifold $\SOB(\Sigma \times S_{\ell}, \{\Id,\ldots,\Id\}, \{\phi_{\ell},\Id,\ldots,\Id\})$ with $k+1$ paper components has $\APT=\AT=k$.
\end{claim}
\begin{proof}
Let $\pi$ denote the projection $\Sigma \times S_{\ell} \to S_{\ell}$ and $\pi_{\partial}$ its restriction to the contact boundary $\partial(\Sigma\times S_{\ell})$. Note that $\pi_{\partial}\circ \phi_{\ell}=\pi_{\partial}$. Consequently, $\pi$ induces a map $$\pi_{\ell}: \SOB(\Sigma \times S_{\ell}, \{\Id,\ldots,\Id\}, \{\phi_{\ell},\Id,\ldots,\Id\}) \to S_{\ell}.$$ In view of \Cref{prop:CZ}, we only need to consider the case of a smaller torsion coming from Reeb orbits in the spine region, namely $\check{\gamma},\hat{\gamma}_1,\ldots,\hat{\gamma}_{k-1}$ lifting a Reeb orbit on $\partial(\Sigma \times S_{\ell})$. Note that $\partial(\Sigma \times S_{\ell})$ is again a spinal open book, and Reeb orbits in the paper region as well as part of the spine near the boundary of $S_{\ell}$ cannot contribute to a torsion of order smaller than $\ell-1$ for homology reasons via the map $\pi_{\ell}$. Finally, the Reeb orbits in the interior of the spine of $\partial(\Sigma \times S_{\ell})$ have sufficiently small Conley-Zehnder indices when $a\gg 0$. Therefore, we have $\APT, \AT\ge k$. Finally, the strong cobordism by capping off $k$ paper regions to $\partial(\Sigma \times \S_{\ell}\times \D)$ annihilates a rational homology class in degree $\dim_{\C}\Sigma$ by \eqref{DG2} of \Cref{ex:DG}. By \Cref{thm:finite_APT_k} and \Cref{thm:finite_AT_k} (case \eqref{AT1}), the claim follows. 
\end{proof}
Consequently, we produce infinitely many examples (e.g., by increasing $\ell$ or $a$ for a fixed $k$) in all odd dimensions $\ge 7$ such that Theorems \ref{thm:finite_APT} and \ref{thm:finite_AT} hold.

\subsubsection{Dimension $5$ case}
For the dimension $5$ case, we cannot use the open books on $(S^3,\xi_{\std})$ as above. The open book from $z^a_0+z_1^a$ is a higher-genus hypersurface with many boundary components, and the computation of Conley-Zehnder indices breaks down. Instead, we consider an open book on the prequantization circle bundle $Y_g$ over the Riemann surface $\Sigma_g$ of genus $g$ of degree $-2$. The open book is given by a meromorphic section with two simple poles. 
\begin{enumerate}
    \item The page $\Sigma_{g,2}$ is a twice-punctured genus $g$ surface, and the monodromy $\phi$ is the Dehn twist along the boundary circles. 
    \item The inclusion of the page $\Sigma_{g,2}\hookrightarrow Y_g$ induces a non-injective map $H_1(\Sigma_{g,2};\Q)\to H_1(Y_g;\Q)$. In particular, capping off by a round handle yields a torsion cobordism from $\BO(\Sigma_{g,2},\phi)$ to $\partial(\Sigma_{g,2}\times D^*S^1 \times \D)$. Similar to the construction in \S \ref{sss:dim7}, $\SOB(\Sigma_{g,2}\times S_{\ell}, \{\Id,\ldots,\Id\}, \{\phi_{\ell},\Id,\ldots,\Id\})$ is the negative boundary of a torsion cobordism. By \Cref{thm:finite_APT_k}, we have
    $$\APT(\SOB(\Sigma_{g,2}\times S_{\ell}, \{\underbrace{\Id,\ldots,\Id}_{k+1}\}, \{\underbrace{\phi_{\ell},\Id,\ldots,\Id}_{k+1}\}))\le k.$$ 
    To consider $\AT$ using \Cref{thm:finite_AT_k} case \eqref{AT2}, it suffices to prove that a second rational homology class of $\Sigma_{g,2}\times S_{\ell}$ is annihilated in the capping-off strong cobordism from the spinal open book. It is sufficient to look at the capping-off cobordism from $\SOB(\Sigma_{g,2}\times S_{\ell}, \{\Id,\Id\},\{\phi_{\ell}, \Id\})$ to $\partial(\Sigma_{g,2}\times S_{\ell} \times \D)$. In particular, it suffices to look at the piece over a cotangent fiber of the vertebra $D^*S^1$, namely, it is sufficient to prove 
    $$H_2(\Sigma_{g,2}\times S_{\ell};\Q)\to H_2((\Sigma_{g,2}\times S_{\ell})\cup_{\phi_{\ell}}(\Sigma_{g,2}\times S_{\ell});\Q) \text{ is not injective.}$$
    This follows from \Cref{lemma:not_inj}.
    \item\label{index} Let $\gamma_1,\gamma_2$ be the two binding orbits of the open book. Note that $c_1^{\Q}(Y_g)=0$ and $\gamma_1,\gamma_2$ have torsion homology classes; their rational Conley-Zehnder indices are well-defined. The Conley-Zehnder indices of $\gamma^m_{1,2}$ are at most $1-2g$ when the binding region is sufficiently flat, i.e., the function $f$ in the proof of \Cref{prop:CZ} is sufficiently flat outside the collar of the binding region. 
    \item Since $Y_g$ is strongly fillable as a prequantization bundle, $\BO(\Sigma_{g,2},\phi)=Y_g\times T^2$ is weakly fillable with a symplectic form $\omega$ that restricts to a volume form on the $T^2$-factor in cohomology. In an exact cobordism from $\BO(\Sigma_{g,2},\phi)$, to ensure that the weak filling can be glued with the exact cobordism via \Cref{lemma:weak_surgery}, we need the $T^2$-factor to map injectively into the rational homology of the cobordism by the universal coefficient theorem. Therefore, similar to the construction in \S \ref{sss:dim7}, $\SOB(\Sigma_{g,2}\times S_{\ell}, \{\Id,\ldots,\Id\}, \{\phi_{\ell},\Id,\ldots,\Id\})$ is weakly fillable. 
    \item Finally, similar to \S \ref{sss:dim7}, when $\ell, g \gg 0$, a smaller torsion ($<k$) must come from lifts of $\gamma^m_1,\gamma^m_2$, whose Conley-Zehnder indices are sufficiently small for $g\gg 0$, which is impossible by dimensional counting. 
\end{enumerate}
Consequently, we produce infinitely many examples (e.g., by increasing $\ell$ or $g$ for a fixed $k$) in dimension $5$ such that Theorems \ref{thm:finite_APT} and \ref{thm:finite_AT} hold.

\begin{lemma}\label{lemma:not_inj}
$H_2(\Sigma_{g,2}\times S_{\ell};\Q)\to H_2((\Sigma_{g,2}\times S_{\ell})\cup_{\phi_{\ell}}(\Sigma_{g,2}\times S_{\ell});\Q)$ is not injective. 
\end{lemma}
\begin{proof}
    Let $b_1$ and $b_2$ be two oriented boundaries of $\Sigma_{g,2}$, then $[b_1]\ne 0 \in H_*(\Sigma_{g,2};\Q)$ and $[b_1]+[b_2]=0$. We use $c_1,\ldots,c_{\ell}$ to denote the boundary component of $S_{\ell}$, where $\phi_{\ell}$ is nontrivial on $\Sigma_{g,2}\times c_1$. Let $I$ be an oriented arc in $\Sigma_{g,2}$ from a point in $b_2$ to a point in $b_1$. Then $a=[\partial(I\times S_{\ell})]$ is a second homology class of $\partial(\Sigma_{g,2}\times S_{\ell})$ that is zero in $H_*(\Sigma_{g,2}\times S_{\ell};\Q)$. Note that the $\phi_{\ell}$ action on $I\times c_1\subset \Sigma_{g,2}\times c_1$ picks up $-2[b_1\times c_1]$ in homology as the boundary Dehn twists on $I$ picks up $2[b_1]$ and $\phi_{\ell}$ is inverse of the boundary Dehn twists on each slice of $\Sigma_{g,2}\times c_1$. As a consequence, we have $(\phi_{\ell})_*a=a-2[b_1\times c_1]$. Note that $[b_1\times c_1]\ne 0\in H_2(\Sigma_{g,2}\times S_{\ell};\Q)$. Then by Mayer-Vietoris similar to \eqref{DG2} of \Cref{ex:DG}, we have $[b_1\times c_1]$ is annihilated in $(\Sigma_{g,2}\times S_{\ell})\cup_{\phi_{\ell}}(\Sigma_{g,2}\times S_{\ell})$.
\end{proof}

\section{Comparisons and other applications}\label{s:comparison_application}
In this part, we will justify \Cref{thm:everything} by reviewing examples from the literature. 
\subsection{Examples in dimension $3$}
Even though the main focus of this paper is higher dimensional contact manifolds, we will first review briefly examples in dimension $3$ that are not strongly/weakly fillable, some of which can be put under the perspectives in this paper. There is a large literature on such examples, and our list is far from complete.
\begin{example}
    Contact manifolds Giroux torsion, e.g.\ the first example of weakly not strongly fillable contact manifolds \cite{Eli96}, or more generally, planar $k$-torsion \cite[Definition 2.13]{zbMATH06218377}, admit strong cobordisms to an overtwisted contact manifold \cite[Theorem 1]{Wen13}. Those manifolds are not strongly fillable \cite[Theorem 1]{zbMATH06218377}, and not weakly fillable if they have fully separating planar $k$-torsion \cite[Corollary 3]{zbMATH06026209}. By \cite[Theorem 6]{LW}, those manifolds have algebraic torsion at most $k$ and fully twisted algebraic torsion at most $k$ if the planar torsion is fully separating. As explained in \cite[Theorem 3.17]{moreno2024landscape}, those manifolds also have algebraic planar torsion $k$, and fully twisted algebraic torsion $k$, respectively. In view of \Cref{ex:ot}, those manifolds arise as the concave boundaries of torsion cobordisms of type I. When the planar torsion is fully twisted, by \cite[Theorem 2]{Wen13}, the torsion cobordism will satisfy conditions in \Cref{prop:twisted}, hence the concave boundary has fully twisted finite algebraic torsion by the results in this paper. 
\end{example}

\begin{example}
    The first example of tight and not weakly fillable contact manifolds in \cite{EtnHon} used filling obstructions in \cite{Lisca} derived from Seiberg-Witten theory. Similar filling obstructions were also used in \cite{zbMATH02105207} to obtain, for example, a contact $+1$ surgery applied to the right–handed trefoil knot with the maximal Thurston–Bennequin number in $(S^3,\xi_{std})$ yields tight contact manifold that is not weakly fillable. In \cite{zbMATH07699409}, such a contact manifold is shown to be algebraically overtwisted, i.e.\ with algebraic (planar) torsion $0$. In \cite{AOT}, we showed that the contact homology of $S^{2n+1}_{+1}(\Lambda)$ is zero, where $\Lambda$ is a Legendrian sphere in $(S^{2n+1},\xi_{std})$ such that the Thurston–Bennequin number is not $(-1)^{\frac{n(n-1)}{2}+1}$. The proofs in \cite{AOT} were very different from those in this paper, in particular, there is no obvious torsion cobordism for examples in \cite{AOT}. In the example of torsion cobordisms from $+1$ surgeries considered in this paper, we only use the homology class of $\Lambda$, while proofs in \cite{AOT} used Legendrian properties of $\Lambda$ in a more essential way.  
\end{example}

\begin{example}
    By \cite[Theorem 3]{LW}, connected $\DG(\Sigma_+,\Sigma_-,\psi)$ has algebraic torsion $1$ if either $\Sigma_+$ or $\Sigma_-$ is disconnected, where $\Sigma_{\pm}$ are Liouville surfaces. Then there exist contact $3$-manifolds that have algebraic torsion $1$ but no Giroux torsion.  Assuming $\Sigma_+$ is disconnected, by capping off $\Sigma_+\times S^1$ as in \Cref{ex:S1}, we get a $1$-torsion cobordism to $\partial(\Sigma_-\times \D)$. Then by \Cref{thm:finite_APT_k}, we have $\APT(\DG(\Sigma_+,\Sigma_-,\psi))\le 1$.
\end{example}

\subsection{Examples in a general dimension $\ge 5$}\label{ss:dim5}
In the following, we will give a complete list of examples in higher dimensions ($\ge 5$) without strong/weak fillings to justify \Cref{thm:everything}.

\begin{example}[Algebraically overtwisted contact manifolds]
To the best of the author's knowledge, the following is the list of all known examples of contact manifolds with vanishing contact homology, i.e.\ $\APT=\AT=0$.
\begin{enumerate}
    \item Overtwisted contact manifolds, when viewed as in \Cref{ex:ot}, can always be realized as the negative boundary of a Liouville torsion cobordism. 
    \item Niederkr\"uger  \cite{Nie06} introduced the concept of plastikstufe, which is a higher-dimensional generalization of an overtwisted disk in dimension $3$. The bordered Legendrian open book (bLob) was then introduced by Massot-Niederkr\"uger-Wendl \cite{MNW} as a further generalization. Contact manifolds that contain those structures are called PS-overtwisted, which are generalizations of overtwisted contact manifolds by the work of Casals-Murphy-Presas \cite[Theorem 1.1-(4)]{zbMATH07010365}. By an unpublished work of Bourgeois and Niederkr\"uger sketched in  \cite[Theorem 4.10]{zbMATH05635066}, PS-overtwisted contact manifolds also have vanishing contact homology. The idea is considering the degeneration of the Bishop family of holomorphic disks in the symplectization, as opposed to symplectic fillings in \cite{Nie06}, one boundary is a degenerated holomorphic disk on the core of the bLob, and the other boundary will give rise to holomorphic planes killing the contact homology. In other words, the proof is phrasing the celebrated proof of the Weinstein conjecture for overtwisted contact manifolds in dimension $3$ by Hofer \cite{zbMATH00549637} in the context of contact homology. However, bLobs are codimension $n$ objects in a $2n+1$ dimensional contact manifold, where the torsion cobordism perspective is essentially codimension $0$. The latter is more or less necessary in order to have a higher $\APT$. Therefore, the connection between bLobs and torsion cobordisms is probably not very strong. PS-overtwisted contact manifolds are the "but one" part in the abstract.
    \item Contact manifolds arose as the concave boundary of \emph{exact} torsion cobordisms of type I or II are algebraically overtwisted by the main theorem in \cite{+1} 
    \item Contact $3$ manifolds derived from applying contact $1/k$ surgery to the right–handed trefoil knot with the maximal Thurston–Bennequin number in $(S^3,\xi_{std})$ for $k\in \N_+$ are algebraically overtwisted by \cite{zbMATH07699409}. More generally, based on the algorithm in \cite{zbMATH07699409}, we showed in \cite{zbMATH07983165} that the same holds for all positive torus knots (besides the unknot) with the maximum Thurston-Bennequin number. Those results are generalized vastly with a completely different method in \cite{AOT} to all dimensions. 
    \item Open books with monodromy certain negative fibered Dehn twists were shown to have vanishing contact homology \cite{zbMATH06359114,zbMATH06562001}.
\end{enumerate}
\end{example}

\begin{example}[Higher dimensional Giroux torsion]
It was shown in \cite[Theorem B]{MNW} that if a contact manifold contains a connected codimension $0$ submanifold with nonempty boundary obtained by gluing together two Giroux domains, then it is not strongly fillable. As the domain has a non-trivial boundary, at least one of the Giroux domains is $S^1\times V$ with $V$ a Liouville domain with disconnected boundary. By blowing up the other Giroux domain, we obtain a strong cobordism to a contact manifold containing a bLob \cite[Proposition 4.9]{MNW}. Then, by the functoriality of $\APT$ under strong cobordisms, we know that those contact manifolds have finite $\APT$. \cite[Theorem B]{MNW} can be applied to all non-strongly/weakly fillable examples in \cite{MNW}, e.g.\ \cite[Theorems A, E, F and G]{MNW},  except those that contain bLobs only.

We can also understand \cite[Theorem B]{MNW} in terms of a generalization of torsion cobordism of type III. Assume there are two Giroux domains $S^1\times V_1,S^1\times V_2$ glued along the common boundary $M\subset \partial V_1,\partial V_2$, such that $\partial V_1\backslash M \ne \emptyset$. If $M=\partial V_2$, by capping off $S^1\times V_1$, we get a strong cobordism to $Y\sqcup \partial(V_2\times \D)$ for a non-empty contact manifold $Y$, that is, we get a torsion cobordism of type III. Similar to the proof of \Cref{thm:finite_APT_k} (the ruling curve intersects once with the non-exact co-core $V_1\times \{0\}$ in the capping cobordism in \Cref{prop:cap}), we get that the contact manifold we start with has $\APT\le 1$. If $V_2\backslash M\ne \emptyset$, then we need to cap off both $S^1\times V_1,S^1\times V_2$ to get a strong cobordism $W$ to $Y\sqcup M\times \CP^1$.  $M\times \CP^1$ has ruling curves without positive punctures from the $\CP^1$-component, whose intersection with each of the two non-exact co-cores of the capping cobordism is one. Similar to the proof of \Cref{thm:torsion_III,thm:finite_APT_k}, by considering the moduli space of curves passing through a path in $W$ connecting a point in $M\times \CP^1$ and a point in $Y$, we can deduce that the algebraic planar torsion of the concave boundary is at most $1$. In some sense, this is phrasing the proof of \cite[Theorem 6.1]{MNW} in terms of rational SFT. Moreover, for planar torsions in dimension $3$, we can cap off the planar torsion domain as in \cite{Wen13} to obtain a strong cobordism to $S^1\times \CP^1\sqcup Y$ behaving like a $k$-torsion cobordism. Therefore, we can use the functorial perspective in this paper to give an alternative proof to \cite[Theorem 6]{LW}.

Using Liouville domains in the form of  $M\times [-1,1]$, a higher-dimensional Giroux torsion was introduced in \cite[before Corollary 8.2]{MNW}. Contact manifolds with such a Giroux torsion are special cases of \cite[Theorem B]{MNW} recalled above. For contact manifolds with (higher-dimensional) Giroux torsions, it was already established in \cite[Theorem 1.6]{moreno2019algebraic} that they have algebraic torsion at most $1$. The proof in \cite{moreno2019algebraic} was based on a direct computation of holomorphic curves and adiabatic limit methods relating ``Morse-Bott" and Morse settings of contact forms. 
\end{example}

\begin{example}[Moreno's example in \cite{moreno2019algebraic}]
    Let $M$ be a manifold with a Liouville pair \cite[Definition 1]{MNW}, in particular, $M\times I$ can be equipped with a Liouville structure $\lambda$, whose boundary is $(M,\alpha_+),(M,\alpha_-)$ with $\alpha_{\pm}$ are two contact forms with opposite orientation. In \cite{moreno2019algebraic}, the following example was considered (after smoothing the corner):
    $$Y:=\left((M\times S_k) \cup_{i=1}^k (M\times I \times S^1) \cup (M\times \Sigma_{g,k}), \quad (\alpha_++\lambda_{S_k})\cup_{i=1}^k(\lambda+\rd\theta) \cup(\alpha_-+\lambda_{\Sigma_{g,k}})
    \right)$$
    where $\lambda_{S_k},\lambda_{\Sigma_{g,k}}$ are Liouville forms on genus $0$, respectively $g$ curve with $k$ disks removed, such that the restriction of the boundary circle is $\rd \theta$. \cite[Theorem 1.3]{moreno2019algebraic} asserted such manifolds have algebraic torsion at most $k-1$ if $g\ge k$. From the respective in this paper, we can cap off the $k$ $M\times I \times S^1$ regions to obtain a strong cobordism $W$ from $Y$ to stable Hamiltonian manifold $(M\times \CP^1,\alpha_+,\rd\alpha_++\omega_{\CP^1})\sqcup (M\times \Sigma_g,\alpha_-,\rd\alpha_-+\omega_{\Sigma_g})$. $M\times \CP^1$ is uniruled by closed curves in the $\CP^1$-direction, while $M\times \Sigma_g$ has no non-trivial closed rational curves if $g>0$. Moreover, in the completion $\widehat{M}\times \CP^1$ with the product almost complex structure, there are no curves (possibly with genus) without positive punctures but with negative punctures by projection to $\widehat{M}$ and the maximum principle. Therefore, similar to the proofs of \Cref{thm:torsion_III,thm:finite_APT_k} by considering degeneration of holomorphic curves in $W$ without positive puncture and passing through a path connecting a point in $M\times \CP^1$ and a point in $M\times \Sigma_k$, we see that the algebraic planar torsion of $Y$ is at most $k-1$ if $g>0$. Moreover, we can argue similarly that the algebraic torsion is also at most $k-1$ if $g>0$, instead of $g\ge k$.
\end{example}

\begin{example}[Bourgeois contact manifolds]
    In \cite[Theorem C]{BGM}, if $\Sigma \hookrightarrow \OB(\Sigma,\phi)$ does not induce an injection on rational homology, the Bourgeois contact manifold $\BO(\Sigma,\phi)$ is not strongly fillable. More general results hold for some $S^1$-invariant contact structures by \cite[Theorem B]{BGM}. Those manifolds arose as the concave boundary of $1$-torsion cobordism by \Cref{ex:S1}, hence they have algebraic planar torsion at most $1$ by \Cref{thm:finite_APT_k}. 
\end{example}

\begin{example}[Exotic contact structures in \cite{bowden2022tight}]
    By Theorems \ref{thm:sphere} and \ref{thm:all}, exotic contact spheres as well as the connected sum with a general contact manifold considered in \cite{bowden2022tight} have finite fully twisted algebraic planar torsions.
\end{example}

\subsection{One more application}\label{ss:app}
Let $\Sigma$ be a Weinstein domain and $\phi\in \pi_0(\symp_c(\Sigma))$. By \cite[Lemma 2.2]{zbMATH07738054}, the inclusion of the page $\Sigma \hookrightarrow \OB(\Sigma,\phi)$ has a non-trivial kernel on the induced rational homology if and only if $\mathrm{var}(\phi)\ne 0$, which occurs in degree $\dim_{\C}\Sigma$.

\begin{theorem}\label{thm:OB}
Let $\Sigma$ be a Weinstein domain and $\phi\in \pi_0(\symp_c(\Sigma))$ such that $\mathrm{var}(\phi)\ne 0$. Then
\begin{enumerate}
    \item At least one of $\OB(\Sigma,\phi)$ and $\OB(\Sigma,\phi^{-1})$ is algebraically overtwisted;
    \item If $\OB(\Sigma,\phi)$ is not algebraically overtwisted, then $\phi^k\ne \Id \in \pi_0(\symp_c(\Sigma))$ for $k>0$.
\end{enumerate}
\end{theorem}
\begin{proof}
By \cite[Proposition 8.2]{Russell}, there is a Weinstein cobordism $W$ from $\OB(\Sigma,\phi)\sqcup \OB(\Sigma,\phi^{-1})$ to $\OB(\Sigma,\Id)$. By assumption, the inclusion of the page $\Sigma$ into $\OB(\Sigma,\Id)$ and then into $W$ has a non-trivial kernel on the induced rational homology. Then, by the Liouville cobordism case of \Cref{thm:torsion} in \S \ref{ss:torsion}, we have $\CH(\OB(\Sigma,\phi)\sqcup \OB(\Sigma,\phi^{-1}))=0$; that is, either $\CH(\OB(\Sigma,\phi))=0$ or $\CH(\OB(\Sigma,\phi^{-1}))=0$. If $\phi^k=\Id$ for $k>0$, we have a Weinstein cobordism $V$ from $\sqcup_{i=1}^k\OB(\Sigma,\phi)$ to $\OB(\Sigma,\Id)$ such that the inclusion of the page $\Sigma$ into $\OB(\Sigma,\Id)$ and then into $V$ has a non-trivial kernel on the induced rational homology because $\mathrm{var}(\phi)\ne 0$. Then we can deduce from the Liouville cobordism case of \Cref{thm:torsion} that $\CH(\OB(\Sigma,\phi))=0$, contradicting the assumption.
\end{proof}

\begin{example}
    The Dehn-Seidel twist $\phi_{DS}$ on $D^*S^n$ has non-trivial $\mathrm{var}(\phi_{DS})$. Since $\OB(D^*S^n,\phi_{DS})$ is $(S^{2n+1},\xi_{\std})$, which is not algebraically overtwisted because it is fillable, it follows that $\OB(D^*S^n,\phi_{DS}^{-1})$, i.e., the overtwisted contact sphere, has vanishing contact homology. Moreover, $\phi_{DS}^k\ne \Id$ in $\pi_0(\symp_c(D^*S^n))$ for any $k>0$. Both are well-known, classical results regarding the Dehn-Seidel twist. More generally, if $L$ is a Lagrangian sphere in a Weinstein domain $\Sigma$ such that $[L]$ is non-trivial in rational homology, then $\mathrm{var}(\phi_{L})\ne 0$ for the Dehn-Seidel twist $\phi_L$ around $L$. Then we have $\CH(\Sigma,\phi_L^{-1})=0$ and $\phi_L^k \ne \Id \in \pi_0\symp_c(\Sigma)$ for any $k>0$. The former result was established in \cite{+1}, and the latter result was established in \cite{BGZ} without the homology assumption\footnote{The homology assumption may be tautological in view of the regular Lagrangian conjecture in \cite[Problem 2.5]{zbMATH07195660}.}.  
\end{example}

Another source of elements in symplectic mapping class groups is fibered Dehn twists on Weinstein domains whose boundaries are circle bundles of pre-quantization bundles. Contact open books from such monodromies and their contact homology were studied extensively by Chiang, Ding, and van Koert \cite{zbMATH06359114,zbMATH06562001}. The following corollary provides more examples in the same flavor.

\begin{corollary}\label{cor:OB}
Let $(X,\omega)$ be a smooth projective variety, and let $D$ be an ample divisor such that the map $H_{*}(X,\Q)\to H_{*-2}(D;\Q)$ given by $A \mapsto A\cap [D]$ is not surjective. Then the ideal contact boundary of $X\backslash D$ is the circle bundle of the normal bundle of $D$, and $\phi$ is the fibered Dehn twist on $X\backslash D$ such that
\begin{enumerate}
	\item $\CH(\OB(X\backslash D,\phi^{-1}))=0$;
	\item $\phi^k\ne \Id$ in $\pi_0 \symp_c(X\backslash D)$ for $k\ne 0$.
\end{enumerate}
\end{corollary}
\begin{proof}
    By \cite[Theorem 1.1]{zbMATH06562001}, the circle bundle $Y$ of the prequantization bundle $\cO(-D)$ over $X$, as a contact manifold, admits an open book $\OB(X\backslash D, \phi)$. Note that this contact manifold $Y$ has a strong filling by $\cO(-D)$; therefore, $\CH(Y)\ne 0$. To apply \Cref{thm:OB}, we need that the map $H_*(X\backslash D) \to H_*(Y)$ has a non-trivial kernel. By the hard Lefschetz theorem and the Gysin exact sequence, the map $H^*(X)\to H^*(Y)$ is surjective for $*\le \dim_{\C}X$, because $H^i(Y)$ is spanned by the pullback of primitive $i$-forms on $X$ for $i\le \dim_{\C}X$. Dually, the map $H_*(Y)\to H_*(X)$ is injective for $*\le \dim_{\C}X$. Since $H_*(X\backslash D)$ is supported in degree at most $\dim_{\C}X$ and the composition $H_*(X\backslash D) \to H_*(Y)\to H_*(X)$ coincides with $\iota_*$ induced by the inclusion $X\backslash D \hookrightarrow X$, it follows that $H_*(X\backslash D) \to H_*(Y)$ is non-injective if and only if $\iota_*:H_*(X\backslash D) \to H_*(X)$ is non-injective. Consider the long exact sequence,
    $$\ldots\to H_{k+1}(X)\to H_{k+1}(X,X\backslash D) \stackrel{\partial_*}{\to} H_k(X\backslash D) \stackrel{\iota_*}{\to} H_k(X) \to \ldots.$$
    Then $\iota_*$ is injective if and only if $\partial_*=0$, which is equivalent to the surjectivity of $H_{k+1}(X)\to H_{k+1}(X,X\backslash D)$. Via the Thom isomorphism, $H_{k+1}(X,X\backslash D)\simeq H_{k-1}(D)$, and the map $H_{k+1}(X)\to H_{k+1}(X,X\backslash D)$ corresponds to the map $H_{*}(X,\Q)\to H_{*-2}(D;\Q)$ given by $A \mapsto A\cap [D]$. Therefore, the claim follows.
\end{proof}

\begin{example}
    Let $X=\mathbb{P}^n$. Then the condition in \Cref{cor:OB} holds as long as $D$ has degree greater than $1$.
\end{example}	
\bibliographystyle{alpha} 
\bibliography{ref}
\Addresses

\end{document}